\documentclass[12pt,a4paper]{amsart}
\usepackage{amsmath,amsfonts}
\usepackage{graphicx}
\usepackage{epsfig}
\usepackage{amssymb}
\newtheorem{theorem}{Theorem}[section]

\newtheorem{cor}[theorem]{Corollary}

\newtheorem{property}[theorem]{Property}
\newtheorem{proposition}[theorem]{Proposition}
\newtheorem{lemma}[theorem]{Lemma}
\newtheorem{definition}[theorem]{Definition}
\newtheorem{fact}[theorem]{Fact}

\newtheorem*{decotheorem}{Decoration Theorem}

\newcommand{\disk}{\overline{{\mathbb D}}}
\newcommand{\opdisk}{{\mathbb D}}
\newcommand{\uncirc}{{\mathbb S}^1}
\newcommand{\sph}{{\mathbb S}^2}
\newcommand{\mcup}{\cup_{ \mathbb{S}^1}}

\newcommand{\Z}{\mathbb{Z}}
\newcommand{\C}{\mathbb{C}}

\newcommand{\LL}{\mathcal L}
\newcommand{\MM}{\mathcal M}

\newcommand{\WW}{\mathcal W}
\newcommand{\UU}{\mathcal U}
\newcommand{\FF}{\mathcal F}
\newcommand{\TT}{\mathcal T}
\newcommand{\DD}{\mathcal D}

\newcommand{\cMM}{\widehat{\mathcal M}}

\newcommand{\RR}{\mathcal R}
\newcommand{\VV}{\mathcal V}

\newcommand{\EE}{\mathcal E}

\newcommand{\XX}{\mathcal X}
\newcommand{\YY}{\mathcal Y}
\newcommand{\ZZ}{\mathcal Z}
\newcommand{\PP}{\mathcal P}
\newcommand{\HH}{\mathcal H}

\newcommand{\BB}{\mathcal B}

\newcommand{\ff}{\mathbf{f}}
\newcommand{\basil}{\mathfrak{B}}

\newcommand{\ovl}{\overline}
\newcommand{\intr}{\operatorname{int}}

\newcommand{\reminder}[1]{\textsf{#1}}

\begin{document}

\title[Matings with laminations]{Matings with laminations}
\author{Dzmitry Dudko}
\maketitle

\begin{abstract}
We give a topological description of the space of
quadratic rational maps with superattractive
two-cycles: its ``non-escape locus'' $\MM_2$ (the
analog of the Mandelbrot set $\MM$) is locally connected, it is the
continuous image
of $\MM$ under a canonical map, and it can be described as
$\MM$ (minus the $1/2$-limb), mated with the lamination of
the basilica. The latter statement is a refined version of a
conjecture of Ben Wittner, which in its original version requires
local connectivity of $\MM$ to even be stated.
Our methods of mating with a lamination also
apply to dynamical matings of certain non-locally connected Julia
sets.
\end{abstract}

\section{Introduction}
A convenient way to describe some rational maps is by means of
matings. If a rational map is a mating of two polynomials, then the
dynamics of this rational map is a sum of polynomial dynamics. In
addition, dynamical matings usually, if not always, have
consequences to the parameter spaces. A classical example is the space of quadratic rational functions
with superattractive $n$--cycles; and, in particular, the case $n=2$ which will be the main objet of our study.

In this paper we introduce a surgery ``matings with laminations''
that generalizes a classical mating construction for certain cases when
one of the sets is not locally connected (some infinitely
renormalizable quadratic polynomials) or local connectivity is an
open question (the Mandelbrot set). Our results could be easily
generalized, for example, to certain components of the space of quadratic rational functions
with superattractive $n$--cycles.

\subsection{Wittner's conjecture}
Consider the space $\VV_1$ of quadratic
polynomials parametrized by $c\in\C$ via
\[\{f_c(z)=z^2+c:\ c\in\C\}.\] Many important dynamical properties of $\VV_1$ are encoded in
the \textit{Mandelbrot set}
\begin{equation}
\label{eq:Mand} \MM=\{c\in\C:\text{ the polynomial }f_c
:z\rightarrow z^2 + c\text{ has connected Julia set}\}.
\end{equation}
For instance, the boundary of the Mandelbrot set is the bifurcation
locus: $z^2+c$ is structurally stable if and only if $c\not\in\partial \MM$.

 McMullen's motto ``the Mandelbrot set is
universal'' \cite{McUniversal} states that there are infinitely
many copies of the Mandelbrot set in every non-trivial one dimensional parameter space
of rational functions with one ``active'' (and non--degenerate) critical point. The Douady-Hubbard straightening theorem \cite{DH}
gives a canonical dynamically meaningful homeomorphism between the
copies and $\MM$. In particular, there are infinitely many copies of
$\MM$ within the Mandelbrot set itself, even within every open set intersecting $\partial \MM$.

\begin{figure}%[hb]
\includegraphics[width=6cm]{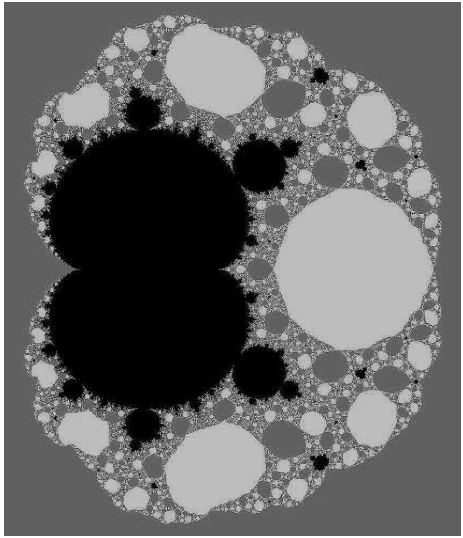}
  \caption{The parameter slice $\VV_2$ of
quadratic rational functions with superattractive two cycles parametrized as $\{\frac{a}{z^2+2}:\ a\in\C\setminus\{0\}\}$.
The set $\MM_2$ (the black set on the figure) is the set of
parameters such that the second critical
  point is not in the attracting basin of the two-cycle. Figure from
  \cite{T2}.
  }
  \label{figure:V2Slice}
\end{figure}

The Mandelbrot set may occur in more ways than predicted by the
straightening theorem. Following the notation of Mary Rees, denote
by $\VV_2$ the space (slice) of quadratic rational maps that have
superattractive two-cycles and by $\MM_2\subset\VV_2$ the set of
non-escaping parameters (see Figure \ref{figure:V2Slice} and its
caption). As in $\VV_1$, a map $g\in \VV_2$ is structurally stable if and only if
$g\not\in
\partial \MM_2$. Further, denote by $K_B$ the
filled-in Julia set of the Basilica polynomial $f_B(z)=z^2-1$. Maps in
$\MM_2$ should be thought as matings of
polynomials in $\MM$ with $f_B$; this has been already worked out in
many cases \cite{Re1}, \cite{Tan}, \cite{Lu}, \cite{AY}, \cite{T} but, in general, is not correct
as there are quadratic polynomials with non locally connected Julia sets. In 1988 Wittner conjectured
\cite{Wi} a parameter counterpart to the dynamical matings:
\begin{equation}
\label{eq:MatingMK} \VV_2\ \  \text{ is the mating }\ \ \MM\mcup
K_B.
\end{equation} More precisely, remove the $1/2$-limb
(the Basilica limb) from the Mandelbrot set, remove the $1/2$-limb
from $K_B$, and glue these sets along the boundaries parametrized by parameter and dynamic rays;
this construction requires $\partial \MM$ (and $\partial K_B$) to be locally
connected. After the
mating, hyperbolic components of the Mandelbrot set become
hyperbolic components of mating type, Fatou components of $K_B$
become hyperbolic components of capture type, and the boundary of
the Mandelbrot set (or, equivalently, $\partial K_B$) becomes the
bifurcation locus $\partial \MM_2$ of $\VV_2$.

\subsection{Main results}
We give here a refinement of Wittner's construction that is
unconditional to the local connectivity of $\partial \MM$. Briefly, we write
$K_B=\disk/L_B$, see the next paragraph, so $ \MM\mcup
K_B=\MM\mcup  (\disk/L_B)$. Then we change the order of operations:
\[\MM\mcup L_B:=(\MM\mcup \disk)/L_B=\VV_1/L_B,\]
where $\MM\mcup \disk=\VV_1$ is a trivial mating (just extending $\MM$ into $\hat\C$).
We call $\MM\mcup L_B$ the
\textit{mating with the lamination}. If $\partial \MM$ is locally connected, then
$\MM\mcup L_B$ is the same as $\MM\mcup K_B$.

\begin{figure}%[hb]
\includegraphics[width=5cm]{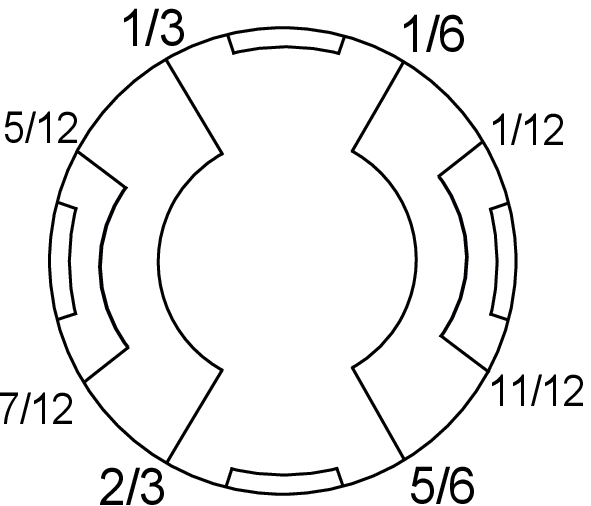}
\includegraphics[width=6cm]{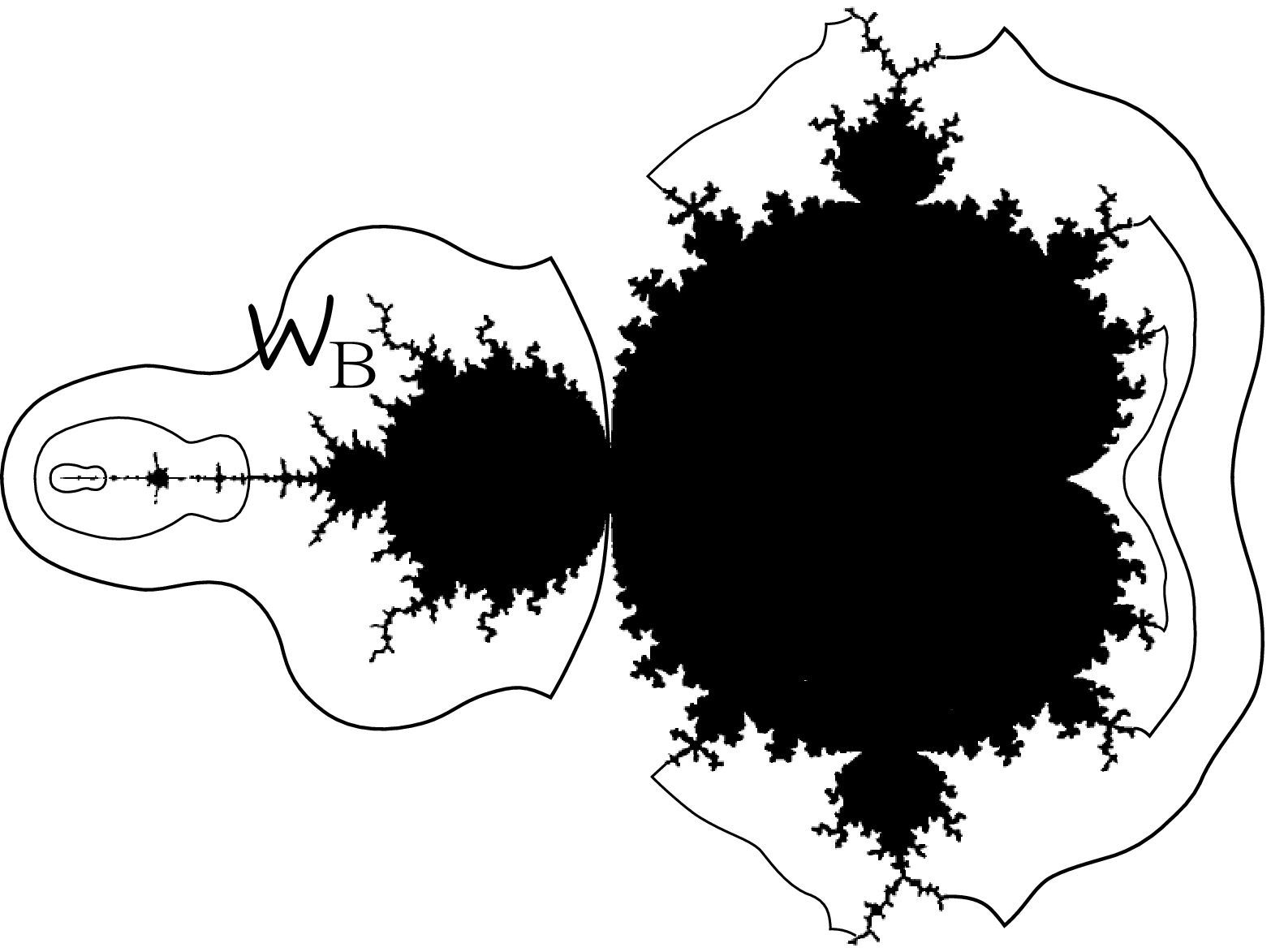}
  \caption{Left: the lamination $L_B$ of the Basilica. The $\langle 1/3,2/3\rangle$
  and $\langle5/6,1/6\rangle$ leaves are symmetric; any other leaf is an iterated pre-image
  of the $\langle5/6,1/6\rangle$ leaf  under $z^{2}$. Right: the embedding $\LL_B$ of $L_B$
  into $\VV_1$; the \textit{Basilica domain} $\WW_B$ is a closed disc surrounded by the $\langle1/3,2/3\rangle$--leaf
 and containing the Basilica limb.
  }
  \label{figure:BasLamin}
\end{figure}

Let us explain the construction in more detail. Consider the
\textit{lamination} $L_B$ of the Basilica (Figure \ref{figure:BasLamin});
$L_B$ is a family of disjoint curves, called \textit{leaves}, in the closed unit disc $\disk$ connecting points on $\partial \disk$
such that $L_B$ is closed as a subset of $\disk$ and the equivalence relation induced by $L_B$
\[x\sim_{L_B}y \ \text{ if and only if }\ x,y \text{ are in a leaf of }L_B\text{ or }x=y\]  models
$K_B$; i.e. $K_B=\disk/L_B.$

Recall that the parameter B\"{o}ttcher function $\BB$ maps $\C\setminus \MM$
conformally to $\C\setminus \disk$. Denote by
\[\LL_B=\{\BB^{-1}\circ (1/z) (l):\ l\in L_B\}\]
the embedding of $L_B$ into $\VV_1\setminus \MM$; let us remark
that $\BB^{-1}\circ (1/z) (l)$ is well defined and is a closed curve as the ends of $\BB^{-1}\circ (1/z) (l)$ land at Misiurewicz or
parabolic parameters
at which $\MM$ has triviality of the fiber.

 We say that two non-equal points are equivalent under $\LL_B$ if these points are either
 in the closure of a leaf in $\LL_B$ or in the Basilica domain
 $\WW_B$ (see Figure \ref{figure:BasLamin} and its caption) which is a closed neighborhood of the Basilica limb.

Define $\VV'_2$ to be the quotient $\VV_1/\LL_B$, define
 $m:\VV_1\to \VV_1/\LL_B=\VV'_2$ to be the quotient map, and define $\MM'_2:=m(\MM)$. A priori,
 the construction of $\VV'_2$ depends on the choice of leaves in $L_B$.

\begin{theorem}[The mating $\MM\mcup
L_B$ is well defined]
\label{th:main1} The restriction $m|_\MM:\MM\to m(\MM)=\MM'_2$ is
canonical (independent of $L_B$). The space $\MM'_2$ is locally
connected.
All spaces $\VV'_2=\VV'_2(L_B)$ for different choices of leaves in $L_B$ are homeomorphic by
homeomorphisms preserving $\MM'_2$.
\end{theorem}

 The next theorem says that  $\MM\mcup
L_B$ is topologically and combinatorially isomorphic to $\VV_2$.

\begin{theorem}[Refined Wittner's Conjecture]
\label{th:main2} The space $\MM\mcup L_B$ is homeomorphic to $\VV_2$
by a homeomorphism that preserves the combinatorics of parapuzzle pieces and coincides with
straightening maps on all small copies of the Mandelbrot set.
This homeomorphism is canonical (unique) over $\MM_2$.
\end{theorem}

\begin{cor}
The set $\MM_2$ is locally connected and is a continuous image of $\MM$.
\end{cor}

 By Theorem \ref{th:main2}
we have $\MM'_2\thickapprox \MM_2$ (canonically) and we view the map $m|_\MM$ in Theorem \ref{th:main1} as a map form $\MM$ to $\MM_2$.
In fact, $m|_\MM$ has a simple description: it collapses the basilica limb to a point and pinches (identifies)
Misiurewicz parameters with external angles $\phi_1,\phi_2$ if and only if $\phi_1,\phi_2$ are identified by $L_B$.

Similarly to $\MM\mcup L_B$ we define the dynamical mating
$f_c\mcup L_B$, see details in Section \ref{sec:Thm3}.
The last theorem is the dynamical counterpart to Theorem
\ref{th:main2}.
\begin{theorem}[$f_c\mcup L_B \thickapprox g_{m(c)}$]
\label{th:main3} If $c\in\MM\setminus \WW_B$ is not on the boundary of a hyperbolic component or $c$ is at least $4$ times renormalizable, then
$f_c\mcup L_B$ is well defined and is topologically conjugate to
$g_{m(c)}$.

If the Julia set of $f_c$ is locally connected, then
\[f_c\mcup L_B  = f_c\mcup f_B. \]
\end{theorem}
This gives a topological description of $g_{m(c)}$ modulo $f_c$.
For instance, $f_c:J_c\to J_c$ is semi-conjugate to
$g_{m(c)}:J_{m(c)}\to J_{m(c)}$; the semi-conjugacy is unique if it preserves the combinatorics of puzzle pieces
and coincides with straightening maps on small filled in Julia sets.

It has been a long standing problem to describe maps in $\VV_2$ in terms of surgeries
over quadratic polynomials. Hyperbolic maps in $\MM_2$ (resp.\ in $\VV_2\setminus \MM_2$) are matings
(resp.\ captures) with the Basilica polynomial \cite{Re1}, \cite{Tan}, \cite{Lu}; this result is an application of Thurston's topological characterization
of postcritically finite rational functions \cite{DH3}. Maps on the external boundary of $\MM_2$
are described in \cite{T} in terms of anti-matings. In \cite{AY} existence and uniqueness of the matings of the Basilica with
non-renormalizable Yoccoz's polynomials in $\MM\setminus \WW_B$ are shown; this result
could be easily generalized for the finite-renormalizable case. Theorem \ref{th:main3} contributes
the infinitely renormalizable case. Because $J_c$ can be non-locally connected the mating $f_c\mcup L_B$ in Theorem \ref{th:main3}
can not be always replaced by
$f_c\mcup f_B$.
\subsection{Parameter slices $\VV_n$}
Similarly to $\VV_1$ and $\VV_2$, denote
by $\VV_n$ the space (slice) of quadratic rational maps that have
superattractive $n$-cycles and by $\MM_n\subset\VV_n$ the set of
non-escaping parameters; again, $\partial \MM_n$ is the bifurcation locus.

For $n\ge 3$ the slice $\VV_n$ consists of finitely many components; some of these components have similar descriptions
as $\VV_2$. For example, $\VV_3$ (Figure \ref{figure:V3Slice})
consists of the Rabbit part $\VV_R$, the Corabbit part $\VV_C$, the Airplane part $\VV_A$, and two main hyperbolic capture components.
The proofs of Theorem \ref{th:main1}, \ref{th:main2}, \ref{th:main3} work for $\VV_R$ and $\VV_C$ with slight modifications so
$\VV_R$ and $\VV_C$ are the matings of the Mandelbrot set with the laminations of the Rabbit and Corabbit respectively. In general, for an
$n$--rabbit $f_{R'}$ (i.e. $f_{R'}$ is the center of a hyperbolic component attached to the main hyperbolic component of $\MM$), define $\VV_{R'}$
to be the part of $\VV_n$ consisting of ``matings and captures with $f_{R'}$''. Then $\VV_{R'}$ is $\MM\mcup L_{R'}$.

\begin{figure}%[hb]
\includegraphics[width=8cm]{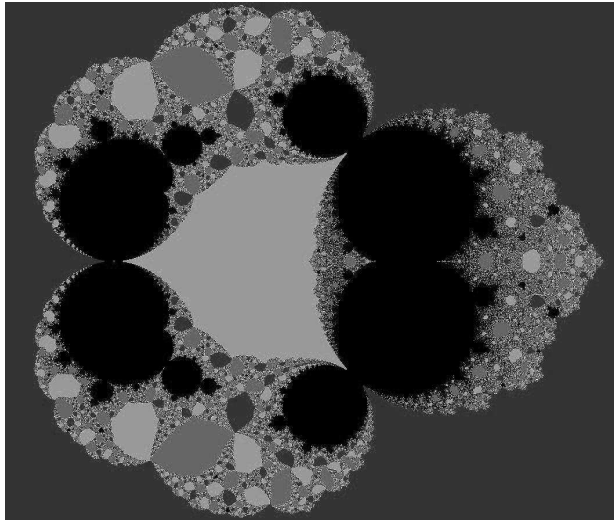}
  \caption{The parameter slice $\VV_3$ consists of $\VV_R$, $\VV_C$ (symmetric to
$\VV_R$), and two main hyperbolic capture component. Figure from
  \cite{T2}.
  }
  \label{figure:V3Slice}
\end{figure}

\subsection{Outline of the paper and conventions}
\textbf{In Section \ref{sec:BasDef}} (Background: $\VV_1$) we review properties of
$\VV_1$.

A puzzle piece (Subsection \ref{sec1}) is a closed topological disc in dynamical plane bounded by (finitely many) periodic
and pre-periodic dynamic rays and truncated by an equipotential such
that the forward orbit of the boundary of a puzzle piece does not
intersect its interior.

Using B\"{o}ttcher functions we combinatorially identify puzzle pieces in
different dynamical planes: puzzle pieces
$Y(c_1),Y(c_2)$ in the dynamical planes of $f_{c_1}$ and $f_{c_2}$
are equivalent (simply, $Y(c_1)=Y(c_2)=Y$) if the composition
$B_{c_2}^{-1}\circ B_{c_1}$ of B\"{o}ttcher functions induces a
homeomorphism between $\partial Y(c_1)\setminus J_{c_1}$ and $\partial Y(c_2)\setminus J_{c_2}$
that extends to a homeomorphism between $\partial Y(c_1)$ and $\partial Y(c_2)$. Equivalently, there
is a homeomorphism between $\partial Y(c_1)$ and $\partial Y(c_2)$
preserving angles and heights of rays and equipotentials.

A closed disc $\YY$ in the parameter plane bounded by parameter rays and truncated by
a parameter equipotential is a parapuzzle piece if $\YY$ is the parameter counterpart to
a dynamic puzzle piece, denoted by $Y$: the composition
$\BB^{-1}\circ B_{c}$ of B\"{o}ttcher functions induces a
homeomorphism between $\partial Y(c)\setminus J_{c}$ and $\partial \YY\setminus \MM$ as above.

  \textit{We will use the following convention}: objects in the parameter
plane will be denoted by calligraphic capital letters (such as
$\ZZ^n_i$, $\RR^\phi$) while those in dynamical planes will
be denoted by Roman capitals (such as $Z^n_i$, $R^\phi$).

In Subsection \ref{subset:QuadrLike} we recall the construction of
straightening maps; small copies of the Mandelbrot set are
discussed Subsection \ref{subsec:SmallCopies}. In Subsection
\ref{subsec:DecorTiling} we review the decoration theorem. In Subsection \ref{subs:LocalConnSubs}
we discuss compact connected locally connected subsets of the $2$--sphere. In
Subsection \ref{subs:lamination} we fix some conventions regarding
laminations.

 \textbf{In Section \ref{sec:ParSp}} (Mating construction) we will prove Proposition
\ref{prop:main} which is a refined version of Theorem \ref{th:main1}.
In Subsection \ref{subsec:CombinOfLeaves} we discuss the relation
between small copies of the Mandelbrot set in $\MM$ and leaves in
$\LL_B$. We will show that points on the boundary of a
component $U$ of $\VV_1\setminus(\MM\cup \LL_B\cup \WW_B)$ are not infinitely renormalizable;
this will imply that $\partial U$ is
locally connected (the first part of Proposition \ref{prop:main}).

Further, using the decoration theorem (Subsection
\ref{subsect:DecorTilingWithLb}) we will show that \textit{big} leaves
(with diameters at least a fixed $\varepsilon>0$) of $\LL_B$ and big
components of $\VV_1\setminus (\MM\cup\LL_B\cup\WW_B)$ can not
accumulate at secondary copies of the Mandelbrot set (Propositions
\ref{prop:bound}, \ref{prop:bound2}, and
\ref{prop:DecTilWithBasil}). Combined with Yoccoz's results this, on
one hand, shows that $\LL_B$ is a closed equivalence relation in
$\VV_1$ (and canonical on $\MM$), and, on the other hand, leads to
local connectivity of $\MM\cup\LL'_B\cup \WW_B$.

 Finally, we
establish local connectivity of $\VV'_2$ and show that $\VV'_2$ does
depend (up to homeomorphism preserving combinatorics) on the choice
 of leaves in $L_B$.

\textbf{In Section \ref{sec:V2}} we review properties of $\VV_2$. We
parameterize $\VV_2$ as \[\{g_a(z)=a/(z^2+2z):a\in\C\setminus\{0\}\}\cup\{g_\infty=1/z^2\},\]
 so that $-1, \infty$ are the
critical points of $g_a$ and $g_a(0)=\infty,$ $g_a(\infty)=0$.

In Subsection \ref{subsec:RayInBubble} we define \textit{rays in bubbles} for $g_a$; using these rays we define
bubble puzzle pieces. There is a natural identification of rays in bubbles
for different dynamical planes; this allows us to combinatorially identify bubble puzzle pieces in different planes.
Bubble parapuzzle pieces are ``the parameter counterparts'' of dynamic bubble pieces.

A special map in $\VV_2$ is $g_1= 1/(z^2+2)$; this map also belongs to $\VV_1$: it is conjugate by $-1/(z+1)$ to $f_B$.
Using the dynamical plane of  $g_1\sim f_B$ we define ``$\VV_2$--twins'' (Definition \ref{defn:V2Dupl} and Subsection \ref{subsect:BubblPuzzle})
for many structures in $\VV_1$ such as rays, puzzle and parapuzzle pieces, small copies of the Mandelbrot set.

\textit{We will often use an upper index $B$} (for example:
$Y^B_\HH,$ $\YY^B_{n\HH}$, and $\MM^B_\HH$) to indicate the
$\VV_2$-twins of objects in $\VV_1$ (respectively: $Y^p_\HH,$
$\YY^{np}_{n\HH}$, and $\MM_\HH$).

\textbf{In Section \ref{sec:ThmParBubblRay}} (Landing of rays in $\VV_2$) we will prove theorem
Theorem \ref{thm:ParBubblRay}. This will imply that for a parapuzzle
piece $\YY$ in $\VV_1\setminus\WW_B$ there is its
\textit{$\VV_2$--twin} $\YY^B$ and every bubble parapuzzle piece is
a $\VV_2$--twin of a parapuzzle piece in $\VV_1\setminus\WW_B$.
Moreover, a dynamic puzzle piece $X$ exists in $\intr(\YY)$ if and only if the $\VV_2$--twin $X^B$ exists in
$\YY^B$.  In other words, the parapuzzle partition in $\VV_1\setminus \WW_B$ is combinatorially the
same as the bubble parapuzzle partition in $\VV_2$; this reflects the
fact that hyperbolic maps in $\MM_2$ are the mating of hyperbolic
polynomials in $\MM\setminus\WW_B$ with the Basilica polynomial..

\textbf{In Section \ref{sec:ProofsOfMainThms}} we prove Theorem
\ref{th:main2}.  By a \textit{combinatorial disc} in $\VV_2$ we mean a closed
topological disc $\DD$ such that $\partial \DD$ is is in a finite union
of parameter bubble rays with rational angles and simple arcs in
   closures of hyperbolic components. A sequence $a_n$\textit{ tends to $a_\infty$ combinatorially}  if for
every combinatorial disc $\DD$ such that $a_\infty\in \intr(\DD)$ all but finitely many
$a_n$ are in $\DD$. Similarly, combinatorial convergence in dynamical planes is defined.

As in $\VV_1$, a
parameter $a\in\VV_2$ is \textit{non-infinitely renormalizable} if
$a$ belongs to at most finitely many copies of $\MM$.
Theorem \ref{YoccResInV2} (Yoccoz's results in $\VV_2$) shows that for a non-infinitely renormalizable parameter $a$ a
sequence $a_n$ tends to $a$ in the combinatorial topology if and
only if $a_n$ tends to $a$ in the usual topology. In Subsection
\ref{subsec:DecorTilV2} we establish the
decoration theorem for $\VV_2$ (Proposition \ref{prop:DecorTilV2}) using the straightening theorem.

An isomorphism between $\VV_2$ and $\VV'_2$ is constructed in
Subsection \ref{subsect:ThmMain2}.

\textbf{In Section \ref{sec:Thm3}} we will prove Theorem \ref{th:main3}.

\medskip\noindent\textbf{Acknowledgements}.
 I am very grateful to Dierk Schleicher, Vladlen Timorin, Laurent Bartholdi, Nikita Selinger,
Yauhen Mikulich, Magnus Aspenberg for very useful discussions.

I am grateful to Dierk Schleicher for his help
in writing this paper.

\section{Background: $\VV_1$}
\label{sec:BasDef}

\subsection{B\"{o}ttcher functions, rays, and equipotentials}
\label{subsec:BFunct} For a polynomial $f_c(z)=z^2+c$ let
\begin{equation}
\rho_c =\lim_{n\to\infty} \frac{\log |f_c^{n}(0)|}{2^n}
\end{equation}
be the escape rate of the critical point $0$. There is a unique holomorphic
\textit{B\"{o}ttcher function} $B_c(z)$ such that
\begin{itemize}
\item $B_c(z)$ is a conformal conjugation between $f_c$ in a neighborhood
of infinity with the dynamics of $z^2$ in the disc $\{z|
\log|z|>\rho_c\}$; and
\item $B_c(z)/z\rightarrow 1$ when $z\rightarrow\infty$.
\end{itemize}
If $c\in\MM$, then the B\"{o}ttcher function maps the basin of
attraction of infinity conformally to the exterior of the unit disc;
if $c\not \in \MM$, then $0$ and all its pre-images are
critical points for the B\"{o}ttcher function.

 The parameter B\"{o}ttcher function is defined as
\begin{equation}
\label{eq:boettcher} \BB(c)=B_c(c).
\end{equation}

\begin{theorem}[\cite{DH1}]
\label{th:conf_isom}
 The map $\BB(c)$ is a conformal isomorphism from $\C\setminus \MM$ to the
 exterior of the unit disc.
\end{theorem}

In particular, it follows from Theorem \ref{th:conf_isom} that $\MM$
is connected.

Dynamic (resp.\ parameter) \textit{rays and equipotentials} are
defined as pre-images under $B_c(z)$ (resp.\ under $\BB(c)$) of
straight rays and concentric circles. Rays and equipotentials are parametrized
by \textit{angles} and \textit{heights} respectively.

Consider a dynamic ray $R^\phi=R^\phi(t),t>\rho_c\ge 0$. If the ray
does not hit any critical point of the B\"{o}ttcher function
$B_c(z)$ (i.e. either $0$ or its pre-images), then $B_c(z)$ is
analytically extendable along the ray $R^\phi(t)$ towards the Julia
set; this also extends the ray $R^\phi(t)$ to all $t \in
(0,\rho_c]$. Further, if the limit $\lim_{t\to 0}R^\phi(t)=p$
exists, then it is said that \textit{$R^\phi$ lands at $p$}.
Similarly, landing of parameter rays is defined. It is known that
parameter rays and (extendable) dynamic rays with rational (i.e.
periodic or pre-periodic) angles \textit{always land} (see for
example \cite{Mi}). For our convenience, we say that a ray does not
land if it goes through a critical point.

\subsection{Puzzle and parapuzzle pieces} \label{sec1}
The combinatorics of dynamic and parameter rays are closely related;
the cornerstone of this connection is expressed by the following
classical result:
\begin{theorem}[Correspondence Theorem]
\label{th:char_ray_pair} Assume that in the dynamical plane of
$f_c=z^2+c$ rays $R^{\phi_1}, R^{\phi_2}$ with rational angles land
together, separate the critical value $c$ from the critical point
$0$; and assume that no forward iterate of $R^{\phi_1}, R^{\phi_2}$
under $f_c$ separate the critical value $c$ from the pair
$R^{\phi_1}, R^{\phi_2}$. Then the parameter rays $\RR^{\phi_1},
\RR^{\phi_2}$ land together and separate $c$ from $0$.

In the inverse direction: if parameter rays $\RR^{\phi_1},
\RR^{\phi_2}$ with rational angles land together, then for every $c$
from the interior of the sector $\WW(\phi_1,\phi_2)$ bounded by
$\RR^{\phi_1}, \RR^{\phi_2}$ and not containing $0$ the rays
$R^{\phi_1}, R^{\phi_2}$ land together in the dynamical plane of
$f_c$ and have the above property. If $\phi_1,\phi_2$ are periodic
angles (i.e. they are of the form $m/(2n-1)$) and $c\not\in
\WW(\phi_1,\phi_2)$, then the rays $R^{\phi_1}, R^{\phi_2}$ do not
land together in the dynamical plane of $f_c$.
\end{theorem}

A ray pair $R^{\phi_1}, R^{\phi_2}$ satisfying the condition of the
above theorem is called \emph{characteristic} \cite{Mi}; the combinatorics of
 dynamic and parameter rays will be described in full generality
 in \cite{Sch4}.

\textit{A puzzle piece} $X(c)$ in the dynamical plane of $z^2+c$ is
a closed topological disc bounded by periodic and pre-periodic rays
and truncated by an equipotential such that the forward orbit of
$\partial X(c)$ does not intersect $\intr(X(c))$.
 By construction, the boundary of $X(c)$ consists
 of finitely many segments which belong to rays or an equipotential. \textit{We
will use floor brackets to untruncate} $X(c)$; i.e.\ $\lfloor
X(c)\rfloor$
 is the closed set containing $X(c)$ and bounded by the rays that form the boundary of
$X(c)$.

  Let $X(c_1)$ be a puzzle
piece in the dynamical plane of $z^2+c_1$. We will say that $X(c_1)$
\textit{exists} for $c_2\not=c_1$ if there exists a puzzle piece
$X(c_2)$ in the dynamical plane of $z^2+c_2$ such that the
boundaries of $X(c_1)$ and $X(c_2)$ are \textit{combinatorially
equivalent} in the following sense: there is an analytic
continuation of the composition $B_{c_2}^{-1}\circ B_{c_1}$ of the
B\"{o}ttcher functions from a neighborhood of infinity (where the
composition is defined canonically) to a neighborhood of
\[\left(\partial \lfloor X(c_1)\rfloor\cup \partial X(c_1)\right)\setminus\{\text{Julia set of } z^2+c_1\}\] such
that $B_{c_2}^{-1}\circ B_{c_1}$ maps $\partial X(c_1)$
homeomorphically onto $\partial X(c_2)$ except finitely many points
in the Julia sets that are landing points of periodic and
pre-periodic rays. Equivalently, there is a homeomorphism between
the boundaries $\partial X(c_1)$ and $\partial X(c_2)$ preserving
angles and heights of rays and equipotentials. \textit{To simplify,
we will often write} $X$ for $X(c_1)$ or $X(c_2)$; this convention
allows us to use the notion of puzzle pieces without referring to a
particular dynamical plane.

\textit{A parapuzzle piece} $\XX$ is a closed topological disc in the parameter plane bounded by
parameter rays and an equipotential such that $\XX$ is the parameter counterpart of a dynamic
puzzle piece, denoted by $X$ in the following sense: the composition $\BB^{-1}\circ
B_{c}$ of the B\"{o}ttcher functions  extends from a neighborhood of
infinity to a neighborhood of $(\partial X\cup
\partial \lfloor X\rfloor)\setminus J_c$ and maps $\partial X$
homeomorphically onto $\partial \XX'$ except finitely many points as
above. \textit{We will write} $\XX=\XX'$ if $\XX'$ is the parameter
counterpart to $X$.

 The following property is classical (it
follows from Theorem \ref{th:char_ray_pair}).
\begin{property}
\label{prop:existence_of_parapieces} If the interior of a puzzle
piece $X_n$ contains the critical value, then the parameter
counterpart $\XX_n$ exists. Moreover, the puzzle piece $X_n$
exists in $\intr(\XX_n)$.
\end{property}

\subsection{Generalized puzzle and parapuzzle pieces}
\label{subsec:DynParCoun}

 In this section we define generalized puzzle and
parapuzzle pieces; these pieces are often appear as the images of
puzzle and parapuzzle pieces in other slices under straightening
maps; see the right part of Figure \ref{figure:ParamTil} for an
example.

 We say that
a closed topological disc $X(c)$ in the dynamical plane of $z^2+c$
is a \textit{generalized puzzle piece} if:
\begin{itemize}
\item $\partial X(c)$ is a Jordan
curve such that $\partial X(c) \cap K_{c}$ consists of finitely many
periodic and preperiodic points;
\item there is a connected graph $\Gamma=\Gamma(X(c))$ in $\widehat{\C}\setminus K_c$, to which we will refer as \textit{marking}, such that
for every component $\beta$ of $\partial X(c)\setminus K_c$ there is
a simple curve $\gamma$ in $\Gamma\setminus \{\infty\}$
connecting $\beta$ and infinity;
\item the forward iteration of $\partial X(c)\cup \Gamma(X(c))$ does not intersect
$\intr(X(c))$.
\end{itemize}

 Let us emphasize that the exact choice of markings will be irrelevant
for us. We need markings only to distinguish generalized puzzle
pieces in the dynamical planes of polynomials with non-connected
Julia sets. We will usually omit the reference to markings.

 A \textit{pre-image}
$X^2(c)$ of a generalized puzzle piece $X(c)$ is the closure of a
connected component of $f_c^{-1}(\intr(X(c)))$. We define the marking
$\Gamma(X^2(c))$ of $X^2(c)$ to be the closure of the union of all
components in $f^{-1}_c(\Gamma(X(c)))\setminus  \{\infty\}$
intersecting $\partial X^2(c)$. By definition, a pre-image
of a generalized puzzle piece is a generalized puzzle piece.

 Two generalized puzzle pieces $X(c_1)$, $X(c_2)$ are
 \emph{combinatorially equivalent} if there is an analytic
continuation of $B_{c_2}^{-1}\circ B_{c_1}$ from a neighborhood of
infinity (where the composition is defined canonically) to a
neighborhood of
\[\partial X(c_1)\cup \Gamma(X(c_1))\setminus K_c\] such
that $B_{c_2}^{-1}\circ B_{c_1}$ maps $\partial
X(c_1)\cup\Gamma(X(c_1))\setminus K_{c_1}$ homeomorphically onto
$\partial X(c_2)\cup\Gamma(X(c_2))\setminus K_{c_1}$ and this map
extends to a homeomorphism between $\partial
X(c_1)\cup\Gamma(X(c_1))$ and $\partial X(c_2)\cup\Gamma(X(c_2))$.

A closed topological disc $\XX$ is a \textit{generalized parapuzzle piece} if $\XX$ is the parameter
counterpart to a dynamic generalized puzzle piece $X=X(c_1)$: the
composition $\BB^{-1}\circ B_{c_1}$ extends from a neighborhood of
infinity to a neighborhood of $\partial X(c_1)\cup
\Gamma(X(c_1))\setminus K_c$ and maps $\partial X(c_1)$
homeomorphically onto $\partial \XX$ except at finitely many points
as above.

Similar to Property
\ref{prop:existence_of_parapieces} we have:

\begin{property}
\label{prop:existence_of_paradiscs} If in the dynamical plane of $f_{c_0}$ the critical value $c_0$ is in the
interior of a generalized puzzle piece $X$, then $X$ has the parameter counterpart $\XX$ and $X$ exists for all
$c\in\intr(\XX)$.
\end{property}
\begin{proof}
The argument is classical, so we will give a sketch of the proof.

Define
\[\XX^o=\{c\in \VV_1:\ X \text{ exists for }f_c\text{ and }c\in\intr(X(c))\}.\]
By assumption, $c_0\in \XX^o$. If $c'\in\XX^o$, then for $c$ in a small neighborhood of $c'$
the set $\cup_{n\ge0}f_c^{-n}(c)$ is disjoint from $\partial X(c)$ because the forward orbit of
$\partial X(c)$ does not intersect $\intr(X)\ni c$. This implies that $\XX^o$ is open.

It follows that $c\in \partial \XX^o$ if and only if $\partial X(c)$ hits $c$ or
there is a parabolic point in $\partial X\cap K_c$.
 It is easy to conclude that $\XX=\ovl{\XX^o}$ is the parameter counterpart of $X$.
\end{proof}

We say that a generalized puzzle piece $X$ \textit{behaves periodically around the Julia set}, if there is a neighborhood $U$ of $K_c$
such that every connected curve $\gamma$ in  $\partial X \cap U$ is eventually periodic:
\[f^{n+p}_c(\gamma)\cap U\subset f^{n}_c(\gamma)\cap U\]for some $n\ge0$ and $p>0$. It is clear that this definition is independent on $c$.

\subsection{Analytic families of quadratic-like mappings}
\label{subset:QuadrLike} Let us recall the construction of the
canonical homeomorphisms between copies of the Mandelbrot set
\cite{DH}.

A branched covering $f_\lambda:U'_\lambda\to U_\lambda$ of degree $2$ is \textit{quadratic--like}
if $U'_\lambda,U_\lambda$ are open subsets of $\C$ and
the closure of $U'_\lambda$ is in $U_\lambda$.

Let \[\ff=(f_\lambda:U'_\lambda\to U_\lambda),\ \ \
\lambda\in \Lambda\] be a family of quadratic--like maps
parametrized by a complex analytic $1$-manifold $\Lambda$.

Define $\mathbf{U}=\{(\lambda,z):z\in U_\lambda\}$,
$\mathbf{U'}=\{(\lambda,z):z\in U'_\lambda\}$. We assume that
$\mathbf{U}, \mathbf{U}'$ are homeomorphic over $\Lambda$ to
$\Lambda\times \disk$. Denote by $\MM_\ff$ the set of parameters in
$\Lambda$ such that $f_\lambda:U'_\lambda\to U_\lambda$ has a
connected Julia set.

Let us outline that in the below construction the
straightening map $\chi:\Lambda\to \VV_1$ will be canonical over
$\MM_\ff$ (independent of the choice of a tubing $T$) and will depend on $T$ over
$\Lambda\setminus \MM_\ff$.

A \textit{tubing} is a diffeomorphic embedding $T$ over $\Lambda$ of
the family of annuli
\[Q_R\times \Lambda=\{z:R^{1/2}\le |z|\le R\}\times \Lambda\] into
$\mathbf{U}$ such that $T_\lambda(Q_R)$ surrounds the filled-in
Julia set of $f_\lambda$ and
\[f_\lambda(T_\lambda(z))=T_\lambda(z^2)\ \ \text{ for }   |z|=R^{1/2}.\]

Fixing $T$ we get a continuous map \[\chi=\chi(T):\Lambda\to \VV_1,\ \ \ \chi(f_\lambda)=\{z\to z^2+\chi(\lambda)\}\]
\cite[Theorem 2]{DH}. In dynamical planes
we have quasiconformal conjugacies, which we also denote by $\{\chi(\lambda)\}_\lambda$, between $\{f_\lambda\}_\lambda$
restricted to the closed discs bounded by the outer boundaries of $\{T_\lambda(Q_R)\}_\lambda$ and
$\{z\to z^2+\chi(\lambda)\}_\lambda$ restricted to
neighborhoods of the filled-in Julia sets. Moreover, the tubing $T$ commutes with
the B\"{o}ttcher coordinates in the following sense:
\begin{equation}
\label{eq:TubBottch}\chi(\lambda_2)\circ T_{\lambda_2}\circ
T_{\lambda_1}^{-1}= B_{\chi(\lambda_2)}^{-1}\circ
B_{\chi(\lambda_1)}\circ \chi(\lambda_1)\ \text{ on } Q_R.
\end{equation}

\cite[Theorem 4]{DH} states that if $\chi$ is not constant, then it is topologically holomorphic
 over $\MM$ (for instance, $\chi$ is locally a homeomorphism over $\MM$ except a discrete set); for some $T$ the map $\chi$ is topologically
holomorphic on a neighborhood of $\MM_\ff$ (see examples in
\cite{L3}). If, in addition, the degree of $\chi$ on $\MM_f$ is $1$, then
$\chi$ is a homeomorphism over $\MM_\ff$.

Let us remark that using the puzzle technique we often get a family
$\ff=(f_\lambda:U'_\lambda\to U_\lambda)$ with $U'_\lambda\subset
U_\lambda$ but $\partial U'_\lambda\cap\partial
U_\lambda\not=\emptyset $. However, in many cases it is possible to
perturb $\partial U'_\lambda$, $\partial U_\lambda$ and obtain
quadratic-like maps. These ``thickening'' constructions are well known (see for example \cite{Mi3}, \cite{L1}) and we will
not concentrate on details.

\begin{figure}%[hb]
\includegraphics[width=11cm]{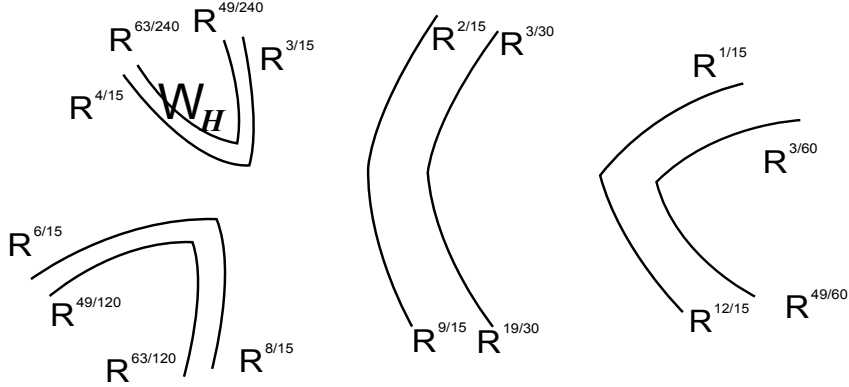}
  \caption{Ray configuration in the dynamical plane of
  $z^2+c$ for $c\in\WW_\HH$ and $\HH$ is the hyperbolic component with external
  angles $\{3/15,4/15\}$.
  }
  \label{figure:raypattern}
\end{figure}

\subsection{Small copies of the Mandelbrot set}
\label{subsec:SmallCopies} Small copies of the Mandelbrot set are in one-to-one
 correspondence with hyperbolic components: for every hyperbolic component $\HH\subset \MM$,
  there is a unique small copy $\MM_\HH\supset\HH$ so that the canonical homeomorphism of
  $\MM_\HH$ sends $\HH$ to the main hyperbolic component of $\MM$,
  and every small copy of $\MM$ is of this type for a unique component $\HH$.

Consider a hyperbolic component $\HH\subset\MM$. There are two
parameter rays $\RR^{\phi_1},\RR^{\phi_2}$ landing at the root of
$\HH$, where $a,b\in\mathbb Q /\Z$. The angles $\phi_1, \phi_2$ are
called the \textit{external angles} of $\HH$. Rays
$\RR^{\phi_1},\RR^{\phi_2}$ separate the closed sector $\WW_\HH$ from the
main hyperbolic component of $\MM$.

 The small copy $\MM_\HH$ is in $\WW_\HH$ and is described in the following way.

Let $p\in\mathbb{N}$ be the period of $\HH$; it means that $2^p\phi_1=\phi_1$,
$2^p\phi_2=\phi_2$, and $p$ is minimal with this property.

 Assume $c$ is in $\intr(\WW_\HH)$. In the dynamical plane of $f_c(z)=z^2+c$
 the rays $R^{\phi_1},R^{\phi_2}$ land together (Theorem \ref{th:char_ray_pair})
and bound the sector $W_{\HH}$ (the dynamical counterpart to $\WW_\HH$)
containing the critical value $c$. This sector contains no forward
images of $R^{\phi_1}$ and $R^{\phi_2}$. The pre-image $f_c^{-1}(W_\HH)$ is a strip
bounded by four rays; two of them are $R^{2^{p-1}\phi_1}$ and $R^{2^{p-1}\phi_2}$. Pulling back this strip
along the orbit of the rays $R^{\phi_1}$ and $R^{\phi_2}$ we get the
strip within $W_\HH$, we call it \textit{renormalization strip} $Y_\HH$ (see Figure
\ref{figure:raypattern}). The parameter renormalization strip
$\YY_{\HH}$ is defined as the parameter counterpart to $Y_{\HH}$.

Let $W_\HH^0$ be $W_\HH$ truncated by an equipotential and $Y_\HH^p$ be the pre-image of $W^0_\HH$ under
$f_c^p:Y_\HH\to W_\HH$; we have $f^p_c:Y_\HH^p\to W_\HH^0$ (see Figure \ref{figure:renormal}).

\begin{fact}
\label{fact:CopiesV1} The small copy $\MM_\HH$ is the closure of the set of
parameters $c\in\intr(\YY_\HH^p)$ such that in the dynamical plane the
critical value $c$ does not escape from $Y_\HH^p$ under iteration of
$f^p_c:Y_\HH^p\rightarrow W_\HH^0$.
\end{fact}

We will refer to $f^p_c:Y_\HH^p\rightarrow W_\HH^0$ as \textit{the
first renormalization map}. Let us also remark that for a puzzle
piece, say $\YY_i^n$, of a map in $\VV_1$ the upper index usually
denotes the \textit{depth} of the puzzle piece; i.e. $\YY_i^n$ is
a pre--image under $f^n_c$ of a puzzle piece with depth $0$.

\begin{figure}%[hb]
\includegraphics[width=9cm]{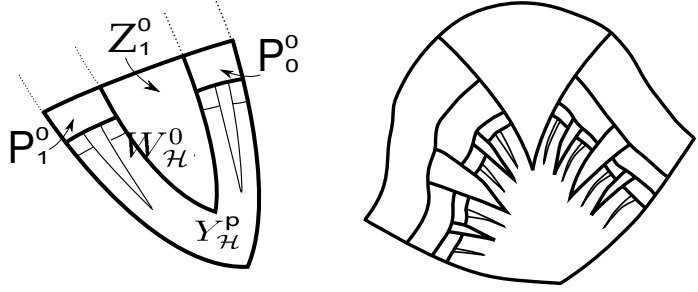}
  \caption{Left: the pieces $P_0^0$, $P_1^0$, $Z_1^0$ and their pre-images under
  $f^p_c:Y_\HH^p\rightarrow W_\HH^0$.
  Right: a scheme for the decoration tiling: pieces in the decoration tiling are iterated pre-images of $P_0^0$, $P_1^0$, $Z_1^0$
  under $f^p_c:Y_\HH^p\rightarrow W_\HH^0$.
  }
  \label{figure:renormal}
\end{figure}

\subsection{Decoration tiling}
\label{subsec:DecorTiling}

A decoration tiling is a certain partition of a neighborhood of a given small copy of the Mandelbrot set. We will first describe
a decoration tiling consisting of parapuzzle pieces (the left part of Figure \ref{figure:ParamTil}), after that we will define a generalized
decoration tiling (the right part of Figure \ref{figure:ParamTil}). In all cases we will first define
a dynamical partition, its parameter counterpart will be a (generalized) decoration tiling.

Choose a non-parabolic non-Misiurewicz parameter $c\in\MM_\HH$; then $c$ always stays in $\intr(Y^p_\HH)$
under iteration of $f_c^p:Y^p_\HH\to W^0_\HH$.

 Truncating $\WW_\HH$, $W_\HH$, and
$\YY_\HH$ at the equipotential at height $t$, we obtain puzzle
pieces $\WW^0_\HH$, $W_\HH^0$, and $\YY_\HH^0$; similarly, let
$Y^p_\HH$ be the truncation of $Y_\HH$ at the equipotential at
height $t/2^p$.  We decompose
$\ovl{W^0_\HH\setminus Y_\HH^p}$ into topological disks $Z_1^0\cup
P^0_0\cup P^0_1$ with a puzzle piece
$Z^0_1:=\overline{W_\HH^0\setminus Y_H^0}$ and two ``parallelogram''
disks $P_0^0$ and $P_1^0$ that are the components of $\overline
{Y_\HH^0\setminus Y_\HH^p}$, see Figure \ref{figure:renormal}. Let $T$ be the set of pullbacks of
$Z^0_1$, $P^0_0$, and $P^0_1$ under iteration of
$f^p_{c}:Y_\HH^p\rightarrow W_\HH^0$. The \textit{the decoration tiling} $\TT=\TT(\WW^0_\HH)$ is the
set of the parameter counterparts of pieces in $T$ (the left part of Figure \ref{figure:ParamTil}).

In \cite{D} we proved (see \cite{PR} for a different proof) the
decoration theorem that says, in a strong form, that the diameter of
pieces in the decoration tiling tends to $0$:
\begin{decotheorem}
\label{th:decor}
 For any $\varepsilon >
0$, there are at most finitely many pieces in the decoration tiling of
$\MM_\HH$ with diameter at least $\varepsilon$.
\end{decotheorem}

Consider a more general situation. Let us include the case
$\MM_\HH=\MM$ (i.e. $\HH$ is the main hyperbolic component of $\MM$) into the
consideration. For non-parabolic non-Misiurewicz $c_0\in\MM=\MM_\HH$ we define $W^0_\HH$ to be the
degenerate puzzle piece bounded by only the $0$-ray $R^0$ and truncated
by some equipotential. Then $Y_\HH^1$ is the pre-image of $W_\HH^0$
under $f_{c_0}$. For $c\in\intr(\WW_\HH^0)$ we have a (degenerate)
renormalization map $f_c:Y_\HH^1\to W_\HH^0$ so that $c\in\MM$ if
and only if $c$ does not escape from $Y_\HH^1$ under iteration of
$f_c:Y_\HH^1\to W_\HH^0$.

\begin{figure}%[hb]
\includegraphics[width=11cm]{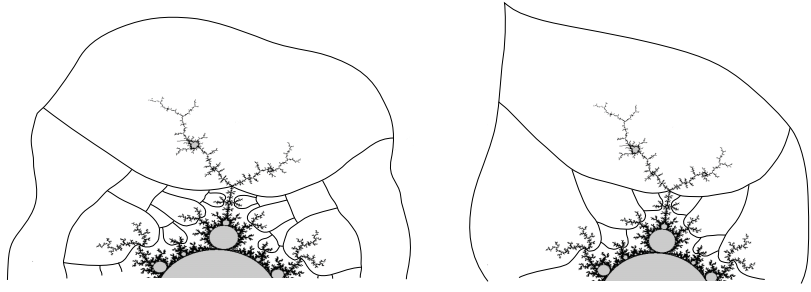}
  \caption{Left: the decoration tiling associated with the first renormalization map
  $f^p_c:Y_\HH^p\rightarrow W_\HH^0$; the scheme for the tiling is in Figure \ref{figure:renormal}.
  Right: an example of a generalized decoration tiling; this is a usual tiling for $\VV_2$.
  }
  \label{figure:ParamTil}
\end{figure}

Let  $\MM_\HH$ be a small copy of $\MM$, the case $\MM=\MM_\HH$ is
allowed, and let $c_0$ be a non-parabolic non-Misiurewicz point in $\MM_\HH$. Define
inductively $Y^{np}_{n\HH}$ as the
pullback of $Y^{(n-1)p}_{(n-1)\HH}$ under $f^p_{c_0}:Y^p_\HH\to
W^0_\HH$, where $Y^{p}_{1\HH}=Y^{p}_{\HH}$.

It follows from Fact \ref{fact:CopiesV1} that a parameter
$c\in \intr(\YY^{np}_{n\HH})$ is in $\MM_\HH$ if and only if $c$
does not escape from $Y^{(n+1)p}_{(n+1)\HH}$ under iteration of $f^p_c:
Y^{(n+1)p}_{(n+1)\HH}\to Y^{np}_{n\HH}$. We will refer to
$f^p_c:Y^{(n+1)p}_{(n+1)\HH}\to Y^{np}_{n\HH}$ as the $n$-th
renormalization map. (Let us note that
$\intr(\YY^{np}_{n\HH})\setminus \MM_\HH$ consists of finitely many
points: one of them is parabolic, all other points are Misiurewicz).

\begin{definition}[Generalized renormalization map]
\label{defn:GenRenormMap} Let $X^{np}$ be a generalized puzzle piece that behaves periodically around
the Julia set such that $\partial X^{np}$ is isotopic to
$\partial Y^{np}_{n\HH}$ relative the Julia set. Define $X^{(n+1)p}$ to be a unique pre-image of $X^{np}$ under $f^p_{c_0}$
such that $\partial X^{(n+1)p}$ is isotopic to $\partial
Y^{(n+1)p}_{(n+1)\HH}$ relative the Julia set.

If $X^{(n+1)p}\subset X^{np}$ and every component of $\partial X^{(n+1)p}\cap \partial X^{np}$ intersects $K_{c_0}$, then $f^p_{c_0}:X^{(n+1)p}\to
X^{np}$ is a \emph{generalized renormalization map}.
\end{definition}

It follows that $c_0\in\intr\left(X^{(n+1)p}\right)\subset\intr(X^{np}) $ and so
the parameter counterparts $\XX^{np}$ and $\XX^{(n+1)p}$ of $X^{np}$ and $X^{(n+1)p}$
exist (Property \ref{prop:existence_of_paradiscs}). For all $c\in \intr(\XX^{np})$ the map
$f^p_{c}:X^{(n+1)p}\to X^{np}$ is well defined; and $c\in\MM_\HH$ if
and only if $c$ does not escape from $X^{(n+1)p}$ under
iteration of $f^p_c:X^{(n+1)p}\to X^{np}$; in particular, the
last map encodes $\MM_\HH$.

It follows that all components of  $\ovl{X^{np}\setminus
X^{(n+1)p}}$ are (up to fixing markings) generalized puzzle pieces.
Let $T$ be the set of pre-images of $\ovl{X^{np}\setminus
X^{(n+1)p}}$ under $f^p_{c_0}:X^{(n+1)p}\to X^{np}$. The
\textit{generalized decoration tiling} $\TT$ associated with
$f^p_{c}:X^{(n+1)p}\to X^{np}$ is the parameter counterpart to $T$.
We have:

\begin{theorem}
\label{th:decor2} The diameter of pieces in the decoration tiling
associated with $f^p_{c}:X^{(n+1)p}\to X^{np}$ tends to $0$.
\end{theorem}

\subsection{Compact connected locally connected  subsets of the sphere}
\label{subs:LocalConnSubs}
Compact connected locally connected sets are characterized by the following theorem, see \cite[Theorem 4.4, p. 113]{Wh}.
\begin{theorem}
A connected compact set $K\subset\sph$ is locally connected if
and only if it satisfies the following properties:
\begin{itemize}
\item For every
$\varepsilon>0$ there are only finitely many connected components of $\sph\setminus K$
with diameter at least $\varepsilon$.
\item The boundary of every connected component of $\sph\setminus K$ is locally connected.
\end{itemize}
\end{theorem}

Let is also recall that the
continuous image of a compact, connected, locally connected set is
locally connected \cite[Theorem 3--22]{HY}.

Let $K_1,K_2\subset \sph$ be compact connected locally connected
subsets of the $2$--dimensional sphere $\sph$. Let $h:K_1\to K_2$ be a homeomorphism
and $U$ be a component of $\sph\setminus K_1$. Then
$h$ \textit{preserves the orientation of $U$}
if for every point $a\in
U$, every choice of different points $x,y,z\in \partial U$, and non-intersecting curves
$\gamma_x,\gamma_y, \gamma_z$ in $U\cup\{x,y,z\}$ connecting $a$ to
$x,y,z$ respectively there is a component $U'$ of $\sph\setminus
K_2$, a point $a'$ in $U'$, and non-intersecting curves
$\gamma'_x,\gamma'_y, \gamma'_z$ in $U'\cup\{h(x),h(y),h(z)\}$
connecting $a'$ to $h(x),h(y),h(z)$ respectively such that the
orientation of $\gamma_x,\gamma_y, \gamma_z$ around $a$ is the same
as the orientation of $\gamma'_x,\gamma'_y, \gamma'_z$ around $a'$.

\begin{lemma}
\label{lem:ORientOfCompl}
There is at most one component $U'$ of $\sph\setminus K_2$ satisfying the above condition.
\end{lemma}
\begin{proof}
By planarity, the set $\{h(x),h(y),h(z)\}$ is on the boundaries of at most two
components of $\sph\setminus K_2$.
If $h(x),h(y),h(z)\in \partial U_1$ and $h(x),h(y),h(z)\in \partial U_2$ for $U_1\not=U_2$, then
$h(x),h(y),h(z)$ have different orientations with respect to $U_1$ and $U_2$.
\end{proof}

A homeomorphism
$h:K_1\to K_2$ \textit{preserves the orientation of the complement}
if it preserves the orientation of every component of $\sph\setminus K_1$.
As a corollary of Carath\'{e}odory's theorem we get:

\begin{proposition}
\label{prop:ExtensOfhomeom} Let $h:K_1\to K_2$ be a homeomorphism
between compact, connected, locally connected subsets of $\sph$
preserving the orientation of the complement. Then $h$ extends to a
homeomorphism of $\sph$.
\end{proposition}
\begin{proof}
By Lemma \ref{lem:ORientOfCompl} the map $h$ induces the surjective map $h'$ from the set of components of
$\sph\setminus K_1$ to the set of components of $\sph\setminus K_2$.
By Carath\'{e}odory's theorem for every component $U$ of
$\sph\setminus K_1$ the map $h$ extends to a homeomorphism from
$K_1\cup U$ onto $K_2\cup h'(U)$. Applying Carath\'{e}odory's
theorem for all components of $\sph\setminus K_1$ we get a
surjective map, call it $\widetilde{h}$, from $\sph$ to $\sph$ such that
$\widetilde{h}|_{K_1}=h$ and $\widetilde{h}|_{\ovl U}$ is a
homeomorphism for every component $U$ of $\sph\setminus K_1$. As the
diameter of components in $\sph\setminus K_1$ and in $\sph\setminus
K_2$ tends to $0$ the map $\widetilde{h}$ is continuous. Therefore,
$\widetilde{h}$ is a homeomorphism.
\end{proof}

\subsection{Laminations}
\label{subs:lamination}

 Laminations were introduced by
Thurston to describe connected, locally connected Julia sets of
polynomials. Consider the closed unit disc $\disk$, denote by
$\uncirc$ the boundary of $\disk$; we parametrize $\uncirc$ by
angles between $0$ and $1$ and we denote by ``$<$'' the cyclic order
in $\uncirc$.

Below we will recall the definition and the construction of a lamination of the Basilica.

\begin{definition}[A lamination of the Basilica]
\label{def:BasLamin} A lamination $L_B$ of the Basilica is a
collection of curves, called leaves, in $\disk$ such that
\begin{itemize}
\item every leaf $l\in L$ connects two points of the form $a/(3\cdot 2^n)$, $b/(3\cdot 2^n)$ on $\uncirc$, i.e.
$l\cap \uncirc =\{a/(3\cdot2^n), b/(3\cdot2^n)\}$ and $a,b$ are
coprime to $6$; for every point of the form $a/(3\cdot2^n)$ there is
exactly one leaf that contains this point;
\item pairs $\{1/3,2/3\}$ and $\{5/6,1/6\}$ are connected by leaves;
\item if a pair $\{a,b\}\not=\{1/3,2/3\}$ is connected by a leaf, then
pairs in  $\{a,b,a/2,b/2\}$ that are on the same component of
 $\uncirc\setminus \{1/3,5/6\}$ are also connected by leaves;
\item the diameter of leaves in $L_B$ tends to $0$.
\end{itemize}
\end{definition}
Note that this definition gives an algorithm to construct $L_B$. Let
us remark that the last condition in Definition \ref{def:BasLamin} is
a special property of the Basilica; this implies, in particularly,
that $L_B$ is a closed set. For a
general polynomial the last condition in Definition
\ref{def:BasLamin} is replaced by weaker conditions.

Two points $a\not=b$ in $\disk$ are \textit{equivalent} under $L_B$ if they are in the same leaf of $L_B$. It follows that
$L_B$ is a closed equivalence relation and the restriction of $L_B$ to $\uncirc$ is invariant under $z\to z^2$: if $a,b\in \uncirc$ and
$a\sim_{L_B}b$, then $a^2\sim_{L_B} b^2$.
Define $K'_B=\disk/L_B$,  $J'_B=\partial K'_B=\uncirc/L_B$, and $f'_B:J'_B\to J'_B$ to be the quotient dynamics:
\[z^2/L_B: \uncirc/L_B \to \uncirc/L_B.\]
Extend $f'_B$ to the interior of $K'_B$ such that $f'_B$ has one superattractive periodic cycle with period $2$ that attracts all points in
$\intr(K'_B)$, and locally at the superattractive periodic cycle $f'^2_B$ is conjugate to $z\to z^2$. It follows from Thurston's theorem
that $f'_B:K'_B\to K'_B$ and $f_B:K_B\to K_B$  are topologically conjugate.

We will usually denote a leaf $l\in L_B$ by $\langle a,b\rangle$
such that $a,b$ are points in $\uncirc$ connected by $l$ and the
length of the arc \[\{x\in\uncirc :\ b\ge x \ge a\}\] is less than
$1/2$. For a leaf $l=\langle a/(3\cdot 2^n),b /(3\cdot 2^n)\rangle$
with $a$, $b$ coprime to $6$ we call $n$ the \textit{depth} of $l$.

 \textit{For simplicity, we will work
with $z^2$-invariant laminations.}

\begin{definition}[$z^2$-invariant laminations]
\label{def:BasRepr} The lamination $L_B$ of the Basilica is
$z^2$-invariant if all leaves are simple curves and
any leaf $\langle a,b\rangle\not=\langle 1/3,2/3\rangle$ is an iterated pre-image
of the $\langle 5/6,1/6\rangle$-leaf under $z\to z^{2}$.
\end{definition}
An example of $z^2$-invariant lamination is in Figure \ref{figure:BasLamin}.

\subsection{Moore's theorem}
\label{sec:Moore'sTh} The following theorem is useful in the mating
construction.

\begin{theorem}[Moore's theorem, \cite{Mu}]
\label{th:moor} Let $\sim$ be a closed equivalence relation on the
$2$-sphere $\mathbb{S}^2$ such that all equivalence classes are
connected and non-separating, and not all points are equivalent.
Then the quotient space $\mathbb{S}^2 /\sim$ is homeomorphic to the
$2$-sphere.
\end{theorem}

\section{Mating construction}
\label{sec:ParSp}

Consider a lamination $L_B$ of the Basilica; we assume that $L_B$ is
$z^2$-invariant. For our convenience, let us complete $\VV_1$ by adding a special point at $\infty$ so that $\VV_1$ is a
Riemann sphere. Then the composition $(1/z) \circ \BB$ of the parameter
B\"{o}ttcher function $\BB$ and $1/z$ maps $\VV_1\setminus\MM$
conformally onto the open unit disc $\mathbb{D}$.

 We denote by \[\LL_B=\bigcup_{l\in L_B} \ovl{(1/z \circ \BB)^{-1}(l|_{\mathbb{D}})}\] the
embedding of $L_B$ into $\VV_1$. By construction, leaves in $\LL_B$
land either at Misiurewicz parameters or at the Basilica parabolic
point, see Figure \ref{figure:BasLamin}.

We keep the same notation for leaves in $\LL_B$ as for $L_B$. This
implies that a leaf $l=\langle a,b\rangle\in \LL_B$ has the same
accesses to the Mandelbrot set as the parameter rays $\RR^a$ and
$\RR^b$.

The parameter leaf $\langle 1/3,2/3\rangle$ surrounds the Basilica
limb. Let $\WW_B$ be the closed topological disc containing the
Basilica limb and bounded by $\langle 1/3,2/3\rangle$.

\begin{definition}
\label{def:BasRelation} Two point $c_1\not=c_2$ in $\VV_1$ are
\emph{equivalent under $\LL_B$} if either $c_1$ and $c_2$ are in
$\WW_B$ or $c_1$ and $c_2$ are in the same leaf in $\LL_B$.
\end{definition}

\subsection{Combinatorics of leaves}
\label{subsec:CombinOfLeaves}

Consider a small copy $\MM_\HH\not\subset \WW_B$ and let $\YY_\HH$
be the renormalization strip of $\MM_\HH$. The boundary of $\YY_\HH$
consists of four rays, call them
$\RR^{\phi_1},\RR^{\psi_1},\RR^{\psi_2},\RR^{\phi_2}$. We assume
that $\RR^{\phi_1}$ and $\RR^{\phi_2}$ are the periodic (\textit{external})
rays of $\HH$ and $0<\phi_1<\psi_1<\psi_2<\phi_2<0$ is the cyclic
order.

We say $\LL_B$ \textit{bounds} $\MM_\HH$ from the right (resp from
the left) if there is a leaf $\langle a,b\rangle\in\LL_B$ such that
$a<\phi_1<\psi_1<b$ (resp.\ $a<\psi_2<\phi_2<b$); in particular, the
leaf $\langle a,b\rangle$ truncates the strip $\YY_\HH$ from the
right (resp. from the left).

\begin{proposition}
\label{prop:bound} The copy $\MM_\HH$ as above is bounded from the
left (resp.\ from the right) by $\LL_B$ if and only if there is an
$n$ such that $5/6<2^n\phi_1<1/6$ (resp.\ $5/6<2^n\phi_2<1/6$).
\end{proposition}
\begin{proof}
We will prove the case of being bounded from the left; the second
case is proved in a similar way.

Moreover, the proposition is equivalent to the following statement
involving $L_B$, where the equivalence follows from the definition
of the B\"{o}ttcher function: there is a leaf $l\in L_B$ separating
both $\phi_1,\psi_1$ from the center of the unit disc if and only if
there is an $n\ge 0$ such that $5/6<2^n\phi_1<1/6$. Let us prove the
last claim.

 If there is a leaf $l=\langle a,b\rangle\in L_B$
separating $\phi_1,\psi_1$ from the center of the unit disc, then
$\langle a,b\rangle\not=\langle 1/3,2/3\rangle$ and it follows that there is an $n$ such
that $\langle 2^n a,2^n b\rangle =\langle 5/6,1/6\rangle$.  By Definition
\ref{def:BasRepr} we have $5/6<2^n\phi_1<1/6$.

Suppose now that there is an $n$ such that $5/6<2^n\phi_1<1/6$. Let
us choose the minimal $n$ with the above property; it follows that
$n\le p$, where $p$ is the period of $\HH$; i.e. $2^p\phi_1 =\phi_1$
and $p$ is minimal.

 We claim that $5/6<2^n\psi_1<1/6$; in other words,
 $\phi_1,\psi_1\in\uncirc$ escape to the arc bounded by
 $\langle 5/6,1/6\rangle$ in the same pattern. Indeed, the above claim is a consequence
of the following trivial statement concerning the combinatorics of
 dynamic rays: for $c\in \MM_\HH\not\subset \WW_B$ the orbit $Y_\HH, f_c(Y_\HH) \dots,
f^{p-1}_c(Y_\HH)$ of the renormalization strip does not contain the rays
$R^{1/6}$ and $R^{5/6}$ (if it did happen, then the sector $\WW_\HH$
would contain $\RR^{1/3}$ or $\RR^{2/3}$, which would contradict to
$\MM_\HH\subset \MM\setminus\WW_B$).

 Pulling
back $\langle 5/6,1/6\rangle$-leaf $n$ times by $z\to z^2$ along the orbit
\newline$\phi_1, 2\phi_1,\dots,2^n\phi_1$ we get a leaf that separates
$\phi_1,\psi_1$ from the center of the unit disc.
\end{proof}

\begin{proposition}
\label{prop:bound2} Every copy $\MM_\HH\not\subset\WW_B$ is bounded
from at least one side. If $\MM_\HH$ is a subcopy of another copy of
the Mandelbrot set or is a satellite copy, then $\MM_\HH$ is bounded
from both sides.
\end{proposition}

\begin{proof}
Let $\RR^{\phi_1},\RR^{\phi_2}$ be the external rays for $\HH$;
together these rays form the boundary of $\WW_\HH$. Since
$\HH\not\in \WW_B$, either $0<\phi_1<\phi_2<1/3$ or
$2/3<\phi_1<\phi_2<1$.

We have:
\begin{itemize}
\item if $0<\phi_1<\phi_2<1/3$, then $5/6<0<\phi_1/2<\phi_2/2<1/6$;
\item if $2/3<\phi_1 <\phi_2<1$, then
$5/6<\phi_1/2+1/2<\phi_2/2+1/2<0<1/6$.
\end{itemize}

Let $p>0$ be the period of $\phi_1$, $\phi_2$; i.e.
$\phi_1=2^p\phi_1$, $\phi_2=2^p\phi_2$, and $p$ is the minimal. It
is well known (follows from the construction of $Y_\HH$, see Figure
\ref{figure:raypattern} for illustration) that either
$\phi_1/2=2^{p-1}\phi_1$, $\phi_2/2 +1/2=2^{p-1}\phi_2$ or
$\phi_2/2=2^{p-1}\phi_2$, $\phi_1/2 +1/2=2^{p-1}\phi_1$. It follows
from
\[0<\phi_1/2,\phi_2/2<1/6\ \ \text{ or } \ \ 5/6<\phi_1/2+1/2,\phi_2/2+1/2<0,\]
and Proposition \ref{prop:bound} that $\MM_\HH$ is bounded from at
least one side.

If $\MM_\HH$ is a satellite copy, then $R^{\phi_1}$ and $R^{\phi_2}$
are in one periodic cycle, thus $\MM_\HH$ is bounded from both
sides.

Let $\MM_{\HH'}$ be a subcopy of $\MM_\HH$ and let
$\RR^{\phi'_1},\RR^{\phi'_2}$ be the external rays of $\HH'$.
Consider the dynamical plane of $f_c$ for $c\in \MM_{\HH'}$.

We have just proved that there is a leaf $\langle a,b\rangle$ that
bounds $\MM_\HH$ from at least one side. As
$\MM_{\HH'}\subset\MM_\HH$ the forward orbits of
$R^{\phi'_1},R^{\phi'_2}$ will not escape from the strip $Y_\HH$
under iteration of $f^p_c:Y_\HH\rightarrow W_\HH$. Moreover,
the forward orbits of $R^{\phi'_1}$ and $R^{\phi'_2}$ visit both
sides of $Y_\HH$, whence $a<2^k\phi'_1<b$ and $a<2^m\phi'_2<b$ for
some $k,m\ge 0$. This implies that $5/6<2^{k+n}\phi'_1<1/6$ and
$5/6<2^{m+n}\phi'_2<1/6$ for some $n\ge0$. Therefore, $\MM_{\HH'}$
is bounded from both sides (Proposition \ref{prop:bound}).
\end{proof}

\subsection{Decoration tiling associated with $\LL_B$}
\label{subsect:DecorTilingWithLb}

Let $\MM_\HH\not\subset\WW_B$ be a small copy of the Mandelbrot set
bounded by $\LL_\HH$ from both sides. The goal of this subsection is
to prove the following proposition.
\begin{proposition}
\label{prop:DecTilWithBasil} There is a generalized decoration tiling $\TT$ of
$\MM_\HH$ such that a component of $\VV_1\setminus \cMM$ or a leaf in $\LL_B\setminus\WW_B$
intersects at most three pieces in $\TT$.
\end{proposition}
\begin{proof}
Let $c$ be a non-parabolic non-Misiurewicz point in $\MM_\HH$. The composition
$B_{c}^{-1}\circ \BB$ of B\"{o}ttcher functions embeds $\LL_B$ into
the dynamical plane of $f_c=z^2+c$; denote by
\[L^c_B=B_{c}^{-1}\circ \BB(\LL_B)=B_{c}^{-1}\circ (1/z)(L_B)=\bigcup_{l\in L_B} \ovl{(1/z \circ B_c)^{-1}(l|_{\mathbb{D}})}\] the embedding.

By construction (Definition \ref{def:BasRepr}), $L^c_B$ has the following description:
\[L^c_B=\langle 1/3,2/3\rangle\cup_{n\ge0}f_c^{-n}(\langle 5/6,1/6\rangle).\]
For instance,  if $l$ is a leaf of $L^c_B$ and $l\not=\langle 1/3,2/3\rangle$,
then pre-images of $l$ under $f^n_c$ are also leaves of $L^c_B$.

Let $\langle a,b\rangle$, $\langle c,d\rangle$ be the leaves of
$L^c_B$ of the maximal depth such that $\langle a,b\rangle$,
$\langle c,d\rangle$ bound the renormalization strip $Y_\HH$ from
the right and left respectively. Consider the second renormalization
map $f_c^p:Y^{2p}_{2\HH}\to Y^p_\HH$ of $\MM_\HH$. By definition,
the leaves $\langle a,b\rangle$, $\langle c,d\rangle$ truncate the
strip $Y_\HH=\lfloor Y^p_\HH\rfloor$.

We will now construct a generalized renormalization map $f^p_c:X^{2p}\to
X^p$ for $\MM_\HH$ with required properties\footnote{If $L_B$ is a
lamination as in Figure \ref{figure:BasLamin}, then it is enough to
define $X^p$ to be $Y_\HH$ truncated by $\langle a,b\rangle$ and $\langle c,d\rangle$}; see Figure \ref{figure:DecTilLaminn} for illustration.

Denote by $R^{\phi_1}, R^{\phi_2}$ and by $R^{\psi_1}, R^{\psi_2}$
the periodic and pre-periodic rays in $\partial Y_\HH$; we assume
that $0<\phi_1<\psi_1<\psi_2<\phi_2<0$ is the cyclic order.

By assumption, there are points $a_0, b_0$ on $\langle a,b\rangle$,
resp.\ points $c_0,d_0$ on $\langle c,d\rangle$, such that $a_0\in R^{\phi_1}$, $b_0\in R^{\psi_1}$,
and $a_0$ is between $a$
and $b_0$ on $\langle a,b\rangle$, resp.\ $c_0\in R^{\psi_2}$,
$d_0\in R^{\phi_2}$, and $c_0$ is between $c$ and $d_0$ on
$\langle c,d\rangle$.
 We denote by $\langle a_0,b_0\rangle$, resp.\
$\langle c_0,d_0\rangle$, the set of points on $\langle a,b\rangle$,
resp.\ on $\langle c,d\rangle$ that are between $a_0$ and $b_0$,
resp.\ $c_0$ and $d_0$. Define inductively $a_n$, $c_n$, resp.\
$b_n$, $d_n$ to be the pre-images of $a_{n-1}$, resp.\ of $d_{n-1}$
under $f_c^{p}:Y^{2p}_{2\HH}\to Y^p_\HH$ such that $a_n\in
R^{\phi_1}$ and $c_n\in R^{\psi_2}$, resp.\ $d_n\in R^{\phi_2}$ and $b_n\in R^{\psi_1}$. Choose simple curves $\alpha_0$, $\beta_0$,
$\gamma_0$, $\delta_0$ connecting the pairs $\{a_0,a_1\}$,
$\{b_0,b_1\}$, $\{c_0,c_1\}$, $\{d_0,d_1\}$ respectively such that
$\alpha_0$, $\beta_0$, $\gamma_0$, $\delta_0$ do not intersect
\[\cup_{n\ge0}f^{-n}_{c}(\langle
a,b\rangle)\cup_{n\ge0}f^{-n}_{c}(\langle c,d\rangle)\cup K_c.\]
 Define inductively $\alpha_n\ni a_n$ and $\gamma_n\ni c_n$, resp.\ $\delta_n\ni d_n$, $\beta_n\ni b_n$
to be the lifts of $\alpha_{n-1}$, resp.\ of $\delta_{n-1}$ under
$f_c^{p}:Y^{2p}_{2\HH}\to Y^p_\HH$.

 By construction, curves $\cup_i\alpha_i$,
$\cup_i\beta_i$, $\cup_i\gamma_i$, $\cup_i\delta_i$ are homotopic to
$R^{\phi_1}$, $R^{\psi_1}$, $R^{\psi_2}$, $R^{\phi_2}$ relative the
Julia set. Furthermore, $\cup_i\alpha_i$, $\cup_i\beta_i$
intersect $\langle a,b\rangle$ at $a_0,b_0$ and
$\cup_i\gamma_i$, $\cup_i\delta_i$ intersect $\langle c,d\rangle$
at $c_0,d_0$.

\begin{figure}%[hb]
\includegraphics[width=7cm]{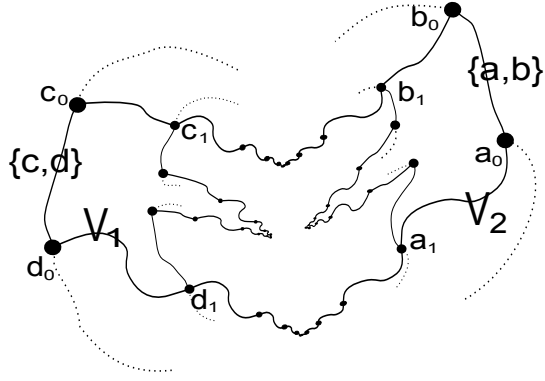}
  \caption{Illustration to the Proof of Proposition \ref{prop:DecTilWithBasil}. The domain $X^p$ is bounded by
  the curves $\cup_{i}\alpha_i$, $\cup_{i}\beta_i$, $\cup_{i}\gamma_i$, $\cup_{i}\delta_i$
  and is truncated by the leaves $\langle a,b\rangle$ and $\langle c,d\rangle$. The sets $V_1$, $V_2$ are
  bounded by $\langle a,b\rangle,$ $\langle c,d\rangle,$ $f^{-p}_c(\langle a,b\rangle\cup \langle a,b\rangle)$, and $K_c$;
  $V_1$, $V_2$ intersect at most three pieces in the generalized decoration tiling associated with $f^p_c:X^{2p}\to X^p$.
  }
  \label{figure:DecTilLaminn}
\end{figure}

Define $X^{p}$ to be the Jordan dick bounded by
\[\left(\cup_{i}\alpha_i\right)\bigcup\langle a_0,b_0\rangle \bigcup\left(\cup_{i}\beta_i\right)\bigcup\left(\cup_{i}
\gamma_i\right)\bigcup \langle
c_0,d_0\rangle\bigcup\left(\cup_{i}\delta_i\right);\] define
$X^{2p}$ to be the pre-image of $X^p$ under $f_c^p$ such that
$X^{2p}\subset X^p$. We obtain a generalized renormalization map
$f_c^p:X^{2p}\to X^p$.

Let $\TT$ be the decoration tiling associated with $f_c^p:X^{2p}\to
X^p$ and $T$ be the dynamical counterpart of $\TT$. Let us show that
$\TT$ satisfies the claims of the proposition; this is equivalent to
the following dynamical version: a component $U\in\C\setminus
(K_c\cup L^c_B)$ intersects at most three pieces in $T$; a leaf
$l\in L^c_B$ intersects at most three pieces in $T$.

Consider the set
\[S= \C \setminus \left(K_c \bigcup_{n\ge0} f^{-n}(\langle a,b\rangle\cup \langle c,d\rangle) \right).\]
By definition, $U$ is in a component, say $V$, of $S$. It follows
that $V$ is bounded by three leaves in $\cup_{n\ge 0}f^{-n}(\langle
a,b\rangle\cup \langle c,d\rangle)$ (see Figure
\ref{figure:DecTilLaminn}); and there are exactly two curves in
\[\Gamma=\bigcup_{n\ge0}f^{-n}\left(\cup_{i}\alpha_i\right)
\bigcup_{n\ge0}f^{-n}\left(\cup_{i}\beta_i\right)
\bigcup_{n\ge0}f^{-n}\left(\cup_{i}\gamma_i\right)
\bigcup_{n\ge0}f^{-n}\left(\cup_{i}\delta_i\right)\] intersecting
$V$. Recall that pieces in $T$ are bounded by curves in $\Gamma$ and
in $\cup_{n\ge 0}f^{-n}(\langle a,b\rangle\cup \langle c,d\rangle)$;
therefore, $V$ intersects at most three pieces in $T$ and so is
$U$.

Consider a leaf $l$. Then ether $l$ is in $\cup_{n\ge
0}f^{-n}(\langle a,b\rangle\cup \langle c,d\rangle)$ or $l$ is in a
component, say $V$, of $S$. If $l\subset V$, then $V$ intersects at most
three pieces in $T$ so the same is true for $l$. If $l$ is in
$\cup_{n\ge 0}f^{-n}(\langle a,b\rangle\cup \langle c,d\rangle)$,
then, by construction, $l$ forms the boundary of two pieces in $T$
and intersects at most one additional piece in $T$.
\end{proof}

As corollaries we have:

\begin{cor}
\label{cor:main}
 The diameter of components in $\VV_1\setminus (\WW_B\cup \LL_B\cup\MM)$
  tends to $0$. The diameter of leaves in
  $\LL_B\setminus \WW_B$ tends to $0$.
\end{cor}
\begin{proof}
By Yoccoz's theorem big pieces (with diameters at least a fixed
$\varepsilon>0$) in $\VV_1\setminus (\WW_B\cup \LL_B\cup\MM)$ can
not accumulate at finitely renormalizable parameters of the
Mandelbrot set. But by Proposition \ref{prop:DecTilWithBasil} and
the decoration theorem big components can not accumulate at
infinitely renormalizable parameters.

Similarly, the diameter of leaves in $\LL_B\setminus \WW_B$ tends to $0$.
\end{proof}
\begin{cor}
\label{cor:LamRe2} Leaves in $\LL_B\setminus \WW_B$ can only
accumulate at the Mandelbrot set.
\end{cor}

\begin{cor}
\label{cor:LamRel} The equivalence relation $\LL_B$ is closed and
satisfies the conditions of Moore's theorem.

The restriction of $\LL_B$ to $\MM$ is canonical (independent
of the choice of $L_B$).
\end{cor}
\begin{proof}
By Corollary \ref{cor:main} the equivalence relation $\LL_B$ is
closed. As different leaves of $\LL_B\setminus\WW_B$ land at
different points, equivalence classes in $\LL_B$ do not separate
points. Finally, points $a,b\in\MM$ are equivalent under $\LL_B$ if
and only if $a=b$, or $a,b$ are the end of a leave in $\LL_B$
(Corollary \ref{cor:main}), or $a,b\in \WW_B$.
\end{proof}

\subsection{The mating $\MM\mcup L_B$}
\label{subsect:MatingDefin}

Denote by $\VV'_2=\VV'_2(\LL_B)$ the quotient $\VV_1/\LL_B$ and by
$m:\VV_1\to \VV'_2$ the induced quotient map. By Moore's theorem $\VV'_2$ is topologically
a sphere.

Recall that a (continuous) map $m:\VV_1\rightarrow\VV'_2$ depends on
$\LL_B$, but the restriction $m:\MM\rightarrow\MM'_2$ is canonical,
i.e. it is independent of $\LL_B$.

Define \[\cMM:=\MM\bigcup\LL_B\bigcup \WW_B\] (recall that $\WW_B$
depends on $\LL_B$); by construction, $m^{-1}(\MM'_2)=\cMM$.

\begin{proposition}
\label{prop:main} The following holds:
\begin{itemize}
\item The boundary of a component of $\VV_1\setminus\cMM$
contains no infinitely renormalizable parameter.

\item The boundary of a component of $\VV_1\setminus\cMM$ is locally connected.

\item The diameter of components in
 $\VV_1\setminus\cMM$ tends to $0$. The set $\cMM$ is locally connected.

 \item The set $\MM'_2$ is locally connected.
All spaces $\VV'_2=\VV'_2(\LL_B)$ are homeomorphic by
homeomorphisms preserving $\MM'_2$.
\end{itemize}
\end{proposition}

\begin{proof}
Let $\EE$ be a component of $\VV_1\setminus\cMM$ and
$c_0\in \partial\EE$. By Proposition \ref{prop:bound2} every
secondary copy of the Mandelbrot set is separated from $\EE$; thus
$c_0$ is not infinitely renormalizable.

Let us prove local connectivity of $\partial \EE$ at $c_0$. If
$c_0\not\in \MM$, then $\partial \EE$ is locally connected at $c_0$
as leaves in $\LL_B$ can only accumulate at the Mandelbrot set
(Corollary \ref{cor:LamRe2}).

Assume that $c_0\in \MM$. Let us also suppose that $c_0$ is not the end
of a leaf in $\LL_B$ and there is exactly one access from $\EE$
to $c_0$.

Let $\PP_1,\PP_2,\dots$ be a sequence of parapuzzle pieces shrinking
to $c_0$ \cite{Hu}; i.e. $\PP_i\cap \MM$ is connected, $c_0\in\intr
\PP_i$, and $diam (\PP_i)\to 0$. Let $\gamma_i$ be the
component of $\PP_i\cap \partial \EE$ containing $c_0$. As $diam
(\PP_i)\to 0$ we have  $diam (\gamma_i)\to 0$.

Let us fix $i$ and show that $\gamma_i$ is a neighborhood of $c_0$
in $\partial \EE$.

 As $\EE$ is an open topological disc we have the cyclic order for
$\LL_B\cap\partial \EE$: if $\langle a_1,b_1\rangle$, $\langle a_2,b_2\rangle$,
$\langle a_3,b_3\rangle$ are in $\LL_B\cap
\partial \EE$, then up to permuting $\langle a_1,b_1\rangle$ and $\langle a_2,b_2\rangle$:\[a_1<b_1<a_2<b_2<a_3<b_3<a_1.\]
Let $\RR^{\phi}$ be the external ray landing at $c_0$ with the same
access to $c_0$ as $\EE$ relative the Mandelbrot set. Consider the
set
\[S_\varepsilon =\{\langle a,b\rangle\in  \LL_B\cap
\partial \EE: |\phi-a|<\varepsilon \text{ or } |b-\phi|<\varepsilon\}.\]
We will show that leaves in $S_\varepsilon$ are in $\PP_i$ for sufficiently
small $\varepsilon$ and if a leaf $\langle a,b\rangle$ is
sufficiently close to $c_0$ in the metric of $\VV_1$, then $\langle
a,b\rangle\in S_\varepsilon$. This will imply that
\[\ovl{\bigcup_{l\in S_\varepsilon}l}\subset \gamma_i\] is a
neighborhood of $c_0$ because leaves in $\LL_B\cap\partial \EE$ are
dense on $\partial \EE$.

It follows from triviality of fibers at non infinitely
renormalizable parameters that if $\langle a,b\rangle\in
S_\varepsilon$ for small $\varepsilon$, then $\langle a,b\rangle$ is
close to $c_0$ in the metric of $\VV_1$. By Corollary \ref{cor:main}
the diameter of leaves in $\LL_B$ tends to $0$. Therefore, there are
only finitely many leaves in $\LL_B$ that are close to $c_0$ but are
not in $\PP_i$. By choosing sufficiently small $\varepsilon$ we get
$S_\varepsilon\subset \PP_i$.

Assume that there is a sequence $\langle a_i,b_i\rangle$ of leaves
in $(\LL_B\cap
\partial \EE)\setminus S_{\varepsilon}$ accumulating at $c_0$. We may
assume that $a_i$ tends to $\psi\not=\phi$. Then $R^{\psi}$ lands at
$c_0$ because the fiber at $c_0$ is trivial. This contradicts the
assumption that there is only one access to $c_0$.

It is well known (follows from the branch theorem \cite{Sch}) that
there are at most two accesses to $c_0$ from $\VV_1\setminus \MM$;
so there are at most two accesses from $\EE$ to $c_0$. If there are
two accesses\footnote{In fact, it would follow from Theorem
\ref{th:main2} that $\partial \EE$ is a Jordan curve as it is so for
the boundaries of components in $\VV_2\setminus \MM_2$.}, then we
can present $\EE$ as $\EE_1\cup \EE_2$ such that $\EE_1,\EE_2$ are
open topological discs that do no intersect in some neighborhood of
$c_0$. The same argument as above shows local connectivity of
$\partial\EE_1$ and $\partial \EE_2$ at $c_0$; this implies local connectivity of
$\partial \EE$ at $c_0$.

If $c_0$ is an end of a leaf $\langle a,b\rangle\in \LL_B\cap \partial \EE$,
then we can apply the above argument to $\overline{\partial
\EE\setminus \langle a,b\rangle}$ and conclude local connectivity of
$\overline{\partial \EE\setminus \langle a,b\rangle}$ at $c_0$. This implies
local connectivity of $\partial \EE$ at $c_0$.

We have proved that $\partial \EE$ is locally connected. By
Corollary \ref{cor:main} the diameter of components in
$\VV_1\setminus\cMM$ tends to $0$. This implies that $\cMM$ is
locally connected.

Since the continuous image of a compact connected locally
connected set is locally connected \cite[Theorem 3--22]{HY} we
conclude that $\MM'_2$ is locally connected.

Finally, it follows from Proposition \ref{prop:ExtensOfhomeom} that $\VV'_2$ is independent of the choice of leaves in $L_B$ up to a homeomorphism
preserving $\MM'_2$.
\end{proof}

\section{Properties of $\VV_2$}
\label{sec:V2}

The second parameter slice $\VV_2$ (Figure \ref{figure:V2Slice}) is
the set of quadratic rational maps that have a periodic critical
point of period $2$. Every map in $\VV_2$ is conjugate to a unique
map of the form
\begin{equation}
\label{eq:V2Normaliz}
g_a(z)=\frac{a}{z^2+2z},
\end{equation}
where $a\in\widehat{\C}\setminus\{0\}$ and the case $a=\infty$
corresponds to the map $1/z^2$.

In normalization \eqref{eq:V2Normaliz}, infinity is a critical point
of period two and $-1$ is a free critical point (unless $a=\infty$).
The set $\MM_2\subset\VV_2$ consists of maps such that the free
critical point is not in the attracting basin of the
$\{\infty,0\}$-cycle.

For our convenience, let us complete $\VV_2$ by adding a special point at $0$ so that $\VV_2$ is Riemann sphere
and $\MM_2$ is a compact subset of $\VV_2$

The map
\[g_1=\frac{1}{z^2+2z}\]
is conjugate to the Basilica polynomial; we denote by
$\rho=-1/(z+1)$ the unique M\"{o}bius transformation that conjugates
$g_1$ to $f_B=z^2-1$:
\[\rho \circ g_1= f_B \circ \rho.\]

The hyperbolic component containing $g_1$ is \textit{the main hyperbolic component} of $\MM_2$; it is the biggest black
component on Figure \ref{figure:V2Slice}.

\subsection{Bubbles}
The idea to use bubbles to construct rays is due to Luo \cite{Lu}.
He also showed how to define bubble puzzles and parapuzzles.

Let $\Omega=\Omega(a)$ be the attracting basin of the
$0$-$\infty$-cycle and $E_0$, $E_\infty$ be the connected (periodic)
components of $\Omega$ containing $0$, $\infty$ respectively.

A \textit{bubble} of $g_a$ is a component $E$ of $\Omega$. The
generation of a bubble $E$ is the smallest non-negative $n = Gen(E)$
for which $g_a^n(E)=E_\infty$.

By definition, $-a\in \Omega(a)$ if and only if $a \in
\VV_2\setminus \MM_2$. It is easy to see that $-a$ can not be in
$E_0$ \cite[Proposition 2.5]{AY} because a non-fixed periodic Fatou
component can not contain all critical values.

 If
$-a\in E_\infty$, equivalently $-1\in E_0$, then $\Omega=E_0\cup
E_\infty$, and the map $g_a$ is quasi-conformally conjugate in a
neighborhood of the Julia set to $z\rightarrow z^{-2}$. The last
property holds if and only if $a$ is in the main unbounded
component on the Figure \ref{figure:V2Slice} which will be denoted by
$\EE_\infty$.

Furthermore, $-a\in
\partial E_\infty$ if and only if $a\in \partial \EE_\infty$; this
induces a parametrization of $\partial \EE_\infty$, see \cite{T}.

If $-a\not\in E_\infty$ (equivalently $a\not\in \EE_\infty$), then
$\Omega$ is a proper superset of $E_0\cup E_\infty$. Moreover, it follows from the following proposition
that $\ovl
E_0\cap \ovl E_\infty=\alpha=\alpha(g_a)$ is a fixed point of $g_a$.
\begin{proposition}
\label{prop:ExtCompLocalConn} The boundary of each bubble is a locally
connected simple closed curve. If $a\not\in \EE_\infty$, then  $\ovl
E_0\cap \ovl E_\infty$ contains a unique fixed point $\alpha=\alpha(g_a)$.
\end{proposition}
\begin{proof}
The case $a\in \ovl{\EE_\infty}$ follows from \cite[Theorem B$^{**}$]{T}.

If $a\in \widehat{\C}\setminus\ovl{\EE_\infty}$, then $\partial
E_0\cup
\partial E_\infty$ depends holomorphically on $a$ by the $\lambda$--lemma (see \cite{MC} for reference). In
particularly,  $\partial E_0$, $\partial E_\infty$ are locally
connected simple closed curves that intersect at a fixed point
because it is so for hyperbolic maps
in $\widehat{\C}\setminus\ovl{\EE_\infty}$. Since every bubble $E$ is
a preimage of $E_\infty$ we see that $\partial E$ is locally
connected simple closed curve.
\end{proof}

\begin{proposition}
\label{prop:Conj} Assume that $-a\not\in E_\infty$. Then there is a unique
conformal map $h_a: E_0(a)\cup E_{\infty}(a)\to E_0(1)\cup
E_{\infty}(1)$ that conjugates $g_a$ and $g_1$
 \begin{equation}
 \label{diagr:PropConj}
 \begin{array}[c]{ccc}
  E_0(a)\cup E_{\infty}(a)&\stackrel{}\longrightarrow^{g_a} &E_0(a)\cup E_{\infty}(a)\\
  \downarrow\scriptstyle{h_a}&&\downarrow\scriptstyle{h_a}\\
  E_0(1)\cup E_{\infty}(1)&\stackrel{}\longrightarrow^{g_1}&E_0(1)\cup
  E_{\infty}(1),
  \end{array}
\end{equation}
with normalization $h'_a(0)=h'_a(\infty)=1$. Moreover, $h_a$
extends homeomorphically to the $\alpha$-fixed point:
$h_a(\alpha(g_a))=\alpha(g_1)$.

Let $\gamma$ be a curve in \[\Omega(a)\cup_{n\ge 0}g_a^{-n}(\alpha(g_a))\] starting at $\infty$.
If $\gamma$ does not contain (hit) $-1$ or its preimage  under $g_a$, then $h_a$
extends uniquely along $\gamma$.
\end{proposition}
\begin{proof}
The first part of the proposition follows from the existence of B\"{o}ttcher coordinates and the definition
of $\alpha(g_a)$.

Let us prove the second part. We can present $\gamma$ as a union $\cup_{n\ge0} \gamma_{2n}$, where $\gamma_{2n}\subset\gamma_{2n+2}$, such that $\gamma_{2n}$ is a curve starting at infinity and
$g_a^{2n}(\gamma_{2n})\subset E_0(a)\cup E_{\infty}(a)\cup\alpha(g_a)$. Then $h_a$ is defined over $g_a^{2n}(\gamma_{2n})$. As $\gamma$
does not hit $-1$ or its preimages the map $g_a^{2n}:\gamma_{2n}\to g_a^{2n}(\gamma_{2n})$ is a covering.
 Hence $h_a$ extends along $\gamma_{2n}$.
\end{proof}

Assume that $-a\in E_\infty$ but $a\not =\infty$ (i.e.
$g_a\not=1/z^2$). Then there is a unique fixed under $g_a^2$ internal ray $\gamma_\infty$ emerging from infinity.
If $\gamma_\infty$ does not hit $-a$, then it lands at a fixed point of $J_a$; we define \emph{$\alpha(g_a)$ to be the landing
point of $\gamma_\infty$}. If $\gamma_\infty$ hits $-a$, then we say that $\alpha(g_a)$ is not defined; it will follow from Theorem \ref{thm:ParDynIsomV2} that
$\alpha(g_a)$ is defined if and only if the parameter $a$ is not on the parameter ray $\BB^0$.

Note that if $a\not\in\infty$, then $h_a$ is defined in neighborhoods of $0$ and
$\infty$; if $\alpha(g_a)$ is defined, then $h_a$ extends along $\gamma_\infty\cup g_a(\gamma_\infty)\cup \alpha(g_a)$.
It is easy to verify the following statement.
\begin{proposition}
\label{prop:Conj2}
Assume that $-a\in E_\infty$ and $\alpha(g_a)$ is defined.  Then $h_a$ exists in neighborhoods of $0$ and
$\infty$ containing $-a$.

Assume further that $\gamma$ is a curve in  \[\Omega(a)\cup_{n\ge 0}g_a^{-n}(\alpha(g_a))\] starting at $\infty$ such that
$\gamma$ is a concatenation $\gamma(1)\#\gamma(2)\#\dots \#\gamma(n)\#\beta  $ or $\gamma(1)\#\gamma(2)\#\gamma(3)\#\dots$, where
$\gamma(i)$ is a preimage of $\gamma_\infty$ and $\beta$ is a curve in $\Omega(a)$. If $\gamma$ does not contain (hit) $-1$ or its preimage  under $g_a$, then $h_a$
extends uniquely along $\gamma$.
\end{proposition}

\subsection{Rays in bubbles}
\label{subsec:RayInBubble} Recall that $g_1$ is conjugate by $\rho$
to the Basilica polynomial. Let $\gamma_0$, resp.\ $\gamma_\infty$,
be the straight internal ray in $E_0(1)$, resp.\ $E_\infty(1)$, that
connects $0$ and the $\alpha$-fixed point, resp.\ $\infty$ and
$\alpha$. We have:
\[g_1(\gamma_0\cup \gamma_\infty) = \gamma_\infty \cup \gamma_0.\]
Define $T$ to be \[\bigcup_{n\ge0} g_a^{-n}(\gamma_0\cup
\gamma_\infty);\] it is clear that $T$ is a tree. Furthermore, $T$
intersects $\partial \Omega$ at the pre-images of the
$\alpha$--fixed point; these pre-images will be called
\textit{$\basil$-points} (if we conjugate $g_1$ by $\rho$, then
$\basil$-points are the landing points of the rays $R^{n/(3\cdot
2^k)}$ for $k\ge 0$ and $n$ coprime to $6$). For $x\in T$ we define
$\gamma(x)$ to be a unique geodesic within $T$ connecting $\infty$
and $x$.

Consider a component $E$ of $\Omega$, denote by $n\ge0$
the generation of $E$. Let $x$ be a unique pre-image of $\infty$ in $E$
 and $\gamma''$ be an internal ray in $E_\infty$
connecting $\infty$ and a non-$\basil$ point on $\partial E_\infty$.
Then $g_1^n(x) =\infty$ and $\gamma''$ lifts conformally along the
orbit of $x$ to the curve, say $\gamma'$ in $E$ connecting $x$ and a
non-$\basil$ point on $E$; we call
\begin{equation}
\label{eq:InnFinRay} \gamma=\gamma(1)=\gamma' \cup \gamma(x)
\end{equation}
a \textit{finite ray in bubbles}.

An \textit{infinite ray in bubbles} is an infinite geodesic, say
$\gamma$, in $T$: \[\gamma:[0,\infty)\to T,\ \  \gamma(0)=\infty\ \
\text{ and }\ \gamma(t_1)\not = \gamma(t_2)\ \text{ if }\
t_1\not=t_2.\]

For $g_a$ define a \textit{finite} (resp.\ \textit{infinite})
\textit{ray in bubbles} $\gamma(a)$ to be a pre-image of a finite
(resp.\ infinite) ray in bubbles $\gamma(1)$ under $h_a$ provided
that $h^{-1}_a$ extends uniquely along $\gamma(1)$; if this is the
case, \textit{we will write} $\gamma=\gamma(1)=\gamma(a)$ and say
that $\gamma$ \textit{exists} for $g_a$.

The following lemma is a consequence of Propositions \ref{prop:Conj} and \ref{prop:Conj2} and the local connectivity of the closure
of a bubble (Proposition \ref{prop:ExtCompLocalConn}).
\begin{lemma}
\label{lem:RayInBubbleExists}
A ray in bubbles $\gamma$ exists for $g_a$ if and only if $\gamma$
does not hit the critical point $-1$ or its pre-image.
 If a
finite ray in bubbles exists, then it lands.
\end{lemma}

Define \[G(a)=\{x\in T=T(1):\ h_a^{-1}\text{ extends along }
\gamma(x)\},\] and define $T=T(a)$ to be $g_a^{-1}(G(a))$. It
follows that $T$ is a tree and any finite ray in bubbles is of the
form $\gamma' \cup \gamma(x)$, where $x$ is a pre-image of $\infty$,
the curve $\gamma(x)$ is a geodesic in $T(a)$ connecting $\infty$
and $x$, and $\gamma'$ is a conformal pre-image of an internal ray
in $E_\infty$.

We say that an angle $\phi$ is \textit{of type ``$\basil$''} if it is of the
form $n/(3\cdot 2^k)$ such that $k\ge 0$ and $n$ is coprime to $6$.
We say that a ray $R^{\phi}$ or $\RR^{\phi}$ is \textit{of type ``\textit{$\basil$}}'' (resp.\ ``non-$\basil$'') if
$\phi$ is a $\basil$ angle (resp.\ non-$\basil$ angle). The landing points of $\basil$ (resp. non-$\basil$) rays will also be
called \textit{$\basil$} (resp. non-$\basil$) points.

Recall that $\rho=-1/(z+1)$ conjugates $g_1$ and $f_B$.

\begin{definition}[$\VV_2$-twins]
\label{defn:V2Dupl} Let $\phi$ be a non-$\basil$ angle. A ray in
bubbles $\gamma$ is the \emph{$\VV_2$-twin} of an external ray
$R^\phi$ if $\rho(\gamma)$ and $R^{-\phi}$ land at the same point in
$J_B$.
 \end{definition}

For every ray with non-$\basil$ angle there is a unique
$\VV_2$-twin. \textit{From now on we will write} $\gamma =B^{\phi}$
if $\gamma$ is the $\VV_2$-twin of $R^\phi$. The \textit{angle} of
$\gamma=B^\phi$ is  $\phi$. The following lemma follows directly from the definition.

\begin{lemma}
Every ray in bubbles is a $\VV_2$--twin of a ray of maps in $\VV_1$.
\end{lemma}

\subsection{Parameter rays in bubbles}
\label{subsec:ParamRayInBubble} Assume $-1\in E\subset\Omega$, where
$E\not=E_\infty,$ and $n= Gen(E)$. Then $h_a$ extends uniquely to
\[\bigcup_{Gen(E)< n}\ovl E.\] In particularly, $h_a(-a)$
is defined and canonical (unique). If $-a\in E_\infty$, then $h_a(-a)$ is
defined and canonical by Proposition
\ref{prop:Conj2}.

It is well known that for components $\EE'\not=\EE''$ of
$\VV_2\setminus\MM_2$ the closures $\ovl{\EE'}$, $\ovl{\EE''}$
either do not intersect or $\ovl{\EE'}\cap\ovl{\EE''}$ is a
Misiurewicz parameter; which will be called a \textit{$\basil$ Misiurewicz parameter};
$\basil$ Misiurewicz parameters in $\VV_2$ are the matings of
$\basil$ Misiurewicz polynomials (i.e. the corresponding parameters are the landing points of $\basil$ parameter rays) in $\MM\setminus\WW_B$ and the Basilica polynomial.

\begin{theorem}
\label{thm:ParDynIsomV2}
The map
\begin{equation}
\label{eq:ParDynIsomV2}a\to \rho(h_a(-a))
\end{equation} is a
homeomorphism between \[\Sigma=\bigcup_{\EE\subset V_2\setminus
\MM_2} \ovl\EE\] (the union is taken over all components in
$V_2\setminus \MM_2$) and
\[\bigcup_{E\in \intr (K_B)\setminus\{1/2\text{--limb}\}}\ovl E\]
(the union is taken over all components in $\intr(K_B)
\setminus\{1/2\text{--limb}\}$).

Furthermore, let $c(\phi)\in \MM\setminus \WW_B$ be the landing
point of a $\basil$ ray $\RR^{\phi}$. Then  the mating
$f_{c(\phi)}\mcup f_B$ is in $\Sigma$ and the image of
$f_{c(\phi)}\mcup f_B$ under \eqref{eq:ParDynIsomV2} is the landing
point of $R^{\phi}$.
\end{theorem}
\begin{proof}
 This theorem is well known. The (difficult) case
$a\in\ovl{\EE_\infty}$ follows from \cite[Theorem C]{T}. The case
$a\in\Sigma\setminus\ovl{\EE_\infty}$ follows from the implicit
function theorem (it is straightforward that
\eqref{eq:ParDynIsomV2} is invertible for
$a\in\Sigma\setminus\ovl{\EE_\infty}$).

The second part of the theorem follows from the definitions of
$\VV_2$--twins and matings.
\end{proof}
In particular, it follows from Theorem \ref{thm:ParDynIsomV2} that the boundary of a component in
$\VV_2\setminus \MM_2$ is a locally connected simple closed curve.

For a non-$\basil$ angle $\phi$ the pre-image of $B^\phi$ under
\eqref{eq:ParDynIsomV2} is the \textit{parameter ray in bubbles}
$\BB^\phi$. We say that $\BB^\phi$ is \textit{the parameter
counterpart} to $B^\phi$ and is the $\VV_2$-\textit{twin} of
$\RR^\phi$.

\subsection{Rational rays. Rationally accessible points. Ray portraits}
We say that a ray is \textit{rational} if it is either periodic or
pre-periodic. A point in the dynamical plane is \textit{rationally accessible} if
it is either a parabolic or repelling periodic point or a pre-image
of a parabolic or repelling periodic point. A parameter in $\VV_1$
or in $\VV_2$ is \textit{rationally accessible} if it is either a parabolic or
Misiurewicz parameter.

Let $x$ be a rationally accessible point of a map in $\VV_1$. Then the
\textit{ray portrait} $P_x$ of $x$ is the set of rays landing at
$x$. Similarly, we define the \textit{ray portrait} $P_x^B$ of a
rationally accessible point $x$ of a map in $\VV_2$.

We say that a ray portrait is a \textit{non--$\basil$} portrait if it does
not contain a $\basil$ ray.

The \textit{parameter ray portrait} $\PP_c$ of a rationally accessible
parameter $c\in \MM$ is the set of parameter rays landing at $c$.
Similarly, we define the \textit{parameter ray portrait} $\PP^B_a$
of a rationally accessible parameter $a$ in $\MM_2$. A portrait $\PP_c$ is
\textit{non-$\basil$} if it does not contain a $\basil$ ray. Note
that for $c\in\VV_1\setminus \WW_B$ any $\basil$ portrait contains exactly
one ray.

We will show later that the set of non-$\basil$ parameter portraits in
$\VV_1\setminus \WW_B$ is in bijection preserving angles with the set of parameter portraits in $\VV_2$
  (Theorem
\ref{thm:ParBubblRay}).

\subsection{Bubble puzzle and parapuzzle}
\label{subsect:BubblPuzzle} A \textit{bubble puzzle piece} $Y^B(a)$
is a closed topological disc in the dynamical plane of $g_a$ bounded
by rational landing rays in bubbles such that the
forward orbit of $\partial Y^B(a)$ under $g_a$ does not intersect
$\intr(Y^B(a))$. By analogy to $\VV_1$ we define $\lfloor
Y^B(a)\rfloor$ to be the union of all rays that form the boundary of
$Y^B(a)$; because there is no equipotential in $\VV_2$ we have
$\partial Y^B(a) \subset\lfloor Y^B(a) \rfloor.$

Two bubble puzzle pieces $Y^B(a')$, $Y^B(a'')$ are
\textit{combinatorially equivalent} (simply, $Y^B=Y^B(a')=Y^B(a'')$)
if the composition $h^{-1}_{a''}\circ h_{a'}$ extends to $\lfloor
Y^B(a') \rfloor$ and induces a homeomorphism between $\partial
Y^B(a')$ and $\partial Y^B(a'')$ except at finitely many points in
$\partial Y^B(a')$, $\partial Y^B(a'')$ that are the landing points
of rays in bubbles. Equivalently, there is a homeomorphism between
$\partial Y^B(a')$ and $\partial Y^B(a'')$ preserving the angles of
rays.

A closed disc $\YY^B$ bounded by rational landing parameter rays in
bubbles is a \textit{parameter bubble puzzle piece} if $\YY^B$ is a
parameter counterpart to a puzzle piece, denoted by $Y^B$, in the
following sense: Map \eqref{eq:ParDynIsomV2} induces a homeomorphism
between $\partial \YY^B$ and $\partial Y^B$ except finitely many
points as above.

We say that a bubble parapuzzle piece $\YY^B$ is the
$\VV_2$-\textit{twin} of a parapuzzle piece $\YY$ in $\VV_1$ if
$\partial \YY^B$ is combinatorially equivalent to  $\partial \YY$ in
the following sense: a pair $\BB^{\phi_1}$, $\BB^{\phi_2}$ of
landing together rays is in $\lfloor \YY^B\rfloor$ if and only if a
pair $\RR^{\phi_1}$, $\RR^{\phi_2}$ of landing together rays is in
$\lfloor \YY\rfloor$.  Note that the height of the equipotential truncating $\YY$
is irrelevant for the definition of $\YY^B$.

Similarly, the $\VV_2$-twins of dynamic puzzle pieces are defined.
\textit{We will use an upper index} ``$B$'' to indicate the
$\VV_2$-twins of puzzle or parapuzzle pieces.

\section{Landing of rays in $\VV_2$}
\label{sec:ThmParBubblRay}
The goal of this section is to prove Theorem \ref{thm:ParBubblRay}
which is a $\VV_2$--analogy of Theorem \ref{th:char_ray_pair}.

 Assume that rational rays $\BB^{\phi_1},$
$\BB^{\phi_2}$ land together. We denote by
$\YY^B(\phi_1,\phi_2)$ the parameter disc bounded by
$\BB^{\phi_1}$, $ \BB^{\phi_2}$ and not containing $0$; later we will show that $\YY^B(\phi_1,\phi_2)$
is a parapuzzle piece. Similarly,
$Y^B(\phi_1,\phi_2)$ is the puzzle piece bounded by $B^{\phi_1},$ $
B^{\phi_2}$  and not containing
$0$ provided that $B^{\phi_1},$ $
B^{\phi_2}$ land together and the forward orbit of $B^{\phi_1},$ $
B^{\phi_2}$ does not intersect $\intr(Y^B(\phi_1,\phi_2))$. By definition, $\YY^B(\phi_1,\phi_2)$  is the parameter
counterpart to $Y^B(\phi_1,\phi_2)$ if $Y^B(\phi_1,\phi_2)$ exists.

We say that two ray portraits are \emph{equivalent} if the angles of rays in these portraits are in bijection.
We say that two sets $S_1,S_2$ of ray portraits \textit{are in bijection preserving angles} if there is a bijection
$\omega:S_1\rightarrow S_2$ such that for every $t\in S_1$ the portraits $t, \omega(t)$ are equivalent.

\begin{theorem}[Correspondence theorem in $\VV_2$]
\label{thm:ParBubblRay} Every rational parameter ray in
bubbles lands. Furthermore:
\begin{itemize}
\item A rational ray in bubbles $B^\phi$
lands in the dynamical plane of $g_a$ if and only if $a\not\in \cup_{n>0} \BB^{2^n\phi}$.

\item If rational rays $\BB^{\phi_1},\BB^{\phi_2}$ land together, then
$a\in\intr(\YY^B(\phi_1,\phi_2))$ if and only if the dynamic rays
$B^{\phi_1}, B^{\phi_2}$ land together and $-a\in
\intr(Y^B(\phi_1,\phi_2))$.

\item Parameter non--$\basil$ ray portraits in $\VV_1\setminus\WW_B$ and
in $\VV_2$ are in bijection preserving angles.

\item If a rational non-$\basil$ ray $\RR^\phi$ lands at $c\in\MM\setminus \WW_B$, then $B^\phi$ lands
at $m(c)$, where the map $c\to m(c)$ is defined in Definition \ref{defn:Map_cTo_ac}.
\end{itemize}
\end{theorem}
\begin{cor}
\label{cor:BubbleParAndPar} A bubble parapuzzle piece $\YY^B_t$
exists if and only if $\YY_t$ exists in $\VV_1\setminus\WW_B$. A
dynamic puzzle piece $X^B_n$ exists in $\intr(\YY^B_t)$ if and only
if $X_n$ exists in $\intr(\YY_t)$.
\end{cor}
\begin{proof}
Follows from Theorem \ref{thm:ParBubblRay} and the observation that
a babble parapuzzle piece $\YY_t^B$ is of the form
\[\ovl{\YY^B(\phi_1,\phi_2)\setminus \left(\cup_i
\YY^B(\psi_{1,i},\phi_{2,i})\right)}.\]
\end{proof}

The proof of Theorem \ref{thm:ParBubblRay} is organized as follows.
We will first prove a general result that every rationally accessible point is
the landing point of at least one rational ray in a tree of
pre-images (Theorem \ref{thm:TreeOfPreim}). Then we will deduce that every rationally accessible point of a
map in $\MM_2\setminus\{\basil \text{ Misiurewicz parameters}\}$ is the landing point of at least one
rational ray in bubbles (Proposition \ref{prop:LandOfRaysAtRatPts}).

Further, thanks to mating constructions, Theorem
\ref{thm:ParBubblRay} holds for hyperbolic components and
Misiurewicz parameters (Subsection \ref{subsec:MatCaptur}).

We say that a rationally accessible point $x$ of a map in $\VV_2$ is
\textit{transitive} if all rays in bubbles landing at $x$ are in
the same grand orbit. Similarly, we define the transitivity for
rationally accessible points of maps in $\VV_1$.

A rationally accessible parameter $x$ in $\MM$ or in $\MM_2$ is
\textit{transitive} if the angles of the rays landing at $x$ are
in the same grand orbit under the doubling map.

 A rational ray (parameter or dynamic) is \textit{transitive} (resp.\ \textit{intransitive}) if it lands at a transitive
 (resp.\ non-transitive or \textit{intransitive}) point.

In Subsection \ref{subsect:IntrCase} we will prove Theorem
\ref{thm:ParBubblRay} for the transitive case. This will allow us
to consider partitions of $\VV_2$ by transitive rational rays.
Using this we will show that Theorem \ref{thm:ParBubblRay} holds for
the intransitive case.

\subsection{Trees of pre-images}
\label{subsec:TreeOfPreIm}
Trees of pre-images are used in
\cite{nek:book} to give a symbolic description of the Julia sets of
sub-hyperbolic maps.

Consider a rational map $g:\widehat{\C}\to\widehat{\C}$ with
degree $d\ge2$ and let $P_g$ be the postcritical set of $g$. Choose
a base point ``$*$'' in $\widehat{\C}\setminus P_g$; for simplicity, let us assume that $*$ is not periodic. Let
\[V=\bigcup_{n\ge0}g^{-n}(*)\]
be the set of pre-images of $*$. If $v\in g^{-n}(*)$, then we say that the \textit{depth} of $v$ is $n$.
For each $*_i$ in $g^{-1}(*)$ choose a curve $\gamma_i$ in $\widehat{\C}\setminus P_g$ connecting
$*$ and $*_i$. Define
\[E=\bigcup_{n\ge0}g^{-n}(\cup_{i=0}^{d-1}\gamma_i)\]
to be the set of pre-images of $\cup_{i=0}^{d-1}\gamma_i$. The
\textit{tree of pre-images} $T=T(*,\{\gamma_i\})$ consists of the
vertex set $V$ and the edge set $E$. It follows that $T$ is a $d$--rooted tree.

A \textit{ray} $r$ in $T$ is an infinite geodesic
$\cup_{n\ge0}\gamma^{(n)}$ starting at $*$, where $\gamma^{(n)}$ is an edge in $T$. The \textit{image} of
$r$ is the ray $g(r)=\cup_{n\ge0}\gamma^{(n+1)}$. A ray $r$ is
\textit{periodic} if $g^p(r)=r$ for some $p>0$.

\begin{theorem}
\label{thm:TreeOfPreim} Every periodic (resp.\ preperiodic) rationally accessible point is the landing
point of at least one periodic (resp.\ preperiodic) ray in $T$.
\end{theorem}
\begin{proof}
Let $x$ be a periodic point with period $p\ge1$. Let us suppose that
$x$ is a repelling point. Then there is a neighborhood $U$ of $x$ so
that the dynamics $g^p$ in $U$ is conjugate to $z\to \lambda z$,
where $|\lambda|>1$. By Montel's theorem there is an $m>0$ such that
$g^{pm}(U)$ contains $T$. It follows that $U$ contains at least one
vertex $v\in V$ with depth $pm$.

Let $\alpha_0$ be the geodesic in $T$ starting at $g^{pm}f(v)=*$ and
ending at $v$. Define inductively $\alpha_n$ to be the lift of
$\alpha_{n-1}$ under \[g^{pm}:U\to g^{pm}(U)\supset T\] to the end
of $\alpha_{n-1}$. It follows that $\alpha_i\subset U$ for $i>0$ and
the diameter of $\alpha_i$ tends exponentially to $0$ because of
local expansion at $x$. We conclude that $\cup_{i\ge0}\alpha_i$
lands at $x$; it is obvious that $\cup_{i\ge0}\alpha_i$ is a
periodic ray ($\alpha_i$ is a preimage of $a_{i-1}$).

If $x$ is a parabolic point, then choose $U$ to be a repelling petal
attached to $x$. The same argument as above shows that $x$ is the
landing point of a periodic ray.

Assume that $x$ is a preperiodic rationally accessible point. Then for some $n>0$
the point $g^{n}(x)$ is a periodic rationally accessible point. We have just
proved that there is a periodic ray $r$ in $T$ landing at
$g^{n}(x)$. The lift of $r$ to $x$ is a preperiodic ray landing at
$x$.
\end{proof}

Recall that a rationally accessible point is non--$\basil$ if it is not a preperiodic point that is on the boundary of a bubble.

\begin{proposition}
\label{prop:LandOfRaysAtRatPts} Let $g_a$ be a map in $\MM_2\setminus\{\basil\text{ Misiurewicz parameters}\}.$
Then every periodic (resp.\ preperiodic) rationally accessible non--$\basil$ point is the
landing point of at least one periodic (resp.\ preperiodic) ray in bubbles.
\end{proposition}
\begin{proof}
Consider a rationally accessible point $x\in J_a$. The case when $x$ is on the
boundary of a bubble follows from local connectivity of the boundary of a bubbles
(Proposition \ref{prop:ExtCompLocalConn}). Let us assume that $x$ is
not on the boundary of a bubble.

Choose a base point $*$ close to $\infty$. Then two pre-images
$g_a^{-1}(*)$ are close to $0$ and $-2$. Denote by $E_{\infty},
E_0, E_{-2}$ the bubbles containing $\infty,0,-2$
respectively. Choose connections
\[\gamma_0\subset \ovl{E_\infty}\cup \ovl{E_0} \text{ and } \gamma_1\subset\ovl{E_\infty}\cup \ovl{E_{-2}}\] from
$*$ to $g_a^{-1}(*)$ so that $\gamma_0,\gamma_1$ coincide with rays
in bubbles except in small neighborhoods of $\{\infty\}$, $ \{0\}$,
and $\{-2\}$. By Theorem \ref{thm:TreeOfPreim} there is a ray $R$ in
$T=T(*,\gamma_0,\gamma_1)$ that lands at $x$. We will now show that there is a ray in bubbles that is ``close''
to $R$ and landing at $x$.

Consider $h_a$ as in Proposition \ref{prop:Conj}; by the assumption
on $g_a$ the map $h_a$ extends to
\[h_a:\bigcup_{E\subset\Omega(a)}\ovl E\to \bigcup_{E\subset\Omega(1)}\ovl E,\]
where the unions are taken over all components in $\Omega(a)$ and
$\Omega(1)$ respectively, and $h_a$ is a homeomorphism over $\cup_{E\subset\Omega(1)}E\cup\{\basil$-Misiurewicz parameters$\}$. Then $h_a(T)$ is a tree of preimages in
the dynamical plane of $g_1$. As $g_1$ is hyperbolic (it is
conjugate to the Basilica polynomial) all rays in $h_a(T)$ lands (see more details in \cite{nek:book}).
Denote by $x'$ the landing point of $h_a(R)$; it is clear that $x'$
is rationally accessible and is not on the boundary of a bubble; otherwise $x$
would also be on the boundary of a bubble.

 By definition, $h_a(R)$ is of the form $\cup_i \alpha_i$ such that
$\alpha_i$ connects vertices, say $*_{i-1}$ and $*_{i}$, of $h_a(T)$
with depth $i-1$ and $i$. Denote by $E(i)$ the bubble containing
$*_i$.  As $h_a(R)$ is rational
there is $m>0$ and $p>0$ such that for all $n>m$ the map $g_1^p$ maps
$E(n+p)$ onto $E(m)$. It follows that every bubble is at most finitely many times
in the sequence $E(0),E(1),E(2),\dots$; otherwise $x'$ would be on
the boundary of a bubble that is infinitely many times in the above sequence.

 Let $B^\phi$ be a ray in bubbles
landing at $x'$; this ray exists by the hyperbolicity of $g_1$. It
is clear that $B^\phi$ is in $\cup_i \ovl{E(i)}$. Moreover, we can
choose a disc $D(i)\subset E(i)$ and an annulus $A(i)\subset E(i)$ such that $B^\phi$
is in $h_a(R)\cup_i D(i)$, the annulus $A(i)$ surrounds $D(i)$, and
the modulus of $A(i)$ is at least a fixed $\varepsilon
>0$. Indeed, we choose some $D(1),\dots D(m)$ and $A(1),\dots A(m)$ and
we define $D(n+p)$, resp.\ $A(n+p)$ to be a conformal pullback of
$D(n)$, resp.\ $A(n)$.

Let us show that $h_a^{-1}(B^\phi)$ lands at $x$. By construction,
\[h_a^{-1}(B^\phi)\subset R\cup_i h^{-1}_a(D(i))\] so it is enough to
show that $diam (h^{-1}_a(D(i))$ tends to $0$.

Observe that the distance between $h^{-1}_a(D(i))$ and $J_a$ tends to
$0$; otherwise $\widehat{\C}$ would have an infinite area. But
$h^{-1}_a(D(i))$ is separated from $J_a$ by $h^{-1}_a(A(i))$ with
modulus at least $\varepsilon>0$. This implies that the diameter of
$h^{-1}_a(D(i))$ tends to $0$.
\end{proof}

\subsection{Properties of dynamic rays in bubbles}

Recall that periodic and preperiodic points of $g_a$ are solutions
of algebraic equations, so these points depends holomorphically on
$a$ unless they collide with other periodic or pre-periodic points
(the parabolic case).

\begin{lemma}[Repelling ray portraits are stable]
\label{lem:StabOfPerOrb}
 Let $B^{\phi}$ be a rational ray in bubbles landing at $x(a')$ such that
 $x(a')$ is a periodic or preperiodic repelling point of $g_{a'}$, then there is a neighborhood $\VV$ of the
parameter $a'$ such that for all $a\in\VV\setminus \cup_{i\ge1}
\BB^{2^i\phi}$ the ray $B^\phi$ lands at $x(a)$.
\end{lemma}
\begin{proof}
Assume that $x(a)$ is a repelling periodic point with period $p$ and
let $U(a)$ be a neighborhood of $x$ so that the dynamics $g_a^p$ is
conjugate to $z\to \lambda_a z$ in $U(a)$. For $a=a'$ there is a
pre-image $b(a')$ of $\infty$ such that $b(a')\in B^\phi\cap U(a')$.
The last property is stable under small perturbation unless $a\in \cup_{i\ge1} \BB^{2^i\phi}$.
Therefore, if $a\not\in  \cup_{i\ge1} \BB^{2^i\phi}$ is close to
$a'$, then $B^\phi$ lands at $x(a)$ as there is a local expansion in
$U(a)\ni x(a)$.

The preperiodic case follows from the periodic case.
\end{proof}

\begin{lemma}
\label{lem:StabOfRaysFromLambdaLemma} For a rational angle $\phi$ the closure of $B^\phi$ depends holomorphically on $a$
in $\VV_2\setminus \cup_{i\ge1} \ovl{\BB^{2^i\phi}}$.
\end{lemma}
\begin{proof}
By construction, $B^\phi$  moves holomorphically in $\VV_2\setminus \cup_{i\ge1} \ovl{\BB^{2^i\phi}}$. By the $\lambda$--lemma (see \cite{MC} for reference)
the motion of $B^\phi$ extends to its closure.
\end{proof}

\begin{lemma}
\label{lem:LandOfRaysFor}
Let $a$ be a parameter in $\MM_2$.
If $\ovl{\cup_{n\ge 0}g^n_a(-a)}$ does not intersect any ray in bubbles, then all rational rays in bubbles land
in the dynamical plane of $g_a$.
\end{lemma}
\begin{proof}
By assumption the map \[g_{a}:
\widehat{\C}\setminus g_{a}^{-1}\left(\ovl{\cup_{n\ge 0}g^n_a(-a)}\cup{0}\cup\{\infty\}\right) \to \widehat{\C}\setminus \left(
\ovl{\cup_{n\ge 0}g^n_a(-a)}\cup{0}\cup\{\infty\}\right)
\] is a covering of hyperbolic surfaces while
\[ \widehat{\C}\setminus
g_{a}^{-1} \left(\ovl{\cup_{n\ge 0}g^n_a(-a)}\cup{0}\cup\{\infty\}\right)
\hookrightarrow \widehat{\C}\setminus \left(
\ovl{\cup_{n\ge 0}g^n_a(-a)}\cup{0}\cup\{\infty\}\right)
\] is a contracting inclusion.

Let $B^{\psi}$ be a periodic ray in bubbles. Then it is of the form $\cup_{i\ge0}
\alpha_i$ such that $\alpha_{i+1}$ is a pre-image of $\alpha_{i}$
under $g_{a_0}^{p}$. As $\alpha_2$ is disjoint from $\ovl{\cup_{n\ge 0}g^n_a(-a)}\cup{0}\cup\{\infty\}$
the curve $\alpha_2$ has a finite hyperbolic diameter in $S=\widehat{\C}\setminus \left(
\ovl{\cup_{n\ge 0}g^n_a(-a)}\cup{0}\cup\{\infty\}\right)$.
Hence the hyperbolic diameter of $\alpha_{i}$ in $S$ is bounded and, moreover, if $B^{\psi}$ does not accumulate at $
\ovl{\cup_{n\ge 0}g^n_a(-a)}$, then the hyperbolic diameter of $\alpha_{i}$ tends to $0$. This implies that the Euclidean diameter of
$\alpha_{i}$ tends to $0$. We conclude that every accumulation point
of $B^{\psi}$ is a solution of $g^p_{a_0}(z)=z$. As the set of solutions
is totally disconnected the accumulation set of $B^{\psi}$ is one point; i.e.
$B^{\psi}$ lands.

 The preperiodic case follows from the periodic
case.

\end{proof}

\subsection{Hyperbolic components and Misiurewicz parameters}
\label{subsec:MatCaptur}
 Let $\HH\subset\VV_1\setminus \WW_B$
be a hyperbolic component. The \textit{$\VV_2$-twin} $\HH^B$ of
$\HH$ is the set of matings of polynomials in $\HH$ with the
Basilica polynomial. It is known that all hyperbolic components of $\MM_2$ are $\VV_2$--twins of hyperbolic
components in $\VV_1$

For a hyperbolic component $\HH\in\MM$ parameters in $\ovl\HH$ are
parametrized by the multiplier $\lambda_\HH$ of the associated with
$\HH$ non-repelling periodic cycle. Similarly, $\HH^B$ is parametrized
by the \textit{multiplier} $\lambda_{\HH^B}$. The \textit{canonical homeomorphism}
\begin{equation}
\label{eq:V2TwinsOfMating} c\to m_\HH(c),\ \ \ \ c\in \ovl\HH,\
m_\HH(c)\in\ovl{\HH^B}
\end{equation}
 is defined so that it preserves the multiplier: $\lambda_\HH(c)=\lambda_{\HH^B}(m(c))$.

\begin{lemma}[The hyperbolic case]
\label{lem:MatHYpPolyn} Non--$\basil$ rational ray portraits of $c\in \HH$
and rational ray portraits of $m_\HH(c)\in\HH^B$ are in bijection preserving angles.
\end{lemma}
\begin{proof}
Follows from the mating construction as it preserves non-$\basil$ ray portraits.
\end{proof}

\begin{cor}
\label{cor:HypCompAreDetermByRayPortr}
Different hyperbolic components in $\MM_2$ have different sets of the ray portraits of rationally accessible
points.
\end{cor}
\begin{proof}
Follows from Lemma \ref{lem:MatHYpPolyn} and the well known fact that different components in $\MM\setminus \WW_B$ have different sets of non-$\basil$
ray portraits.
\end{proof}

The \textit{root} $h$ of a hyperbolic component $\HH^B$ is a unique point $h$ on $\partial\HH^B$ such that $\lambda_{\HH^B}(h)=1$. The following proposition
says that the parabolic bifurcation in $\MM_2$ is the same as in $\MM\setminus\WW_B$.

\begin{proposition}
\label{prop:ParabolBirf}
Let $m(h)$ be a parabolic parameter in $\VV_2$.
\begin{itemize}
\item[(A)] There is a unique hyperbolic component $\HH^B$ of $\MM_2$ such that $m(h)$ is the root of $\HH^B$.
\end{itemize}
Let $\ovl q$ be the parabolic cycle in the dynamical plane of $m(h)$. Assume that $p$ is the period of $\ovl q$ and
$rp$ is the number of attracting parabolic petals attached to $\ovl q$.
\begin{itemize}
\item[(B)] In the dynamical planes of maps in $\HH^B$ there is a bounded attracting cycle $\ovl y$ (i.e. $\ovl y\not=\{\infty,0\}$)
with period $rp$ and there is a repelling periodic cycle $\ovl x$ with period $p$ such that:
\begin{itemize}
\item points in $\ovl x$ are on the boundaries
of periodic Fatou components that form the immediate attracting basin of $\ovl y$;
\item $\ovl x$ (of maps in $\HH^B$) and $\ovl q$ (of the map $g_{m(h)}$) have the same ray portraits.
\end{itemize}
\item[(C)] If $a\in \HH^B$ tends to $m(h)$, then $\ovl x$ and $\ovl y$ tends to $\ovl q$.
\item[(D)] Rational ray portraits of $m(h)$ and of maps in $\HH^B$ are in bijection preserving angles.
\item[(E)] If $r=1$, then $\HH^B$ is a unique hyperbolic component of $\MM_2$ containing $m(h)$ on its boundary.
\end{itemize}

Denote by $h$ be the root of $\HH$.
\begin{itemize}
\item[(F)] Repelling non-$\basil$ ray portraits of $h$ and $m(h)$ are in bijection preserving angles.
\item[(G)] Parabolic non-$\basil$ ray portraits of $h$ and $m(h)$ are in bijection preserving angles.
\end{itemize}

Assume that $r>1$ and let $\lambda_{\ovl q}$ be the multiplier of $\ovl q$. Then:
\begin{itemize}
\item[(H)] There is a unique hyperbolic component $\HH_2^B$ of $\MM_2$ such that $m(h)\in \ovl{\HH_2^B}$
and $\lambda_{\ovl q}=\lambda_{\HH^B_2}(m(h))$.
\item[(I)] In the dynamical planes of maps in $\HH_2^B$ there is a bounded attracting cycle $\ovl x$ (i.e. $\ovl x\not=\{\infty,0\}$)) with period $p$
and there is a repelling periodic cycle $\ovl y$ with period $rp$ such that
\begin{itemize}
\item points in $\ovl y$ are on the boundaries
of periodic Fatou components that form the immediate attracting basin of $\ovl x$;
\item the set of rays landing at points in $\ovl y$ (of maps in $\HH_2^B$) is in bijection preserving angles with
the set of rays landing at points in $\ovl q$ (of the map $g_{m(h)}$).
\end{itemize}
\item[(J)] The parameter $h$ is on the boundary of $\HH_2^B$.
\item[(K)] The components $\HH^B$ and $\HH_2^B$ are the only hyperbolic components in $\MM_2$ containing $m(h)$ on their boundaries.
\end{itemize}
\end{proposition}

\begin{proof}
The proof repeats arguments of the proof of \cite[Theorem 4.1]{Mi}, see also \cite[Lemma 4.4]{Mi}),
so we give only the sketch of the proof; and we will skip the calculations involving Taylor series.

Let us first verify (A), (B), (C), (D), (H), (I).
Parabolic points of a rational map $g$ are multiple
roots of equations $g^n(z)=z$; or, in other words, parabolic cycles are collisions of periodic cycles. If $g$ is a map in $\VV_2$,
then it has at most one
parabolic cycle that splits into two periodic cycles under a small perturbation because there is only one free critical point.

Let $\ovl q$, $p$, $r$ be as in the conditions of the proposition. Let $\UU$ be a sufficiently small neighborhood of $m(h)$.
For $a\in \UU\setminus\{m(h)\}$ the cycle $\ovl  q$ splits into two cycles with periods
$p$ and $rp$, call them $\ovl x$ and $\ovl y$; we assume that $\ovl x$ has period $p$ and $\ovl y$ has period $rp$.
Note that if $r=1$, then $\ovl x$ and $\ovl y$ interchange when
$a$ circles $m(h)$; i.e. $\ovl x$, $\ovl y$ are not distinguishable.

Let $\lambda_{\ovl x}(a)$, resp.\ $\lambda_{\ovl y}(a)$ be the multiplier of $\ovl x$, resp.\ $\ovl y$. When $a$ tends to $m(h)$
the multiplier $\lambda_{\ovl x}(a)$ tends to $\lambda_{\ovl q}$ while $\lambda_{\ovl y}(a)$ tends to $1$. As $\lambda_{\ovl y}(a)$ is a holomorphic
function it attains all values sufficiently close to $1$. Therefore, we can find a parameter path $\gamma^+(t)$ in $\UU$, where $t\in[0,\delta]$, starting at $m(h)$ such that $\lambda_{\ovl y}(\gamma(t))$
is real and is greater than $1$ for $t>0$. In the dynamical planes of $g_{\gamma^+(t)}$ attracting parabolic petals attached to $\ovl q$ become attracting Fatou components, the cycle
$\ovl y$ is attracting, and ray portraits of $\ovl q$ become ray portraits of $\ovl x$.

Define $\HH^B$
to be the component containing parameters on $\gamma^+(t)$ for $t>0$. Parts (B), (C) follows from the construction.
By Lemma \ref{lem:LandOfRaysFor} all rays in bubbles land in the dynamical plane of $g_{m(h)}$ because $\ovl{\cup_{n\ge 0}g^n_{m(h)}(-m(h))}\subset\{$attracting
petals$\}\cup\ovl q$. This implies that all ray portraits are stable on $\gamma^+(t)$:
parabolic ray portraits are stable by Part (B) while repelling portraits are stable by Lemma \ref{lem:StabOfPerOrb}. This shows Part (D).
If $h$ is also the root of a hyperbolic component $\HH'^B$, then maps in $\HH^B$ and in $\HH'^B$ have the same ray portraits;
it follows from Corollary \ref{cor:HypCompAreDetermByRayPortr} that $\HH'^B=\HH^B$. This finishes the proof of Part (A).

Assume that $r>1$; equivalently, $\lambda_{\ovl q}\not=1$. Then we can find a parameter path $\gamma^-(t)$ in $\UU$, where $t\in[0,\delta]$, starting at $m(h)$ such that $\lambda_{\ovl y}(\gamma(t))$
is real and is less than $1$ for $t>0$. In the dynamical planes of $g_{\gamma^-(t)}$ groups of
attracting parabolic petals attached to the same point in $\ovl q$ become attracting Fatou components, the cycle
$\ovl x$ is attracting, and rays in bubbles landing at points in $\ovl q$ land at points in $\ovl y$. Define $\HH^B_2$
to be the component containing parameters on $\gamma^-(t)$ for $t>0$. Parts (H), (I) are verified as above.

For maps in $\HH^B$ the periodic cycle $\ovl x$ is the unique periodic cycle with period $p$ on the boundaries of periodic Fatou components
that form the immediate attracting basin of $\ovl y$.
Let $\ovl x^{(1)}$
be the unique periodic cycle with period $p$ on the boundaries of periodic bounded Fatou components of maps in $\HH$.
By the mating construction ray portraits of $\ovl x$
and $\ovl x^{(1)}$ are in bijection preserving angles. Furthermore, periodic parabolic ray portraits of $h$ are in bijection preserving angles with ray portraits of $\ovl x^{(1)}$ for maps in $\HH_2$.
Therefore, Parts (F), (G) follows from Part (B).

If $r>1$, then $h$ is also on the boundary of a hyperbolic component $\HH'_2\not=\HH$. It follows that $\HH'_2$ and $\HH_2$ have the same ray portraits. Therefore,
$\HH'_2=\HH_2$ (Part $J$).

If a hyperbolic component $\HH'^B$ contains $m(h)$ on the boundary, then ray portraits of maps in $\HH'^B$
either in bijection preserving angles with ray portraits of maps in $\HH^B$ or in bijection preserving angles with ray portraits of maps in $\HH^B_2$.
Therefore, by Corollary \ref{cor:HypCompAreDetermByRayPortr}
either $\HH'^B=\HH^B$ or $\HH'^B=\HH^B_2$ (Parts (E) and (K)).
\end{proof}

It follows from Proposition \ref{prop:ParabolBirf} that the following definition is correct (does not depend on the choice of $m_\HH$).
\begin{definition}[The map $c\to m(c)$]
\label{defn:Map_cTo_ac}  If $c$ is
in the closure of a hyperbolic component $\HH$ of $\subset \MM\setminus \WW_B$, then $m(c)=m_\HH(c)\in \ovl{\HH^B}$ is
defined as in \eqref{eq:V2TwinsOfMating}.

If $c\in\MM\setminus \WW_B$ is a Misiurewicz
parameter, then $m(c)$ is the mating of $f_c$ and $f_B$.
\end{definition}

\begin{lemma}[Non--$\basil$ Misiurewicz case]
\label{lem:MatMisiurPolyn} Let $c\in \VV_1\setminus\WW_B$ be a
non-$\basil$ Misiurewicz parameter and $g_{m(c)}=f_c\mcup f_B$ be
the mating. Then non-$\basil$ ray portraits of $f_c$ and $g_{m(c)}$
are in bijection preserving angles.
\end{lemma}
\begin{proof}
Follows from the mating construction as it preserves non-$\basil$ ray portraits.
\end{proof}

\begin{lemma}[$\basil$ Misiurewicz case]
\label{lem:MatMisiurPolynTriplingCase} Let $c\in
\VV_1\setminus\WW_B$ be a $\basil$ Misiurewicz parameter and
$g_{m(c)}=f_c\mcup f_B$ be the mating. Consider a
ray portrait $P_y$. The ray portrait $P_y^B$ exists for $g_{m(c)}$ if and only if
\[m(c)\not \in \bigcup_{R^\phi\in P_y}\bigcup_{n>0}\BB^{2^n\phi}.\]
\end{lemma}
\begin{proof}
The portrait $P_y$ survives during the mating if and only if for
every $R^{\phi}$ in $P_y$ the ray in bubbles $B^\phi$ exists in the
dynamical plane of $g_{m(c)}$. Recall that $B^\phi$ exists if and
only if it does not hit a pre-image of $-a$. Equivalently, $B^\phi$
exists if and only if there is no $n>0$ such that $-a\in
B^{2^n\phi}$. This is equivalent to the condition stated in the
lemma.
\end{proof}

Consider \[\Sigma=\bigcup_{\EE\subset V_2\setminus \MM_2} \ovl\EE,\]
as in Theorem \ref{thm:ParDynIsomV2} (the union is taken over all
components in $V_2\setminus \MM_2$). By definition, $\Sigma$ is an
infinite tree of closed discs; note that $\Sigma$ is neither closed
nor open. Recall that all parameter rays in bubbles are in $\Sigma$,
start at infinity, and have the cyclic order around infinity. For
rational angles $\phi_1$, $\phi_2$ with $0<\phi_1<\phi_2<0$ we
denote by $\Sigma(\phi_1,\phi_2)$ the set of parameters in $\Sigma$
that are between the rays $\BB^{\phi_1}$ and $\BB^{\phi_2}$ in the
anti-clockwise direction. If $\YY^B(\phi_1,\phi_2)$ exists, then
$\Sigma(\phi_1,\phi_2)=\Sigma \cap \intr(\YY^B(\phi_1,\phi_2))$.

\begin{lemma}[Theorem \ref{thm:TreeOfPreim} holds for $\Sigma$]
\label{lem:CorrTheoremForSigma} Assume rays $\RR^{\phi_1},$ $\RR^{\phi_2}$ land together.
Then the following are equivalent:
\begin{itemize}
\item $a\in \intr(\Sigma(\phi_1,\phi_2))$.
\item $Y^B(\phi_1,\phi_2)$ exists for $g_a$ and
$-a\in\intr(Y^B(\phi_1,\phi_2))$.
\end{itemize}
\end{lemma}
\begin{proof}
Consider a component $\UU$ of $\intr(\Sigma)\setminus (\cup_{n\ge0}
\BB^{2^n\phi_1}\cup_{n\ge0} \BB^{2^n\phi_2})$. We need to verify the
claim for maps in $\UU$. Note that maps in $\UU$ are hyperbolic so all
rays in bubbles that exist land. We only need to verify that the
rays land in expected patterns.

Choose a $\basil$ Misiurewicz parameter $m(c)\in
\partial \UU \setminus (\cup_{n\ge0} \BB^{2^n\phi_1}\cup_{n\ge0}
\BB^{2^n\phi_2})$. By the stability of periodic rays (Lemma
\ref{lem:StabOfPerOrb}) it is sufficient to verify the claim for
$m(c)$.

By the second part of Theorem \ref{thm:ParDynIsomV2} we have:
\begin{itemize}
\item[(A)] $m(c)\in \Sigma(\phi_1,\phi_2)$
\end{itemize} if and only if
\begin{itemize}
\item[(B)] $c\in
\intr(\YY(\phi_1,\phi_2))$.
\end{itemize} It follows from Lemma
\ref{lem:MatMisiurPolynTriplingCase} that the following are equivalent:
\begin{itemize}
\item[(C)] $Y^B(\phi_1,\phi_2)$
exists for $g_{m(c)}$ and $-m(c)\in\intr(Y^B(\phi_1,\phi_2))$;
\item[(D)] $Y(\phi_1,\phi_2)$ exists for $f_c$ and
$c\in\intr(Y(\phi_1,\phi_2))$.
\end{itemize} By the
correspondence theorem for $\VV_1$ (Theorem \ref{th:char_ray_pair}) Claims (B) and (D) are equivalent. Therefore,
(A) is equivalent to (C).
\end{proof}

\subsection{The transitive case}
\label{subsect:IntrCase}

\begin{proposition}[Theorem \ref{thm:TreeOfPreim}, the transitive case]
\label{prop:IntransCase} Assume that transitive rays
$\RR^{\phi_1}$, $\RR^{\phi_2}$ with rational angles land together in
$\VV_1\setminus \WW_B$. Then $\YY^B(\phi_1,\phi_2)$ exists in
$\VV_2$ and for all $a\in\intr \YY^B(\phi_1,\phi_2)$ the bubble
puzzle piece $Y^B(\phi_1,\phi_2)$ exists and $-a\in
\intr(Y^B(\phi_1,\phi_2))$.

If $h$ is the landing point of $\RR^{\phi_1},$ $\RR^{\phi_2}$, then $m(h)$ is the landing point of
$\BB^{\phi_1},$ $\BB^{\phi_2}$.
\end{proposition}

\begin{proof}
We will adopt arguments in \cite{Mi}.

Assume first that $\phi_1,\phi_2$ are periodic angles. Denote by
$\XX$ the set of parameters $a\in \VV_2$ such that the rays
$B^{\phi_1}$, $B^{\phi_2}$ land together in the dynamical plane of
$g_a$, the common landing point $\gamma(a)$ is transitive
repelling, and $-a\in\intr(Y^B(\phi_1,\phi_2))$.
 It is easy to see that $\XX$ is open. We need to show that the rays
$\BB^{\phi_1}$, $\BB^{\phi_2}$ land together at $m(h)\in\VV_2$ and
$\XX=\intr(\YY^B(\phi_1,\phi_2))$.

Observe that $\intr(\XX)$ and $\intr(\VV_2\setminus \XX)$ are not
empty. Indeed, it follows from Lemma \ref{lem:MatHYpPolyn} that a
hyperbolic component $\HH$ is in $\YY(\phi_1,\phi_2)$ if and only if
$\HH^{B}$ is in $\XX$. Hence the claim follows from the easy fact
that there are hyperbolic components inside as well as outside of
$\YY(\phi_1,\phi_2)$.

Consider a point $a_0\in \partial \XX$. By definition, one of the
following holds:
\begin{itemize}
\item[(A)] $\gamma(a_0)$ is not a repelling point;
\item[(B)] $Y^B(\phi_1,\phi_2)$ does not exist;
\item[(C)] $Y^B(\phi_1,\phi_2)$ exists but $-a_0\not\in\intr(Y^B(\phi_1,\phi_2))$;
\item[(D)] $\gamma(a_0)$ is not transitive.
\end{itemize}

If \textbf{Case (A)} occur, then $a_0$ is on the boundary of a hyperbolic component.
There are finitely many hyperbolic components such that the associated non-repelling cycle has the same period
as $\gamma(a)$. Therefore, we can write $a_0\in \cup_{i=1}^n \partial\ovl {\HH_i^B}$.
Denote by $\mathcal I$ the set of indifferent parameters (Siegel or Cremer parameters) on $ \cup_{i=1}^n\partial\ovl {\HH_i^B}$. Then either:
\begin{itemize}
\item[(A1)] $a_0\in \mathcal I$; or
\item[(A2)] $a_0$ is a parabolic parameter.
\end{itemize}
Note that $\mathcal I$ is totally disconnected because it is totally disconnected subset of a closed set $\cup_{i=1}^n \partial\ovl {\HH_i^B}$;
 we will exclude Case (A1) later.

 Consider \textbf{Case (A2)}. Let $\ovl q$ be the parabolic periodic cycle  in the dynamical plane of $g_{a_0}$; by the assumption,
$\gamma(a_0)$ is in $\ovl q$ (more precisely,  $\ovl q$ is a collision of the periodic  cycle containing $\gamma(a_0)$
with another periodic cycle).
By Lemma \ref{lem:LandOfRaysFor} all rational rays in bubbles land in the dynamical plane of $g_{a_0}$ because $\ovl{\cup_{n\ge 0}g^n_{a_0}(-a_0)}
\subset\{$attracting petals$\}\cup \ovl q$.
If $B^{\phi_1}$
 does not land at $\gamma(a_0)\in \ovl q$, then it lands at a periodic repelling point and,
by the stability of repelling portraits (Lemma \ref{lem:StabOfPerOrb}), $B^{\phi_1}$ does not land at $\gamma(a)$
 in a neighborhood of $a_0$. This contradicts the assumption that $a_0\in \partial \XX$. Therefore,
$B^{\phi_1}$ (and, similarly, $B^{\phi_2}$) lands at
$\gamma(a_0)\in\ovl q$. We claim that there is a unique parameter
with the above properties.

 Let us write
$a_0=m(c_0)$ as in Definition \ref{defn:Map_cTo_ac} so that $c_0$ is the root of a hyperbolic component, say $\HH$. By Proposition \ref{prop:ParabolBirf}, Part (G) the rays $R^{\phi_1}$ and $R^{\phi_2}$ land together at a parabolic periodic point in the dynamical
plane of $f_{c_0}$. Then $R^{\phi_1}$, $R^{\phi_2}$ separate $0$ and the critical value $c_0$ because this property depends only on the combinatorics of $R^{\phi_1}$, $R^{\phi_2}$.
This implies that $c_0$ is the landing point of the parameter rays $\RR^{\phi_1}$ and $\RR^{\phi_2}$ by the correspondence theorem in $\VV_1$. In the
inverse direction: by Lemma \ref{lem:MatHYpPolyn} the component $\HH^B$ is in $\XX$ and so its root $a_0=m(c_0)$ is on $\partial \XX$
by Proposition \ref{prop:ParabolBirf}, Part (B). As $R^{\phi_1}$, $R^{\phi_2}$ are transitive $\HH$ is not primitive. Therefore, there is
a component $\HH_2\not=\HH$ such that $c_0\in \ovl{\HH_2}$. It follows that $\HH_2\not\subset\YY(\phi_1,\phi_2)$; this implies that
$\HH_2^B\subset \VV_2\setminus\ovl \XX$. We get $a_0=m(c_0)\in \partial(\VV_2\setminus\ovl \XX)$.

We denote by $m(h)=a_0=m(c_0)$ the unique point in Case (A2).

Consider \textbf{Case (B)}. If $Y^B(\phi_1,\phi_2)$ does not exist, then
either:
\begin{itemize}
\item[(B1)] at least one of the rays $B^{\phi_1}$, $B^{\phi_2}$ does not
exist;
\item[(B2)] $B^{\phi_1}$, $B^{\phi_2}$ exist but exactly one of the rays $B^{\phi_1}$, $B^{\phi_2}$ does not not land at $\gamma(a_0)$;
\item[(B3)] $B^{\phi_1}$, $B^{\phi_2}$ exist but both rays $B^{\phi_1}$, $B^{\phi_2}$ do not land at $\gamma(a_0)$;
\end{itemize}

By Lemma \ref{lem:RayInBubbleExists} one of the rays $B^{\phi_1}$, $B^{\phi_2}$ does not
exist if and only if $a_0\in \BB^{2^n\phi_1}\cup \BB^{2^n\phi_2}$.
This describes \textbf{Case (B1)}.

Let us show that \textbf{Cases (B2), (B3), (C), and (D)} do not occur. The
rays $B^{\phi_1}$, $B^{\phi_2}$ are in one periodic orbit and so if
one of them lands at $\gamma(a_0)$, then so does another one. This
excludes Case (B2).

Assume $B^{\phi_1}$, $B^{\phi_2}$ do not land at $\gamma(a_0)$. Then by
Proposition \ref{prop:LandOfRaysAtRatPts} there is a ray $B^{\phi}$
landing at $\gamma(a_0)$. By Lemma \ref{lem:StabOfPerOrb} the ray
$B^{\phi}$ lands at $\gamma(a)$ in some neighborhood of $a_0$. This
contradicts the assumption that $a_0\in\partial \XX$ because in
$\XX$ the ray $B^\phi$ does not land at $\gamma(a)$. This excludes
Case (B3). Similarly, Case (D) is excluded: if $\gamma(a_0)$ is not transitive, then there is a new ray
$B^\phi$ landing at $\gamma(a_0)$ and, by Lemma \ref{lem:StabOfPerOrb}, the ray $B^\phi$ lands at $\gamma(a)$ in a neighborhood of $a_0$ which is
the contradiction of $a_0\in\partial \XX$.

 If $Y^B(\phi_1,\phi_2))$ exists but $-a_0\not\in\intr(Y^B(\phi_1,\phi_2))$, then $-a_0\in\partial
Y^B(\phi_1,\phi_2)$. Observe that $-a\not = \gamma(a)$ because
$\gamma(a)$ is periodic. Thus $-a_0\in B^{\phi_1}\cup B^{\phi_2}$,
but then these rays do not exist because they hit singularities. This exclude Case (C).

We have shown that \[\partial \XX\subset \cup_{n\ge
0}\BB^{2^n\phi_1}\cup_{n\ge 0}\BB^{2^n\phi_2}\cup \{m(h)\}\cup \mathcal I,\]
where $\mathcal I$ and $m(h)$ are defined in Case (A). By Lemma
\ref{lem:CorrTheoremForSigma} the set \[\left(\cup_{n\ge
0}\BB^{2^n\phi_1}\cup_{n\ge
0}\BB^{2^n\phi_2}\right)\setminus(\BB^{\phi_1}\cup \BB^{\phi_2}) \]
can intersect $\partial \XX$ only at $\basil$ Misiurewicz
parameters. Therefore:
\[\partial \XX\subset \BB^{\phi_1}\cup\BB^{\phi_2}\cup \{m(h)\}\cup\{\text{$\basil$ Misiurewicz parameters}\}\cup \mathcal I.\]
Also, by Lemma
\ref{lem:CorrTheoremForSigma}, we have $\BB^{\phi_1}\cup \BB^{\phi_2}\subset \partial\XX$. This implies that
the accumulation sets of $\BB^{\phi_1}$ and $\BB^{\phi_2}$ are in $\partial \XX$. Observe that the parameter rays
can not accumulate on points in hyperbolic components. We conclude that the accumulation sets of $\BB^{\phi_1}$ and $\BB^{\phi_2}$
are in $\mathcal L=\{m(h)\}\cup\{\text{$\basil$ Misiurewicz parameters}\}\cup \mathcal I$. Note that $\mathcal L$ is
totally disconnected as it is a union of the totally disconnected set $\mathcal{I}\subset\cup_{i=1}^n \partial\ovl {\HH_i^B}$
and a countable set that is disjoint from $\cup_{i=1}^n \partial\ovl {\HH_i^B}.$
We conclude that $\BB^{\phi_1}, \BB^{\phi_2}$ land at points in $\mathcal L$.
Recall that $m(h)\in\partial \XX$ and $m(h)\in \partial (\VV_2\setminus \ovl\XX)$. This can happen only if $\BB^{\phi_1}, \BB^{\phi_2}$
land at $m(h)$.

We obtain $\ovl\XX=\YY^B(\phi_1,\phi_2)$. Moreover, for $a\in\intr
(\YY^B(\phi_1,\phi_2))$ the rays
$B^{\phi_1}, B^{\phi_1}$ land together and $-a\in \intr(Y^B(\phi_1,\phi_2))$.

The proof in the \textbf{pre-periodic case} is similar to the periodic case,
so we will only explain the differences. Assume that $\phi_1,$
$\phi_2$ are pre-periodic with pre-period $s$. Let us use the induction on
$s$.

As above denote by $\XX$ the set of
parameters $a\in \VV_2$ such that the rays $B^{\phi_1}$,
$B^{\phi_2}$ land together in the dynamical plane of $g_a$, the
common landing point $\gamma(a)$ is transitive preperiodic repelling, and
$-a\in\intr(Y^B(\phi_1,\phi_2))$.

It is well known (follows from Theorem \ref{th:char_ray_pair}) that there are periodic angles
$\psi_1,\psi_2$ in the forward orbit of $\phi_1,\phi_2$ so that $\RR^{\psi_1}$, $\RR^{\psi_2}$ land together
and separate $\RR^{\phi_1}$, $\RR^{\phi_2}$ from $0$. We may assume that $0<\psi_1<$$\phi_1<\phi_2<$$\psi_2<0$
is the cyclic order. Further, if $\psi_1<2^k\phi_1<\psi_2$ for $s>k>0$, then $\psi_1<2^k\phi_2<\psi_2$, the rays $\RR^{2^k\phi_1}$, $\RR^{2^k\phi_2}$
land together, and, by the induction assumption, the rays  $\BB^{2^k\phi_1}$, $\BB^{2^k\phi_2}$ land together in $\VV_2$.
Denote by $\YY^B$ the closure of the connected component of
\[\VV_2\setminus\left(\bigcup_{\psi_1<2^k\phi_1<\psi_2}\left(\BB^{2^k\phi_1}\cup \BB^{2^k\phi_2}\right)
\cup \left(\BB^{\psi_2}\cup \BB^{\psi_1}\right)\right)\]
containing the accumulation set of $\BB^{\phi_1}$. It follows that $\XX\subset \YY^B$
and for all $a\in \intr(\YY^B)$ the puzzle $Y^B(\phi_1,\phi_2)$ exists.

 Consider $a_0\in\partial
\XX$; as in the periodic case we have Cases (A,1--2), (B,1--3), (C), (D).
Case (A) does not occur. Case (B1) occurs only if
$a_0\in (\BB^{\phi_1}\cup \BB^{\phi_2})\cap(\BB^{\psi_1}\cup
\BB^{\psi_2}) $. Cases (B2), (B3), and (D) do not occur because $Y^B(\phi_1,\phi_2)$ exists
in $\YY^B$. Case (C)
occurs if and only if $a_0\in (\BB^{\phi_1}\cup
\BB^{\phi_2})\setminus(\BB^{\psi_1}\cup \BB^{\psi_2})$ or $a_0$ is the Misiurewicz parameter
that is the mating of the Basilica polynomial and the common landing
point of $\RR^{\phi_1}$, $\RR^{\phi_2}$. In the last case $a_0$ is also on the boundary of $\VV_2\setminus \ovl\XX$.
\end{proof}

\subsection{More on small copies of the Mandelbrot set}
Recall that a small copy $\MM_\HH\subset\MM$ is primitive if $\HH$
is not attached to a hyperbolic component of smaller period.
Equivalently, the root of $\MM_\HH$ is not transitive. We need the
following well known fact about $\VV_1$.
\begin{proposition}
\label{prop:PrimSmallCopies} Let $\MM_\HH\subset\MM$ be a primitive
copy of the Mandelbrot set and let $p$ be the period of the renormalization for $\MM_\HH$. Then there is a sequence $\YY_n$ of
parapuzzle pieces bounded by transitive rays such that:
\begin{itemize}
\item $\YY_n\subset\intr(\YY_{n-1})$;
\item $\MM_\HH=\cap_n\YY_n$;
\item if $c\in \YY_{n-1}$, then $f_c^p:Y_{n-1} \to Y_{n}$ is
a quadratic like map.
\end{itemize}
\end{proposition}
\begin{proof}
The statement is well known so we will give the sketch of the proof.

Consider $c_0$ in $\MM_\HH$. Let $K_1,\dots, K_p$ be the cycle of small filled in Julia sets associated
with $\MM_\HH$; we assume that $c_0\in K_1$.

Let $\MM_{\HH'}\not=\MM_\HH$ be the smallest copy of $\MM$
containing $\MM_\HH$. If $\MM_\HH$ is maximal, then take $\MM_{\HH'}:=\MM$.
Let $\ovl{p}$ be the periodic cycle associated with $\HH'$;
i.e. $\ovl{p}$ is attracting if and only if the parameter is in $\HH'$ (if $\MM_{\HH'}=\MM$, then $\ovl{p}$ is the $\alpha$--fixed point).

In the dynamical plane of $f_{c_0}$ the cycle $\ovl{p}$ is repelling. Then there is a puzzle piece $Y_1$ bounded by rays landing at
points in $\ovl{p}$ or their preimages such that $K_1\subset Y_1$ but $K_i\cap\subset Y_1=\emptyset$ for $i\not=1$.

Define $Y_n\supset K_1$ to be the pullback of $Y_1$ under $f_{c_0}^{(n-1)p}$ along the orbit of $K_1$. The sequence
$\YY_n$ satisfies the conditions of the proposition.
\end{proof}

\begin{proposition}
\label{prop:PrimSmallCopiesInV2} Let $\YY_n$ be a sequence as in
Proposition \ref{prop:PrimSmallCopies} such that $\YY_n\not\subset
\WW_B$. Then for all $a\in \cap_n\YY^B_n$ all rational rays in bubbles land
in the dynamical plane of $g_a$.
\end{proposition}
\begin{proof}
By Proposition \ref{prop:IntransCase} the $\VV_2$--twins $\YY_{n}^B$
exist and $g_a^p:Y^B_{n-1}\to Y^B_n$ is quadratic like for $a\in
\YY^B_{n-1}$. Hence for $a_0\in\cap_n\YY^B_n$ the map $g_{a_0}$ has a
 cycle $K^B_1,\dots, K^B_p$ of connected small filled in Julia sets.

Observe that rays in bubbles do not intersect $\cup_i
K^B_i$. Indeed, $\cup_iK^B_i$ does not intersect the attracting basin
of the $\infty$--$0$ cycle so it is sufficient to verify that
$\cup_iK^B_i$ does not contain $\alpha(g_a)$ (the fixed point
$\ovl{E_\infty}\cap \ovl{E_0}$). This is so because $Y_n$ does not
contain the landing points of $R^{1/3}$, $R^{2/3}$.

The proposition follows from Lemma \ref{lem:LandOfRaysFor} because $\cup_i K^B_i \supset \ovl{\cup_{n\ge 0}g^n_a(-a)}$.

\end{proof}

\begin{lemma}
\label{lem:AccumulSet} If a rational non-$\basil$ ray $\RR^\phi$
lands in a parapuzzle piece $\YY$ and $\YY^B$ exists, then the
accumulation set of $\BB^\phi$ is in $\YY^B$.

If a rational non-$\basil$ ray $R^\phi$ lands in a puzzle piece $Y$
in the dynamical plane of $f_c$ and if $Y^B$ and $B^\phi$ exist in the
dynamical plane of $g_a$, then the accumulation set of $B^\phi$ is
in $Y^B$.
\end{lemma}
\begin{proof}
If $\RR^\phi$ lands in $\YY$, then either $\RR^\phi$ is in
$\partial\lfloor\YY\rfloor$ or $\RR^\phi$ lands in $\intr(\YY)$. In
the first case $\BB^\phi$ is in $\partial\lfloor\YY^B\rfloor$ and
the claim is obvious. In the second case there is a moment when
$\BB^\phi$ enters $\intr(\YY^B)$; after that $\BB^\phi$ always stays
in $\intr(\YY^B)$.

Similarly, the second (dynamical) part
of the proposition is verified.
\end{proof}

\subsection{The intransitive case}
\begin{proposition}
\label{prop:TransCase} Proposition \ref{prop:IntransCase} holds for
intransitive rays $R^{\phi_1},$ $R^{\phi_2}.$
\end{proposition}
\begin{proof}
The proof is similar to the proof of Proposition \ref{prop:IntransCase} so we will explain only the differences.

 Assume $\phi_1$, $\phi_2$ are periodic. Denote by
$\XX$ the set of parameters $a\in \VV_2$ such that the rays
$B^{\phi_1}$, $B^{\phi_2}$ land together in the dynamical plane of
$g_a$, the common landing point $\gamma(a)$ is
repelling, and $-a\in\intr(Y^B(\phi_1,\phi_2))$.
 It is easy to see that $\XX$ is open. Consider a point $a_0\in \partial \XX$. By definition, one of the
following holds:
\begin{itemize}
\item[(A)] $\gamma(a_0)$ is not a repelling point;
\item[(B)] $Y^B(\phi_1,\phi_2)$ does not exist;
\item[(C)] $Y^B(\phi_1,\phi_2)$ exists but $-a_0\not\in\intr(Y^B(\phi_1,\phi_2))$.
\end{itemize}
Furthermore, Case (A) is subdivided into Cases (A1), (A2) and Case (B) is subdivided into Cases (B1)--(B3); see the proof of Proposition \ref{prop:IntransCase}.

Cases (A1), (A2), (B1), (B3), and (C) are the same as in  Proposition \ref{prop:IntransCase}.
Let us now show that Case (B2) does not occur.
 Recall that we consider the parameter $a_0\in\partial \XX$;
without loss of generality we may assume that $B^{\phi_1}$ does not
land at $\gamma(a_0)$. It follows that $B^{\phi_1}$ does not land at
any point because landing is stable under small perturbations (Lemma
\ref{lem:StabOfPerOrb}).

Note that $a_0$ is
in the accumulation set of a ray $\BB^{2^n\phi_1}$ for some $n\ge
0$; otherwise $\ovl{B^{\phi_1}}$ would be stable by Lemma \ref{lem:StabOfRaysFromLambdaLemma}.
Choose a parapuzzle piece $\YY_t$ bounded by transitive rays such that $\YY(\phi_1,\phi_2)\subset\YY_t$
and $\YY_t$ does not contain rays in $\cup_{n\ge0}\RR^{2^n\phi_1}\setminus \{\RR^{\phi_1}\}$.

 By Proposition
\ref{prop:IntransCase} the $\VV_2$--twin $\YY^B_t$ exists
and, moreover, $\XX\subset \YY_t^B$. By Lemma \ref{lem:AccumulSet} the parapuzzle piece $\YY^B_t$
does not contain accumulation set of $\BB^{2^n\phi_1}\not=\BB^{\phi_1}$. Therefore, $a_0$ is in
$\ovl{\BB^{\phi_1}}$.

Denote by $h\in \VV_1$ the landing point of $\RR^{\phi_1}$. As
$R^{\phi_1}$ is intransitive $h$ is the root of a small primitive copy
$\MM_\HH$ of the Mandelbrot set. Let $\YY_n$ be a sequence as in
Proposition \ref{prop:PrimSmallCopies} for $\MM_\HH$. Then $a_0\in
\cap_n\YY^B_n$. By Proposition \ref{prop:PrimSmallCopiesInV2} the
ray $B^{\phi_1}$ lands in the dynamical plane of $g_{a_0}$; this
contradicts the assumption on $a_0$.

\end{proof}

It follows from Propositions \ref{prop:IntransCase} and
\ref{prop:TransCase} that:
\begin{cor}
\label{cor:V1TwinsExists} For a parapuzzle piece $\YY$ in
$\VV_1\setminus\WW_B$ the $\VV_2$--twin $\YY^B$ exists and $a\in
\intr(\YY^B)$ if and only if $Y^B$ exists in the dynamical plane of $g_a$
and $a\in \intr(Y^B)$.
\end{cor}

\subsection{Triviality of fibers at Misiurewicz parameters}

\begin{lemma}
\label{lem:TrivOfFibersAtMis} If $a\in\MM_2$ is a non-$\basil$ Misiurewicz
parameter, then there is a sequence $\YY^B_n$ of parameter bubble
pieces such that $a\in\intr(\YY^B_n)$ and $\cap\YY^B_n=\{a\}$.

If $a\in\MM_2$ is a $\basil$ Misiurewicz
parameter, then there are sequences $\YY^B_n$ and $\YY'^B_n$ of parameter bubble
pieces and two components $\EE$, $\EE'$ of $\VV_2\setminus\MM_2$ such that $a\in\intr(\YY^B_n\cup \YY'^B_n\cup\EE\cup\EE')$
and $\cap\YY^B_n=\cap\YY'^B_n=\{a\}$.
\end{lemma}
\begin{proof}
Assume that $g_a$ is a non-$\basil$ Misiurewicz parameters; we write $a=m(c)$
as in Definition \ref{defn:Map_cTo_ac}.

By definition, in the dynamical plane of $f_c$  the critical value
$c$ maps after $k$ iterations to a periodic point, say $p$ of period
$q$. Let us choose a puzzle piece $Y_0$ containing $p$ and not
containing any other point of the post-critical set of $f_c$. Then
$f_c^{k+nq}(c)=p\in Y_0$; define $Y_n\ni c$ to be the pullback of
$Y_0$ under $f_c^{k+nq}$ along the orbit of $c$. We get $c\in
\intr(\YY_n)$ and $\cap_n\YY_n=\{c\}$ by triviality
of fibers for  Misiurewicz parameters in $\VV_1$.

By Corollary \ref{cor:V1TwinsExists} the pieces $\YY_n^B$, $Y_n^B$,
$Y_0^B$ exist and $a\in\intr(\YY^B_n)$. We need to show that $\cap_n \YY_n^B=\{a\}$.

Assume that $a'$ is in $\cap_n \YY^B_n$. Then in the dynamical plane
of $g_{a'}$ the critical value $-a'$ is in $Y^B_n$. Moreover,
$g_{a'}^{k+nq}(-a')\in Y^B_0$ for all $n$. We claim that $g_{a'}^{k+nq}(-a')=g_{a'}^{k+(n-1)q}(-a')=p^B$ is a periodic point. Indeed, $g_{a'}^q:g_{a'}^k(Y^B_{n+1})\to
g_{a'}^k(Y^B_{n})$ has degree $1$ because it is so for
$f_c^q:f_c^k(Y_{n+1})\to f_c^k(Y_{n})$. Further,
$g_{a'}^k(Y^B_{n+1})\subset g_{a'}^k(Y^B_{n})\subset
g_{a'}^k(Y^B_{0})$, so there is exactly one (periodic) point, call it $p^B$, that does not
escape from $Y^B_{n+1}$ under iteration of $g_{a'}^q$.
 Therefore, $a'=a$ as there is only one Misiurewicz parameter with this property (Thurston's rigidity theorem).

If $a\in\VV_2$ is a $\basil$ Misiurewicz parameter, then there are two $\basil$ Misiurewicz parameters
$c$, $c'$ in $\VV_1$ such that $a=m(c)=m(c')$. Let $\YY_n$ (resp.\ $\YY'_n$) be a sequence parapuzzle pieces around $c$ (resp.\ around $c'$)
as above. Then bubble parapuzzle pieces $\YY^B_n$, $\YY'^B_n$ contain $a$ on the boundaries. The same argument as in the non-$\basil$ case shows
that if $a'\in \cap\YY^B_n$ or $a'\in \cap\YY'^B_n$, then $a'=a$. Note that $a$ is on the boundary of exactly two components of $\VV_2\setminus\MM_2$,
denoted by $\EE$, $\EE'$. It is clear that $a\in\intr(\YY^B_n\cup \YY'^B_n\cup\EE\cup\EE')$.
\end{proof}

\subsection{Proof of Theorem \ref{thm:ParBubblRay}}
\label{subsect:ProofOfLandThm} We have showed that if $\RR^{\phi_1}$,
$\RR^{\phi_2}$ land together at $c$ and $\phi_1\not=\phi_2$, then the rays
$\BB^{\phi_1}$, $\BB^{\phi_2}$ land together at $m(c)$.

Assume now that a non-$\basil$ ray $\RR^{\phi}$ lands alone at $\gamma'$. Then, by Lemma
\ref{lem:TrivOfFibersAtMis}, the ray $\BB^{\phi}$ lands.  We need to show that $\BB^{\phi}$ lands alone at
$m(\gamma')=f_{\gamma'}\mcup f_B$.

 By Corollary \ref{cor:V1TwinsExists} the parameter $\gamma'$ is in a parapuzzle piece $\YY\subset \VV_1\setminus\WW_B$
 if and only if $a$ is in
$\YY^B$. By Lemma \ref{lem:TrivOfFibersAtMis}
parapuzzle pieces around the parameter $a$ shrink to $\{a\}$. Combining with Lemma \ref{lem:AccumulSet} we get $a=m(\gamma')$.

Suppose that $a$ is also the landing point of
a rational ray $\BB^\psi$. Denote by $\gamma''$ the landing points of
$\RR^\psi$ respectively. The same argument as above shows that $a=m(\gamma'')$. Then $\gamma'=\gamma''$ because
$m$ is injective on non-$\basil$ Misiurewicz parameters.

We have shown that all parameter non-$\basil$ rational rays land. By Lemma
\ref{lem:StabOfRaysFromLambdaLemma} the dynamic ray $\ovl{B^{\phi}}$
is stable unless the parameter is on $\cup_n\ovl{\BB^{2^n\phi}}$. We
conclude that $B^{\phi}$ lands unless $a\in B^{2^n\phi}$ as it is so for at least one
parameter in every component of $\VV_2\setminus
\cup_n\ovl{\BB^{2^n\phi}}$.

\subsection{The map $c\to m(c)$}

Recall that the map $c\to m(c)$ in Definition \ref{defn:Map_cTo_ac} is defined
for the parameters on the closures of hyperbolic components and for Misiurewicz parameters.
\begin{proposition}
\label{prop:V2ParabBirf}The equality $m(c)=m(c')$ holds if and only if $c=c'$ or
$c$, $c'$ are $\basil$ Misiurewicz parameters such that the external angles of $c$, $c'$ are identified by $L_B$.
\end{proposition}
\begin{proof}
The Misiurewicz case is well known. The statement is obvious if $c$ or $c'$ is in a hyperbolic component.
Further, if $c$ is parabolic, resp.\ indifferent (Siegel or Cremer parameter), then so is $m(c)$. Therefore, if $c$ is indifferent and
$c'$ is parabolic, then $m(c)\not=m(c')$.

The parabolic case, i.e. $c$, $c'$ are parabolic parameters, follows from Proposition \ref{prop:ParabolBirf}. Assume $c$, $c'$
are indifferent parameters. Then there are rays $R^\phi$, $R^\psi$, and a hyperbolic component $\HH$ such that
$R^\phi$, $R^\psi$ land at $\ovl\HH$ and $\ovl{R^\phi}\cup \ovl{R^\psi}\cup \HH$ separate $c,$ $c'$. By Theorem \ref{thm:ParBubblRay}
the set $\ovl{B^\phi}\cup \ovl{B^\psi}\cup \HH^B$ separate $m(c)$ and $m(c')$.
\end{proof}

\section{Proof of Theorem \ref{th:main2}}
\label{sec:ProofsOfMainThms}

\subsection{Small copies of the Mandelbrot set in $\VV_2$}
\label{subsec:DecorTilV2}

Let us now describe small copies of the Mandelbrot set in $\VV_2$.
Let $\MM_\HH$ be a small copy in $\MM\setminus\WW_B$. Consider the first
renormalization map $f_c^p:\YY^{p}_{\HH}\to \WW^{0}_{\HH}$. By
Theorem \ref{thm:ParBubblRay} the pieces
$\YY^{B}_{\HH},\WW^{B}_{\HH}$ exists and for $a\in \intr
(\WW^{B}_{\HH})$; we have $g_a^p:\YY^{B}_{\HH}\to \WW^{B}_{\HH}$.

Denote by $\BB^{\phi_1},$ $\BB^{\phi_2}$ the rays that form $\partial\WW^B_{\HH}$; we assume that $0<\phi_1<\phi_2<0$
is the cyclic order. As $\phi_1,\phi_2$ are periodic there are two
possibilities for $\BB^{\phi_1}$ (and for $\BB^{\phi_2}$) either it
is an infinite ray in bubbles or a ray within $\EE_\infty$.

\begin{lemma}
For the map $g_a^p:\YY^{B}_{\HH}\to \WW^{B}_{\HH}$ as above the ray
$\BB^{\phi_1}$ (resp.\ $\BB^{\phi_2}$) is infinite if and only if
$\MM_\HH$ is bounded by $\LL_B$ from the right (resp.\ from the
left).
\end{lemma}
\begin{proof}
 The ray $\rho(B^{\phi_1})$ (resp.\
$\rho(B^{\phi_2})$) in the dynamical plane of the Basilica is finite
if and only if $R^{-\phi_1}$ (resp.\ $R^{-\phi_2}$) land at the
closure of the (periodic) Fatou component $F_0$ containing $0$ (note
that $\rho(B^{\phi_1})$ and $\rho(B^{\phi_2})$ can not land at the
Fatou component containing $-1$ as $\HH\not\subset \WW_B$). This
is equivalent to the landing of
 $R^{\phi_1}$ (resp.\ $R^{\phi_2}$) at
$\ovl{F_0}$. The last is equivalent to the non-boundedness of $\MM_\HH$
from the right (resp.\ from the left).
\end{proof}

It follows from Proposition \ref{prop:bound2} that at least one of $\BB^{\phi_1},$ $\BB^{\phi_2}$ is an infinite ray and
if $B^{\phi_1}$
or $B^{\phi_2}$ is a finite ray, then $\MM_\HH$ is a maximal
primitive copy.

\begin{definition}[Small copies: the first case] If $\MM_\HH$ is bounded by $\LL_B$ from both
sides, then the \emph{$\VV_2$--twin $\MM^B_\HH$} of $\MM_\HH$ is the
closure of the set of parameters $a\in \intr(\YY^B_\HH)$ such that
in the dynamical plane the critical value $-a$ does not escape from
$Y^B_\HH$ under iteration of $g_a^p:Y^B_\HH\to W^B_\HH$.
\end{definition}
As in the quadratic case perturbing the map $g_a^p:Y^B_\HH\to
W^B_\HH$ (thickening $Y^B_\HH, W^B_\HH$) we always can obtain a
quadratic-like map. Therefore, by the straightening theorem
$\MM^B_\HH$ is canonically homeomorphic to $\MM$ as well as to $\MM_\HH$.

We will refer to $g_a^p:Y^B_\HH\to W^B_\HH$ as \textit{the first
renormalization map}. As in $\VV_1$ for $a\in
\intr(\YY^B_{(n-1)\HH})$ we define $Y^B_{n\HH}$ to be the preimage
of $Y^B_{(n-1)\HH}$ under $g_a^p:Y^B_\HH\to W^B_\HH$; we will refer
to $g_a^p:Y^B_{n\HH}\to Y^B_{(n-1)\HH}$ as the \textit{$n$--th
renormalization map}. Let us note that if $\MM_\HH$ is bounded by
$\LL_B$ from both sides, then $Y^B_\HH$ (and so $Y^B_{n\HH}$ for
$n\ge 1$) does not intersect the periodic bubble $E_\infty$.
However, if $\MM_\HH$ is bounded by $\LL_B$ from only one side, then
all $\YY^B_{(n\HH)}$ have $\infty$ on the boundaries. We modify the
map $g_a^p:Y^B_\HH\to W^B_\HH$ as follows. Let $\dot W_\HH^B$ be
$W_\HH^B$ without an open neighborhood of $\infty$ bounded by an
internal equipotential in $E_\infty$ and let $\dot Y_\HH^B$ be the pre-image of
$\dot W_\HH^B$ under $g_a^p:Y^B_\HH\to W^B_\HH$ for $a\in\intr(\dot\WW_\HH)$.
\begin{definition}[Small copies: the second case]
If $\MM_\HH$ is bounded by $\LL_B$ from exactly one side, then the
\emph{$\VV_2$--twin} $\MM^B_\HH$ of $\MM_\HH$ is the closure of the set of
parameters $a\in \intr(\dot\YY^B_\HH)$ such that in the dynamical plane
the critical value $-a$ does not escape from $\dot Y^B_\HH$ under
iteration of $g_a^p:\dot Y^B_\HH\to \dot W^B_\HH$.
\end{definition}
Again, the straightening theorem implies that $\MM^B_\HH$ is
canonically homeomorphic to $\MM$ as well as to $\MM_\HH$. We will refer to $g_a^p:\dot
Y^B_\HH\to \dot W^B_\HH$ as \textit{the first renormalization map} (of $\MM^B_\HH$).

If a parameter $a\in\MM_2$ is not in any $\MM^B_\HH$ and is not in the closure of the main hyperbolic component of
$\MM_2$, then the critical tableau used in
\cite{AY} is not periodic, i.e. $g_a$ is not renormalizable, and by
\cite[Theorem 9.1]{AY} the intersection of parapuzzle pieces
containing the parameter $a$ is $\{a\}$. In other words, triviality of the fiber holds at $a$.

\begin{definition}[The $\VV_2$--twins of renormalizable
parameters] \label{defn:V2TwinsOfRenormPar} The $\VV_2$--twin of
$c\in\MM_\HH$ is $m(c)=\chi_\HH(c)$, where $\chi_\HH:\MM_\HH\to
\MM^B_\HH$ is the canonical homeomorphism.
\end{definition}

For a non-parabolic non-Misiurewicz parameter $a\in\MM_\HH^B$ define $T$ to be the set of pre-images of
$\ovl{\WW_\HH^B\setminus \YY_\HH^B}$ (resp.\ of
$\ovl{\dot\WW_\HH^B\setminus \dot\YY_\HH^B}$) under iteration of the
first renormalization map. The \textit{decoration tiling} $\TT$ of
$\MM^B_\HH$ is the parameter counterpart to $T$; see Figure \ref{figure:bubble}.

\begin{figure}%[hb]
\includegraphics[width=11cm]{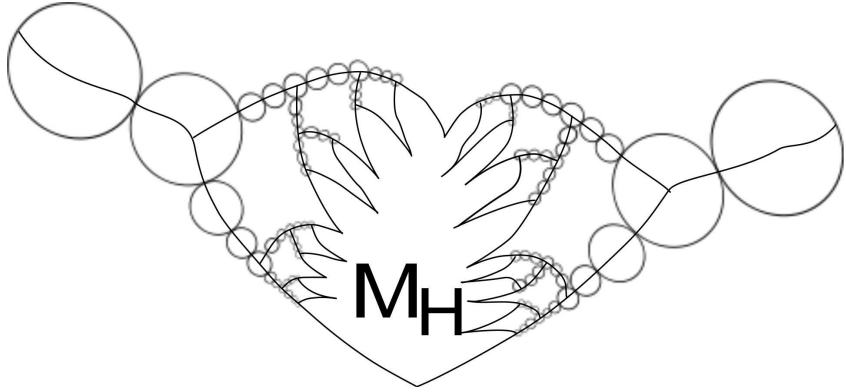}
\caption{A decoration tiling in $\VV_2$ (compare to Figure
\ref{figure:renormal}).
  }
  \label{figure:bubble}
\end{figure}

\begin{proposition}[The decoration theorem in $\VV_2$]
\label{prop:DecorTilV2} For a small copy $\MM_\HH^B\subset \MM_2$ of
the Mandelbrot set the diameter of pieces in the decoration tiling
$\TT$ of $\MM_\HH$ tends to $0$.
\end{proposition}

\begin{proof}
We will consider the case when $\WW_\HH^B$ is bounded by infinite
rays; if one of two rays that bound $\WW_\HH^B$ is finite, then one
need to modify Figure \ref{figure:bubble2} accordingly; all other
arguments are the same.

Let $g_a^p:Y^B_{2\HH}\to Y^B_{\HH}$ be the second renormalization map
of $\MM_\HH^B$. It is sufficient to prove the proposition for pieces
in $\YY^B_\HH$.

Let us construct a tubing for $g_a^p:Y^B_{2\HH}\to Y^B_{\HH}$, where
$a\in \intr(Y^B_{2\HH})$; see Figure \ref{figure:bubble2} and its
caption.

 Denote by $B^{\phi_1},B^{\phi_2}$ the periodic and by
$B^{\psi_1},B^{\psi_2}$ the pre-periodic rays that bounds $Y^B_\HH$;
we assume that $0<\phi_1<\psi_1<\psi_2<\phi_2<0$ is the cyclic
order. As $B^{\phi_1},B^{\phi_2}$, $B^{\psi_1},B^{\psi_2}$ are infinite, the rays in the pair $B^{\phi_1}$,
$B^{\psi_1}$ as well as in the pair $B^{\phi_2}$, $B^{\psi_2}$ are
coincide till $Y^B_\HH$. Choose $a_0\in \intr(\YY^B_{2\HH})$ and
choose in the dynamical plane of $g_{a_0}$ points $x\in
B^{\phi_1}\cap B^{\psi_1}$, $y\in B^{\phi_2}\cap B^{\psi_2}$ that
are sufficiently close to $Y^B_\HH$.

For an $a\in \intr{\YY^B_\HH}$ define $x(a)$, $y(a)$ to be the
images of $x(a_0)$, $y(a_0)$ under $h_a^{-1}\circ h_{a_0}$.

 In the dynamical plane of $g_a$ choose a smooth Jordan curve
$\gamma(a)$ such that:

\begin{itemize}
\item $\gamma(a)$ surrounds $Y^B_\HH$ and is sufficiently close to $\partial Y^B_\HH$;
\item $\gamma(a)\cap(B^{\phi_1}\cap B^{\psi_1}) =\{x\}$ and $\gamma(a)\cap(B^{\phi_2}\cap B^{\psi_2})
=\{y\}$;
\item the pre-image $\gamma_2(a)$
of $\gamma(a)$ along $g_a^p:Y^B_\HH\to Y^B_{2\HH}$ is inside the
open disc surrounded by $\gamma(a)$ and containing $Y^B_\HH$; and
\item $\gamma(a)$ depends smoothly on $a$.
\end{itemize}

Choose a smooth embedding $T$ of $Q_R\times \intr(\YY^B_\HH)$, where
$R>1$, into $\C\times \intr(\YY_\HH^B)$ over $\intr(\YY_\HH^B)$ such
that:
\begin{itemize}
\item $T_a(Q_R)$, is the closed annulus bounded by
$\gamma(a)$ and $\gamma_2(a)$;
\item for all $z$ in
\begin{equation}
\label{eq:PropDecorTil11}T_a(Q_R)\cap (B^{\phi_1}\cup B^{\psi_1}\cup
B^{\phi_2}\cup B^{\psi_2})\end{equation} we have
\begin{equation}
\label{eq:PropDecorTil2} T_{a'}\circ T^{-1}_{a''}=h_{a'}^{-1} \circ
h_{a''};
\end{equation}
i.e. $T_a$ and $h_a$ agree on \eqref{eq:PropDecorTil11}.
\end{itemize}

By the straightening theorem we get the map \[\chi=\chi(T):\intr(\YY^B_{2\HH})\to \VV_1.\]

\begin{figure}%[hb]
\includegraphics[width=12cm]{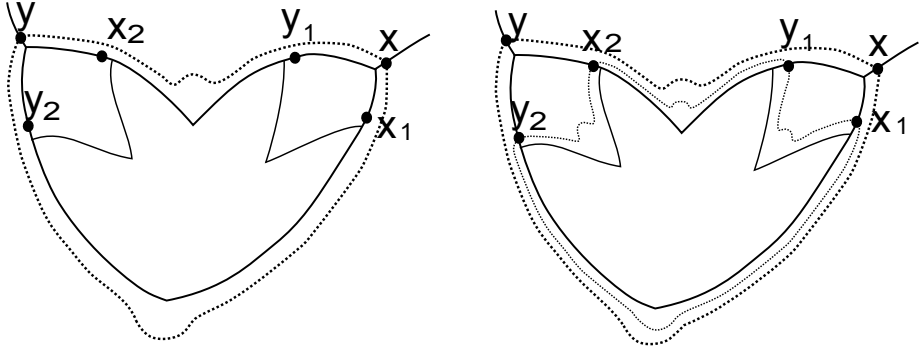}
\caption{Illustration to the construction of a tubing in the proof
of Proposition \ref{prop:DecorTilV2}. The copy $\MM^B_\HH$ is
encoded by the second renormalization map $g_a^p:Y^B_{2\HH}\to
Y^B_\HH$. The rays in bubbles that bound $ Y^B_{2\HH}, Y^B_\HH$ are
drawn.
 Left: a curve $\gamma=\gamma(a)$ surrounds $Y^B_\HH$ and
intersects the rays  in bubbles at $x$ and $y$. Right: the pre-image
of $\gamma$ under $g_a^p:Y^B_{2\HH}\to Y^B_\HH$ surrounds
$Y^B_{2\HH}$ and intersects the rays in bubbles at
$x_1,x_2,y_1,y_2$, where $g_a^p(x_1)=g_a^p(x_2)=x$ and
$g_a^p(y_1)=g_a^p(y_2)=y$. We defined $T_a(Q_R)$ to be the closed
annulus bounded by $\gamma(a)$ and $\gamma_2(a)$.
  }
  \label{figure:bubble2}
\end{figure}

We will now show that $\chi(\TT)$ is a generalized decoration tiling
of $\chi(\MM^B_{\HH})=\MM$. Note that $\chi$ may not be a
homeomorphism, but it is a homeomorphism on $\MM^B_{\HH'}$ and on
$\partial \YY^B_{n\HH'}$ as we will show.

 Define: \[\XX^{1}=\chi(\YY^B_{\HH}),\ \ \XX^{2}=\chi(\YY^B_{2\HH}),\]
\[X^{1}(\chi(a))=\chi(a)(Y^B_{\HH}(a)),\ \ X^{2}(\chi(a))=\chi(a)(Y^B_{2\HH}(a)).\]
We need to show that $f_c:X^{1}(c)\to X^{2}(c)$ for $c=\chi(a)\in\intr(\XX^2)$,
$X^{1}(c)$, $ X^{2}(c)$ are generalized puzzle pieces, and
$X^{1}(c),X^{2}(c)$ are the dynamical counterparts of $\XX^{1},
\XX^{2}$.

The first claim follows directly from the
construction of the straightening map because $g_a^p$ and $f_{\chi(a)}$ are
quasiconformally conjugate in neighborhoods of their filled in Julia sets, Subsection \ref{subset:QuadrLike}.
This also implies that the forward orbit of $\partial X^{1}(c)$, $\partial X^{2}(c)$ does not intersect
the interiors of $X^{1}(c)$, $X^{2}(c)$. Because $X^{1}(c)$, $X^{2}(c)$ contain the critical value, markings for
$X^{1}(c)$, $X^{2}(c)$ are irrelevant. Let us now show that
$X^{1}(c')$, $X^{2}(c')$ are combinatorially equivalent to
$X^{1}(c'')$, $X^{2}(c'')$.

By
\eqref{eq:PropDecorTil2} and \eqref{eq:TubBottch} we have:
\[\chi(a')\circ (h_{a'}^{-1}\circ
h_{a''})=\chi(a')\circ T_{a'}\circ
T_{a''}^{-1}=B_{\chi(a')}^{-1}\circ B_{\chi(a'')}\circ \chi(a'')\]
or
\begin{equation}
\label{eq:PropDecorTilV2} \chi(a')\circ (h_{a'}^{-1}\circ
h_{a''})=B_{\chi(a')}^{-1}\circ B_{\chi(a'')}\circ \chi(a'')
\end{equation}
on $(B^{\phi_1}\cup B^{\phi_2}\cup B^{\psi_1}\cup B^{\psi_2})\cap T_{a''}(Q_R)$. As $g_a^{p}:\partial\YY^B_{2\HH}\to \partial\YY^B_{\HH}$ is a
covering and as all points in \[\partial Y^B_{\HH}\cup \partial Y^B_{2\HH}\setminus \{\text{the landing points of rays}\}\]
are preimages of $(B^{\phi_1}\cup B^{\phi_2}\cup B^{\psi_1}\cup B^{\psi_2})\cap T_{a''}(Q_R)$, taking pullbacks of
\eqref{eq:PropDecorTilV2} under iteration of
$g_a^{p}:\partial\YY^B_{2\HH}\to \partial\YY^B_{\HH}$ we obtain that
\eqref{eq:PropDecorTilV2} holds on $\partial Y^B_{\HH}(a'')\cup \partial Y^B_{2\HH}(a'')\setminus J_{a''}$.
Therefore, $B_{\chi(a')}^{-1}\circ B_{\chi(a'')}$ maps \[\partial X^1(\chi(a''))\cup\partial  X^2(\chi(a''))=\chi(a'')\left(\partial Y^B_{\HH}(a'')\cup \partial Y^B_{2\HH}(a'')\right)\]
homeomorphically to $\partial X^1(\chi(a'))\cup\partial  X^2(\chi(a'))$ except at finitely many points in the Julia sets that are the landing points of rays.

This
shows that $X_1(c''), X_2(c'')$ are combinatorially equivalent
to $X_1(c'),$ $X_2(c')$. Similarly, $\XX_1$, $\XX_2$ are the parameter counterparts to $X_1(c''), X_2(c'')$.
We conclude that $\chi(\TT)$
is a generalized decoration tiling.

Assume now that the diameter of pieces in $\TT$ does not tend to
$0$. Then there is an infinite sequence $\TT_1,\TT_2,\dots $ of
pieces in $\TT$ and there are points $a_1\not=a_2$ such that $a_1,a_2$ are the
limiting points of $\{T_i\}$. We claim that $a_1,a_2\in
\MM_{\HH}^B$. Indeed, if, say, $a_1\not\in \MM_{\HH}^B$, then in the
dynamical plane of $g_{a_1}$ the critical value $-a_1$ escapes from
$Y^B_{2\HH}$ under iteration of $g_{a_1}^p:Y^B_{2\HH}\to Y^B_{\HH}$, say in $N$ iterations. Then
for all $a$ close to $a_1$ the critical value $-a$ escapes from
$Y^B_{2\HH}$ under iteration of $g_{a}^p:Y^B_{2\HH}\to Y^B_{\HH}$ in less or equal than $N+2$ iterations.
Observe that there are only finitely many pieces in $\TT$ with bounded by $N+2$ escaping time. Whence $a_1$ (and $a_2$) is in
$\MM_{\HH}^B$.

As $\chi$ is a homeomorphism on $\MM_\HH^B$, it follows that
$\chi(a_1)\not=\chi(a_2)$. As $\chi$ is continuous the sequence
$\{\chi(T_i)\}$ has two limiting points $\chi(a_1)$ and $\chi(a_2)$. This
contradicts Theorem \ref{th:decor2} that says that the diameter of
pieces in $\TT$ tends to $0$.
\end{proof}

\subsection{No ghost limbs in $\VV_2$}

No ghost limb theorem in $\VV_1$ says that if a parameter $c$ is in
$\WW_\HH\cap \MM$, then either $c\in\ovl \HH$ or $c$ is in a limb
attached to $\ovl\HH$.

The next proposition is a $\VV_2$--analogy of no ghost limb theorem in $\VV_1$;
for our convenience, \textit{we say that $\WW_\HH$ is $\VV_2$ if $\HH$ is the main hyperbolic component of $\MM_2$.}

\begin{proposition}
\label{prop:NoGhostLimsV2}
 If $a\in \intr(\WW^B_\HH)$, then either
$a\in\ovl {\HH^B}$ or $a\in\WW^B_{\HH'}$ for a component $H'^B\in
\WW^B_\HH$ attached to $\HH^B$.
\end{proposition}
\begin{proof}
Consider a periodic cycle $\ovl p=\{ p_1,\dots,p_n\}$ associated
with $\HH^B$. If $a\in \intr(\WW^B_\HH)$ but $a\not\in\ovl{\HH^B}$, then
$\ovl p$ is a repelling cycle and either there are periodic rays in
bubbles landing at $\ovl p$ or $a$ is in a parameter ray in bubbles with a periodic angle.

If $a$ is in a parameter ray, then the
statement is obvious.

If there are periodic rays landing at $\ovl p$, then choose two rays
$B^{\phi_1},$ $B^{\phi_2}$ landing at $\ovl p$ and separating $-a$
from all other rays landing at $\ovl p$. Then $a$ is in the interior
of the parapuzzle piece $\YY^B(\phi_1,\phi_2)$ bounded by
$\BB^{\phi_1},$ $\BB^{\phi_2}$. It follows from Proposition \ref{prop:V2ParabBirf} that $\YY^B(\phi_1,\phi_2)$ is $\WW^B_{\HH'}$ for some $\HH'^B\in \WW^B_\HH$
attached to $\HH^B$.
\end{proof}

\subsection{Triviality of fibers for hyperbolic components}
Recall that by a \textit{combinatorial disc} in $\VV_2$ we mean a closed
topological disc $\DD$ such that $\partial \DD$ is in a finite union of
parameter rational rays in bubbles and simple curves in the closures
of hyperbolic components. A sequence $a_n$\textit{ tends to $a_\infty$ combinatorially} if for
every combinatorial disc $\DD$ such that $a_\infty\in \intr(\DD)$ all but finitely many
$a_n$ are in $\DD$.

Later we will prove Yoccoz's results for $\VV_2$ which say that a sequence $a_n$ tends combinatorially to $a_\infty$
if and only if $a_n$ tends to $a_\infty$ in the usual topology (Theorem
\ref{YoccResInV2}).

\begin{proposition}
\label{prop:TrivOfFibersAtParabMis} Theorem \ref{YoccResInV2} holds
for the closures of hyperbolic components in $\MM_2$.
\end{proposition}
\begin{proof}
We will use ideas from \cite{Sch}. The claim is obvious for the
interiors of hyperbolic components.

 Let $a$ be a
parameter on $\partial \HH^B\subset \MM_2$. We will construct a sequence $\{\FF_n^B\}$ of
nested combinatorial topological discs such that
$a\in\intr(\FF_n)$ and $\cap_n \FF_n=a$. This will imply the claim of the proposition.

 \textbf{Indifferent case.} Suppose that $a$ is an indifferent (non--parabolic) point point on $\partial \HH$. Choose monotone sequences $\{b_n\}$, $\{c_n\}$ on $\partial \HH$ (recall that
$\partial \HH$ is a Jordan curve) such that $b_n,c_n$ are parabolic points and
 $\{b_n\}$, $\{c_n\}$ converge from different sides to $a$. Choose periodic rays $\BB^{\phi_n}$, $\BB^{\psi_n}$ landing at $b_n$, $c_n$ respectively and a simple curve
 $\gamma_n$ in $\ovl {\HH^{B}}$ connecting $b_n$, $c_n$ such that $\gamma_n$ tends to $a$, does not contain $a$, and $\gamma_{n-1}$ and $a$
 are in different components of $\ovl {\HH^{B}}\setminus\gamma_{n}$; i.e. $\gamma_n$ is a monotone sequence. We define
 $\FF_n$ to be the closed topological disc bounded by $\BB^{\phi_n}\cup \BB^{\psi_n}\cup \gamma_n$ and containing $a$. By definition, $\FF_n$ is
 a combinatorial neighborhood of $a$. It follows from the construction that $(\cap_n \FF_n)\cap \ovl{\HH^B}=a$.

Assume that $a'\in \cap_n \FF_n $. We claim that $a'=a$. Indeed,
$a'$ can not be in $\WW^B_{\HH'}$ for a component $H'^B\in
\WW^B_\HH$ attached to $\HH^B$ because $\WW^B_{\HH'}$ intersects
only finitely many sets in $\FF_n$. By Proposition
\ref{prop:NoGhostLimsV2} we have $a'\in \HH^B$ and so $a=a'$.

\textbf{Parabolic non-primitive case.} Suppose that $a\in
\ovl{\HH^B}\cap\ovl{\HH'^B}$, where $\HH'^B$ is a hyperbolic
component attached to $\HH^B$. Choose four monotone sequences
$\{b_n\}$, $\{c_n\}$, $\{d_n\}$, $\{f_n\}$ on $\partial \HH^B\cup
\partial \HH'^B$ consisting of parabolic points and converging to
$a$ from different sides of $\partial\HH^B\cup \partial \HH'^B$.
Choose simple curves $\beta_n$, $\gamma_n$ in $\ovl{\HH^B}$,
$\ovl{\HH'^B}$ respectively connecting pairs of points in $b_n$, $c_n$,
$d_n$, $f_n$ such that $\beta_n$, $\gamma_n$ converge to $a$, do not
contain $a$, and $a$, $\beta_{n-1},$ $\gamma_{n-1}$ are in different
components of
$\ovl{\HH^B}\cup\ovl{\HH'^B}\setminus(\beta_n\cup \gamma_n)$. We define
 $\FF_n\ni a$ to be the closed topological disc bounded by periodic rays landing at $b_n$, $c_n$, $d_n$, $f_n$ and by $\beta_n$, $\gamma_n$.
 As in the indifferent case $a\in\intr{F_n}$ and $\cap F_n=\{a\}$.

\textbf{Parabolic primitive case.} Suppose that $\HH^B$ is a primitive (i.e. it is not attached to a hyperbolic component of smaller period), $\HH^B$
is not the main hyperbolic component of $\MM_2$, and $a$ is the root of $\HH^B$. Let $\{\BB^{\phi_\infty}$, $\BB^{\psi_\infty}\}$ be the pair of periodic rays landing at $a$; we assume that $0<\phi_\infty<\psi_\infty<0$
 is the cyclic order. Choose a sequence of pairs $\{\BB^{\phi'_n}$, $\BB^{\psi'_n}\}$ landing together rational rays such that $\phi'_n\to \phi_\infty$ and
 $\psi'_n\to \psi_\infty$ and $0<\phi'_n<\phi_\infty<\psi_\infty<\psi'_n<0$ is the cyclic order, this sequence exists by Theorem \ref{thm:ParBubblRay} and the well known fact that the same statement is true for $\VV_1$.

Let us remind that $\BB^{\phi_\infty}$ or $\BB^{\psi_\infty}$ is possibly a finite ray in $\EE_\infty$. Denote by $\UU_n$ the open disc
in $\EE_\infty$ bounded by the internal equipotential of height $1/2^{n}$; we have $\UU_i\subset \UU_{i+1}$ and $\EE_\infty=\cup_n \UU_n$.
  Choose sequences $\{\BB^{\phi_n}\}$, $\{\BB^{\psi_n}\}$, $\{\gamma_n\}$ as in the indifferent case. Define
 $\widetilde{\FF_n}$ to be the closed topological disc bounded by $\BB^{\phi_n}\cup \BB^{\psi_n}\cup \gamma_n\cup \BB^{\phi'_n}\cup \BB^{\psi'_n}$  and containing $a$.
Define $\FF_n:=\widetilde{\FF_n}\setminus \UU_n$.

Consider $a'\in\cap_n\intr{\FF_n}$. If $a'\in \WW^B_\HH$, then the same
argument as in the indifferent case shows that $a'=a$.
Let us assume $a'\not\in\WW^B_\HH$.

By construction, all except possibly $\EE_\infty$ components of $\VV_2\setminus \MM_2$ intersect at most finitely many $\FF_n$.
A given point in $\EE_\infty$ is also in at most finitely many $\FF_n$ because $\EE_\infty=\cup_n \UU_n$.
Therefore, $a'\in\partial\MM_2$.

If $a'$ is not in any copy of the Mandelbrot set, then by
\cite[Theorem 9.1]{AY} there is a sequence $\YY_n^B$ of parapuzzle
pieces containing $a'$ such that $diam(\YY_n^B)$ tends to $0$.
Hence for big $n$ the piece $\YY_n^B$ does not intersect $\ovl{\HH^B}$.
Therefore, $\YY_n^B$ is separated from $\HH^B$ by rational rays. This implies that for big $m$ the sets $\YY_n^B$ and $\FF_m$ do not
intersect; this contradicts the definition of $a'$.

Assume that $a'$ is in a small copy $\MM_{\HH'}^B$ of the Mandelbrot
set. Then either $a'\in \MM_{\HH'}^B$ or $a'\not\in \MM_{\HH'}^B$. If
$a'\in\MM_{\HH'}^B$, then $a'=m(c')$ and $a=m(c)$ as in
Definition \ref{defn:Map_cTo_ac} and there are landing together rational rays
$\RR^{\phi_1},$ $\RR^{\phi_2}$ separating $c'$ and $c$. Therefore,
$\BB^{\phi_1},$ $\BB^{\phi_2}$ land together and separate $a'$ and
$a$.

If $a'\not\in\MM_{\HH'}^B$, then the rays $\BB^{\phi'_n}$,
$\BB^{\psi'_n}$ must separate $\MM_{\HH'}^B$ and $\HH^B$ for
sufficiently big $n$ as this is true in $\VV_1$ for $\RR^{\phi'_n}$,
$\RR^{\psi'_n}$, $\MM_{\HH'}$, and $\HH$.

\textbf{The special case.}
Assume $a$ is the root of the main hyperbolic component or $a=0$.
Choose sequences $\{\BB^{\phi_n}\}$, $\{\BB^{\psi_n}\}$, $\{\gamma_n\}$ as in the indifferent case.
 Let $\UU_n$ be the open disc
in $\EE_\infty$ bounded by the internal equipotential of height $1/2^{n}$ (see the primitive parabolic case).
 Define
 $\widetilde{\FF_n}$ to be the closed topological disc bounded by $\BB^{\phi_n}\cup \BB^{\psi_n}\cup \gamma_n$ and containing $a$.
Define $\FF_n:=\widetilde{\FF_n}\setminus \UU_n$. The same arguments as above shows that $\cap_n \FF_n=\{a\}$.

\end{proof}

\subsection{Yoccoz's results in $\VV_2$}
\label{subsec:CombV2YoccRes}

\begin{theorem}[Yoccoz's results in $\VV_2$]
\label{YoccResInV2} Assume $a'$ is not infinitely renormalizable. Then a
sequence $a_n$ tends to $a'$ combinatorially if and
only if $a_n$ tends to $a'$ in the usual topology.
\end{theorem}
\begin{proof}
The proof is the combination of \cite[Theorem 9.1]{AY}, Propositions \ref{prop:TrivOfFibersAtParabMis}, \ref{prop:DecorTilV2},
and Lemma \ref{lem:TrivOfFibersAtMis}.

The statement is obvious for points in hyperbolic components.
The case when $a$ is a Misiurewicz parameter follows from Lemma \ref{lem:TrivOfFibersAtMis}. The case when $a$
is on the boundary of a hyperbolic component in $\MM_2$ follows from Propositions \ref{prop:TrivOfFibersAtParabMis}.
\cite[Theorem 9.1]{AY} covers the case when $a$ is not
renormalizable and is not in the closure of a hyperbolic component of
$\MM_2$.

Let us verify the last case: $a$ is in a small copy $\MM_\HH^B$ of
the Mandelbrot set and $a$ is neither a Misiurewicz parameter nor on
the boundary of a hyperbolic component of $\MM_2$.

Denote by $\chi_\HH:\MM_\HH^B\to\MM_\HH$ the straightening map. We will
consider two cases.

\textbf{Case 1.} If $a_n\in \MM_\HH^B$, then $\chi(a_n)$ tends to
$\chi(a')$ combinatorially. By Yoccoz's results for $\MM$ we see that
$\chi(a_n)$ tends to $\chi(a')$ in the usual topology. As $\chi$ is a homeomorphism,
$a_n$ tends to $a'$ in the usual topology.

\textbf{Case 2.} If $a_n\not\in \MM_\HH^B$, then we may assume (by
removing finitely many points in $a_n$) that all $a_n$ are in
$\WW^B_{\HH}$ (resp.\ in $\dot \WW_\HH^B$ if $\WW_\HH^B$ is bounded by a finite ray). Consider the decoration tiling $\TT$ for $\MM_\HH^B$;
denote by $\TT_n$ the piece in $\TT$ containing $a_n$; choose a point $b_n$
in $\TT_n\cap \MM_\HH^B$.

As $a'$ is not a Misiurewicz parameter, a piece in $\TT$ may occur only finitely many times
in $\{\TT_n\}$. By Proposition \ref{prop:DecorTilV2} the diameter of
pieces in $\{\TT_n\}$ tends to $0$. Thus the distance between $a_n$
and $b_n$ tends to $0$. As a consequence, $b_n$ tends to $a'$ combinatorially. By Case 1 the sequence $b_n$ tends to $a'$ in the usual
topology. As the distance between $a_n$ and $b_n$ tends to $0$ the
sequence $a_n$ tends to $a'$ in the usual topology.
\end{proof}

\begin{cor}
\label{cor:DiamCompInV2}  The set $\MM_2$ is locally connected.
\end{cor}
\begin{proof}
Recall that the boundary of a component of $\VV_2\setminus \MM_2$ is
locally connected (Theorem \ref{thm:ParDynIsomV2}). We need
to show that the diameter of components in $\VV_2\setminus \MM_2$ tends
to $0$.

By Theorem \ref{YoccResInV2} big components (with diameters at least some
fixed $\varepsilon>0$) of $\VV_2\setminus\MM_2$ can not accumulate
at non infinitely renormalizable parameters. But by Proposition
\ref{prop:DecorTilV2} and the observation that a component of
$\VV_2\setminus \MM_2$ can intersect only three pieces in a
decoration tiling (see Figure \ref{figure:bubble}) it follows that
big components of $\VV_2\setminus\MM_2$ can not accumulate at a
small copy $\MM_\HH^B$ of the Mandelbrot set.
\end{proof}

\subsection{Proof of Theorem \ref{th:main2}}
\label{subsect:ThmMain2} Recall that by $\LL_B$ we denote a
lamination of the Basilica embedded into $\VV_1$, by
$\pi:\VV_1\to\VV'_2$ the quotient map, by $\MM'_2=\pi(\MM)$ the
image of $\MM$ in $\VV'_2$, and
$\cMM=\MM\cup\LL_B\cup\WW_B=\pi^{-1}(\MM'_2)$.

By \textit{small copies }of $\MM$ in $\MM'_2$ we mean the images of small copies of $\MM\setminus\WW_B$ under $\pi$. It follows that
small copies of $\MM$ in $\MM'_2$ are canonically homeomorphic to $\MM$.

\subsection*{Step 1}Let us construct a bijection $\omega$ between $\MM'_2$ and
$\MM_2$.

If $c\in \ovl\HH\subset (\MM\setminus \WW_B)$ or $c\in \MM\setminus \WW_B$ is a Misiurewicz parameter,
then define $\omega(\pi(c)):= m(c)$ as in Definition \ref{defn:Map_cTo_ac}.

If $c$ is in a small copy $\MM_\HH$, then we define
$\omega(\pi(c)):=\chi_\HH(c)$, where $\chi_{\HH}:\MM_\HH\to
\MM_\HH^B$ is the canonical homeomorphism.

 If
$c\in \VV_1\setminus \WW_B$ is not in the closure of a hyperbolic component, is not
a Misiurewicz parameter, and is not an infinitely renormalizable
parameter, then define
$\omega(\pi(c))$ to be the unique point in $\cap_\YY\YY^B$, where the intersection is taken over all puzzle
pieces containing $c$.

All above cases are compatible; this defines the bijection
$\omega:\MM'_2\to \MM_2$. Moreover, $\omega$ preserves the combinatorics:
\begin{equation}
\label{eq:thm2CombIsCompart}
\omega(\pi(\YY_t\cap\MM))=\YY_t^B\cap\MM_2
\end{equation}
for a parapuzzle piece $\YY_t$ in $\VV_1\setminus \WW_B$.

\subsection*{Step 2} The bijection $\omega:\MM'_2\to\MM_2$ is a homeomorphism.

We need to prove that if $c_n\in\MM'_2$ tends to $c_\infty\in\MM'_2$, then
$\omega(c_n)$ tends to $\omega(c_\infty)$.

\subsection*{Case 1} Assume $c_\infty$ belongs to at most
finitely many small copies of the Mandelbrot set; then the same is
true for $\omega(c_\infty)$. By \eqref{eq:thm2CombIsCompart} the
sequence $\omega(c_n)$ tends to $\omega(c_\infty)$ combinatorially. By Theorem \ref{YoccResInV2} the sequence
$\omega(c_n)$ tends to $\omega(c_\infty)$ in the usual topology.

\subsection*{Case 2} Assume $c_\infty$ and all $c_n$ belong to a single small copy $\pi(\MM_\HH)$.
Then the statement follows from the definition of $\omega$ because
$\omega$ coincides  with the canonical homeomorphism from $\pi(\MM_\HH)$
to $\omega(\pi(\MM_\HH))=\MM_\HH^B$ and $\pi$ restricted to $\MM_\HH$ is a homeomorphism .
\subsection*{Case 3} Assume $c_\infty$ belongs to infinitely many copies
of the Mandelbrot set (i.e., $c_\infty$ is infinitely
renormalizable), $\MM_{\HH}\subset\MM\setminus\WW_B$ is a small copy
containing $\pi^{-1}(c_\infty)$, and assume that all $\pi^{-1}(c_n)$
do not intersect $\MM_\HH$. Recall that $\pi^{-1}(c_n)$ is either a
single point or a pair of $\basil$ Misiurewicz points connected by a
leaf of $\LL_B$.

Let $\TT_n$ be a component of $\MM \backslash \MM_\HH$ that
intersects $\pi^{-1}(c_n)$.
 Let $a_n$ be the
intersection of $\MM_\HH$ and $\TT_n$.
 As $\pi^{-1}(c_\infty)$ belongs to infinitely many copies of $\MM$, it follows that
 $\pi^{-1}(c_\infty)\neq a_n$ for all $n$. Therefore, the diameter of $\pi^{-1}(c_n)$ tends to $0$
 (Corollary \ref{cor:main}) and only finitely many $\pi^{-1}(c_n)$ intersect $\TT_k$ for each fixed $k$.
 Hence by the decoration theorem the distance between $\pi^{-1}(c_n)$ and $a_n$ tends to
 $0$, and so
 the sequence $a_n$ tends to $\pi^{-1}(c_\infty)$.
  By Case $2$ we obtain that $\omega(\pi(a_k))$ tends to
 $\omega(c_\infty)$.

  By removing finitely many elements in the sequence $c_n$
we may assume that $\TT_n$ does not contain the main hyperbolic
component of $\MM$ for all $n$. Then $a_n$ is a single Misiurewicz
point. It follows from \eqref{eq:thm2CombIsCompart} that a
decoration of $\pi(\TT_n)$ is in a unique piece, say $\TT^B_n$, of
the decoration tiling $\TT$ of $\MM^B_\HH$. As the diameter of
pieces in $\TT$ tends to $0$ the distance between $\omega(a_n)$ and
$\omega(c_n)$
 tends to $0$. We conclude that $\omega(c_n)$ tends to $\omega(c_\infty)$.

\subsection*{Step 3} We claim that the homeomorphism $\omega:\MM'_2\to\MM_2$ preserves the orientation
of the complement. Let us first show how to identify the
components of $\VV_1\setminus \cMM$ with the components of
$\VV_2\setminus \MM_2$. Let $\UU$ be a component of $\VV_1\setminus
\cMM$. Then $U:=(1/z\circ\BB)(\UU)$ is a component of
$\opdisk\setminus L_B$. Under the quotient map
$\disk\to\disk/L_B\approx K_B$ the component $U$ is mapped to the
component of $\intr(K_B\setminus\{1/2\text{-limb}\})$. By Theorem
\ref{thm:ParDynIsomV2} the components of
$\intr(K_B\setminus\{1/2\text{-limb}\})$ are identified with the
components of $\VV_2\setminus \MM_2$; denote by $\UU^B$ the
component in $\VV_2\setminus \MM_2$ corresponding to $\UU$ under the
above identifications.

We claim that if $x\in\partial\pi(\UU)$, then $\omega(x)\in \partial
\UU^B$ and, moreover, $\omega$ preserves the orientation of $\UU$.
The above claims are obvious (the last part of Theorem
\ref{thm:ParDynIsomV2}) for the set
\[Mis(U)=\{\pi(y):\ y\text{ is a $\basil$ Misiurewicz parameter on }\partial\UU\};\]
as $Mis(U)$ is dense in $\partial \pi(\UU)$ the statement holds for all
points on $\partial \pi(\UU)$.

\subsection*{Step 4} By Proposition \ref{prop:ExtensOfhomeom} the homeomorphism $\omega$ extends to a homeomorphism from $\VV'_2$ onto $\VV_2$.

\section{Matings with laminations in the dynamical planes}
\label{sec:Thm3}

Consider a quadratic polynomial $z^2+c$, where
$c\not\in \WW_B$. Let $L_B$ be a $z^2$-invariant lamination of the
Basilica. The composition $B_{c}^{-1}\circ 1/z$ embeds $L_B$ into
the dynamical plane of $f_{c}=z^2+c$; denote by

\[L^{c}_B=B_{c}^{-1}\circ (1/z)(L_B)=\bigcup_{l\in L_B} \ovl{(1/z \circ B_c)^{-1}(l|_{\mathbb{D}})}
\] the embedding.

It follows from Definition \ref{def:BasRepr} that $L^c_B$ has the
following description:
\begin{equation}
\label{eq:DescrOfLeavesInL^c}
L^{c}_B=\langle 1/3,2/3\rangle\cup_{n\ge0}f_{c}^{-n}(\langle 5/6,1/6\rangle).
\end{equation}
For instance, if $l$ is a leaf of $L^c_B$ and $l\not=\langle 1/3,2/3\rangle$,
then pre-images of $l$ under $f^n_{c}$ are also leaves of $L^c_B$.

\subsection{The set $\widehat{K_{c}}=K_{c}\cup L^c_B$}
Define
\[\widehat{K_{c}}=K_{c}\cup L^c_B.\]

\begin{proposition}
\label{prop:DynamLeaves}
If $c\in\MM_\HH$ and $\MM_\HH$ is bounded by $\LL_\HH$ from both sides, then
the diameter of components
in $\C\setminus \widehat{K_{c}}$ tends to $0$ and the diameter of
leaves in $L_B^{c}$ tends to $0$.
\end{proposition}
\begin{proof}
In the proof we will use the hyperbolic contraction. Let $K_1,K_2,
\dots,$ $K_p$ be the periodic cycle of small filled in Julia sets
associated with $\MM_\HH$. It follows from the boundedness condition
that no point in a component of $\C\setminus (\widehat{K_{c}})$  is
in $\cup_i K_i$.

By construction, the composition $B_{c}^{-1}\circ (1/z)$ identifies
components of $\opdisk\setminus L_B$ and $\widehat{\C}\setminus \widehat{K_c}$. By Definition \ref{def:BasRepr}:
\[\opdisk\setminus L_B=U_{0}\cup U_{-1}\cup_{n\ge 0} (z^{-2^n})(U_1),\]
where $U_0$, $U_{-1}$, and $U_1$ are the components of
$\opdisk\setminus L_B$ that are mapped under the quotient map
$\disk/L_B \approx K_B$ onto the components of $\intr(K_B)$
containing $0,$ $-1$, and $1$ respectively.

Define $V_0$, $V_{-1}$, and $V_1$ to be the images of $U_0$, $U_{-1}$, and $U_1$ under the composition $B_{c}^{-1}\circ (1/z)$.
We get:
\[\widehat{\C}\setminus \widehat{K_{c}}=V_{0}\cup V_{-1}\cup_{n\ge 0} f^{-n}_c(V_1).\]

As the critical orbit is in $\cup_i K_i$ the map
\[f_{c}:\C\setminus f^{-1}_{c}\left(\bigcup_i K_i\right)\longrightarrow \C\setminus \bigcup_i K_i\]
is a covering of hyperbolic surfaces, and
\[\C\setminus f^{-1}_{c}\left(\bigcup_i K_i\right)\hookrightarrow \C\setminus\bigcup_i K_i
\]
is a contracting inclusion.

As $V_1$ is disjoint from $\cup_i K_i\cup\{\infty\}$ we see that $V_1$ has a
finite hyperbolic diameter in $S=\C\setminus \cup_i K_i$.
Then the hyperbolic diameter of components in $\C\setminus
\widehat{K_{c}}$ is uniformly bounded; moreover, the hyperbolic
diameter of components in $\C\setminus \widehat{K_{c}}$ that are in
a fixed compact subset of $S$ tends to $0$.  Therefore, the Euclidean
diameter of components in $\C\setminus \widehat{K_{c}}$ tends to
$0$.

Similarly, the Euclidean diameter of leaves in $L_B^c$ tends to $0$.
\end{proof}

 We say that two non-equal points are equivalent under $L_B^{c}$ if these points are
in a leaf of $L_B^{c}$.

\begin{cor}
\label{cor:DynEquivRelatLBIsClosed}
The equivalence relation $L_B^{c}$ is closed, canonical on $K_{c}$, and satisfies the condition of
Moore's theorem.
\end{cor}
\begin{proof}
Follows from Proposition \ref{prop:DynamLeaves} and the observation
that leaves in $L^{c}_B$ land at different points ($\basil$ rays
land alone as $c\not\in\WW_B$).
\end{proof}

By Moore's theorem $\widehat{\C}/L_B^{c}$ is topologically a sphere
$\sph$. Denote by $\pi_{c}$ the quotient map $\widehat{\C}\to
\widehat{\C}/L_B^{c}\thickapprox\sph$ and by $K'^B_{m(c)}=\pi_c(K_c)$ the image
of $K_c$ in $\sph$. We have
$\pi_c^{-1}(K'^B_{m(c)})=\widehat{K_{c}}$.

\subsection{Combinatorial convergence in dynamical planes}
\label{subs:CombTopologyInDynamPlanes}
Recall that by $m(c)$ we denote the image of $c$ in $\VV_2=\MM\mcup L_B$.

By a \textit{combinatorial disc} in the dynamical plane of $g_a:\widehat{\C}\to \widehat{\C}$ we mean a closed
topological disc $D$ such that $\partial D$ is in a finite union
of rational rays in bubbles and simple curves in
the closures of attracting Fatou components.

By a \textit{combinatorial disc} in the dynamical plane of $f_c:\C\to \C$ we mean a closed
topological disc $D$ such that $\partial D$ is in a finite union
of rational rays, simple curves that are disjoint from $K_c$, and simple curves in
the closures of bounded attracting Fatou components (i.e. Fatou components of $K_c$).

A sequence $x_n$ in the dynamical plane of $g_a$ or of $f_c$ \textit{tends to $x_\infty$ combinatorially} if for
every combinatorial disc $D$ such that $x_\infty\in \intr(D)$ all but finitely many $x_n$ are in $D$.
A point $x$ in the dynamical plane of $f_c$ or of $g_{m(c)}$
is \textit{renormalizable} if it belongs to a small connected filled in Julia set;
$x$ is \textit{infinitely renormalizable} if it belongs to infinitely many small connected filled in Julia sets.

It follows from Theorem \ref{thm:ParBubblRay} that a puzzle piece $Y$ exists for
$f_c$ if and only if $Y^B$ exists for $g_{m(c)}$.
Consider a copy $\MM_\HH\ni c$ bounded by $\LL_B$ from both sides. Then $m(c)\in \MM_\HH^B$. Denote by $\{K_i\}_{i\in I}$ the grand orbit
of small filled in Julia sets of $f_c$ associated with $\MM_\HH$; the map $f:I\to I$ is defined so that
$f(K_i)=K_{f(i)}$. There is a unique periodic cycle of $f:I\to I$; all other filled in Julia sets are preimages of that in the
cycle. For $g_{m(c)}$ we have the associated ($\VV_2$-twin) grand orbit $\{K_i^B\}_{i\in I}$ of small filled in Julia sets such that
$K_i^B$ is in $Y^B$ if and only if $K_i$ is in $Y$ for any puzzle piece $Y$.

We define \textit{the filled in Julia set} $K_{m(c)}$ of $g_{m(c)}$ to be the complement of the attracting basin
of the $\infty$--$0$ cycle. The following proposition is the dynamical version of Theorems \ref{YoccResInV2} and \ref{th:decor}.

\begin{proposition}
\label{prop:DynamYoccResAndDecorThm} Let $\MM_\HH\subset\MM\setminus
\WW_B$ be a copy of the Mandelbrot set bounded by $\LL_B$ from both
sides and assume that $c\in\MM_{\HH''}\subsetneq\MM_{\HH'}\subsetneq
\MM_\HH$. Denote by $K_i$, $K^B_i$, $K'_j$, $K'^B_j$ the grand
orbits of small filled in Julia sets associated with  $\MM_\HH$,
$\MM^B_\HH$,  $\MM_{\HH'}$, $\MM^B_{\HH'}$ and parametrized by $i\in
I$ and $j\in J$. Then the following holds in the dynamical planes of
$f_c$ and $g_{m(c)}$.
\begin{itemize}
\item The filled in Julia set $K_{m(c)}$ is locally connected.
\item If $x_\infty\not\in \cup_{j\in J} K'_j$, then a sequence $x_n$ tends to $x_\infty$ combinatorially if and only if
$x_n$ tends to $x_\infty$ in the usual topology. If, in addition,
$x_\infty\in J_{c}$, then the intersection of puzzle pieces
containing $x_\infty$ in the interior is $\{x_\infty\}$.
\item If $x_\infty\not\in \cup_{j\in J} K'^B_j$, then a sequence $x_n$ tends to $x_\infty$ combinatorially if and only if
$x_n$ tends to $x_\infty$ in the usual topology. If, in addition, $x_\infty\in J^B_{m(c)}$ and $x_\infty$ is not a $\basil$ point,
then the intersection of puzzle pieces containing $x_\infty$ in the interior is $\{x_\infty\}$.
\item Let $f^p_{c}:Y_\HH^p\to W_\HH^0$ be the first renormalization map. Then the diameter of preimages of $\ovl{W_\HH^0\setminus Y_\HH^p}$
under $f_{c}^p:Y_\HH^p\to W_\HH^0$ tends to $0$.
\item Let $g^p_{m(c)}:Y^B_\HH\to W^B_\HH$ be the first renormalization map. Then the diameter of preimages of $\ovl{W^B_\HH\setminus Y^B_\HH}$
under $g_{m(c)}^p:Y_\HH^B\to W_\HH^B$ tends to $0$.
\end{itemize}
\end{proposition}
\begin{proof}
We will prove the first, the third, and the fifth claims; the second and the fourth claims are proved in the same way
as the third and the fifth claims.

Let \[K''^{B}_1,K''^{B}_2,\dots, K''^{B}_q\] be the periodic cycle
of small filled-in Julia sets associated with $\MM_{\HH''}^B$. Then
\[g_{m(c)}:\widehat{\C}\setminus g_{m(c)}^{-1}\left(\{0\}\cup\{\infty\}\bigcup_{i=1}^q K''^B_i\right)\longrightarrow \widehat{\C}\setminus \left(\{0\}\cup\{\infty\}\bigcup_{i=1}^q K''^B_i\right)\]
is a covering of hyperbolic surfaces, and
\[\widehat{\C}\setminus g_{m(c)}^{-1}\left(\{0\}\cup\{\infty\}\bigcup_{i=1}^q K''^B_i\right)\hookrightarrow  \widehat{\C}\setminus \left(\{0\}\cup\{\infty\}\bigcup_{i=1}^q K''^B_i\right)\]
is a contracting inclusion.

Recall that the boundaries of components in $\widehat{\C}\setminus
K^B_{m(c)}$ are locally connected (Proposition
\ref{prop:ExtCompLocalConn}). By the hyperbolic contraction the
diameter of components in $\widehat{\C} \setminus K^B_{m(c)}$ tends
to $0$. Therefore, $K^B_{m(c)}$ is locally connected. Similarly, the
fifth claim is verified.

Let us prove the third claim. The statement is obvious if $x_\infty\not\in J^B_{m(c)}$. Assume $x_\infty\in J^B_{m(c)}$ and
$x_\infty$ is not a $\basil$ point.
If $x_\infty\not\in \cup_{j\in J} K'^B_i$, then there are infinitely many $m_1,m_2,\dots$ such that
$g_{m(c)}^{m_i} (x)$ is in the interior of $\ovl{W^B_{\HH'}\setminus Y^B_{\HH'}}$. Define $Y_i^B\ni x_\infty$ to be the pullback of $\ovl{W_\HH\setminus Y_\HH}$
under $g_{m(c)}^{m_i}$ along the orbit of $x_\infty$. By the hyperbolic contraction the diameter of $Y_i^B$ tends to $0$.

If $x_\infty$ is a $\basil$ point, then there are two sequences $Y^B_i$, $Y'^B_i$ of puzzle pieces constructed as above and two bubbles $E$, $E'$ such that
the diameters of $Y^B_i$, $Y'^B_i$ tends to $0$ and $x_\infty$ is in the interior of $E\cup E'\cup Y^B_i\cup Y'^B_i$ for all $i$.
\end{proof}

As a corollary of Proposition \ref{prop:DynamYoccResAndDecorThm} and \ref{prop:DynamLeaves} we have:

\begin{cor}
\label{cor:DynCounterpartToThm1}
The boundary of a component of
$\widehat{\C}\setminus\widehat{K_{c}}$ or of $\sph\setminus
K'^B_{m(c)}$ is locally connected. The sets $\widehat{K_{c}}$,
$K'^B_{m(c)}$ are locally connected.

All spaces $\sph=\sph(L_B^c)$ are homeomorphic by homeomorphisms preserving $K'^B_{m(c)}$
\end{cor}
\begin{proof}
The proof is the same as the proof of Proposition \ref{prop:main}.
\end{proof}

By definition, the restriction of the equivalence relation $L_B^c$ to $K_c$ is invariant under $f_c$.
We define $g'_{m(c)}:K'^B_{m(c)} \to K'^B_{m(c)}$ to be the quotient dynamics
\[f_{c}/L_B^c:K_{c}/L_B^c \to K_{c}/L_B^c.\]

We extend $g'_{m(c)}$ to a branched cover over $\sph$ such that
$g'_{m(c)}$ has one superattractive periodic cycle with period $2$
that attracts all points in $\sph\setminus  K'^B_{m(c)}$, and
locally at the superattractive periodic cycle $g'_{m(c)}$ is
conjugate to $z\to z^2$. We will leave it as an exercise to show
that this extension exists (compare to Proposition
\ref{prop:ExtensOfhomeom}).

\subsection{Proof of Theorem \ref{th:main3}}
Consider a parameter $c\in\MM\setminus \WW_B$ such that $c$ is not on the boundary of a hyperbolic
component or $c$ satisfies the assumption of Proposition \ref{prop:DynamYoccResAndDecorThm}.

The postcritically finite case of Theorem \ref{th:main3} is well known, see \cite{Re1}, \cite{Tan}. The non-renormalizable non-hyperbolic
case of Theorem \ref{th:main3} is in \cite{AY}. If $c$ is finitely many times renormalizable and is not hyperbolic, then
by Theorem \ref{YoccResInV2} triviality of the fiber holds at $m(c)$; the same arguments as in \cite{AY} proves Theorem \ref{th:main3} in the non-infinitely
renormalizable non-hyperbolic case.

\begin{proposition}
\label{prop:ConjBetwJulSets}
Under the assumption of Proposition \ref{prop:DynamYoccResAndDecorThm}
the maps $g_{m(c)}:K^B_{m(c)}\to K^B_{m(c)}$ and $g'_{m(c)}:K'^B_{m(c)} \to K'^B_{m(c)}$ are topologically conjugate.

The conjugacy is unique if it preserves the combinatorics of puzzle pieces
and coincides with the straightening maps on small filled in Julia sets
\end{proposition}
\begin{proof}
The proof will follow the proof of Theorem \ref{th:main2}.

We will use the same notations as in Proposition \ref{prop:DynamYoccResAndDecorThm}.
Define a bijection $\omega: K'^B_{m(c)} \to K^B_{m(c)}$ as follows.

For $\pi(K_i)$, where $i\in I$, define $\omega: \pi(K_i)\to K_i^B$ to be the dynamical straightening map.

If $x\in K_c$ is the landing point a $\basil$ ray $R^\phi$, then denote by $x'$ the landing point of $R^\phi$
in the dynamical plane of $f_B$ and define $\omega (\pi(x))$ to be $h_a^{-1}(\rho^{-1}(x'))$.

If $x\not\in \cup_{i\in I} K_i$ and $x$ is not the landing point of a $\basil$ ray $R^\phi$, then define
$\omega(\pi(x))$ to be a unique point in $\cap_{Y\ni x} Y^B$, where the intersection is taken over all puzzle pieces containing $x$.

All above cases are compatible; this defines the bijection
$\omega: K'_{m(c)}\to K_{m(c)}$. Furthermore, $\omega$ preserves the combinatorics:
\begin{equation}
\label{eq:thm3CombIsCompart}
\omega(\pi( Y_t\cap K_c))= Y_t^B\cap K_{m(c)}
\end{equation}
for a puzzle piece $Y_t$ in the dynamical plane of $f_c$.

It is also clear that $\omega$ is a  conjugacy between $g_{m(c)}:K^B_{m(c)}\to K^B_{m(c)}$ and $g'_{m(c)}:K'^B_{m(c)} \to K'^B_{m(c)}$.
Indeed, by \eqref{eq:thm3CombIsCompart} the map $\omega$ conjugates the dynamics of puzzle pieces, thus the non renormalizable case follows
from Proposition \ref{prop:DynamYoccResAndDecorThm}. On $\cup_{i\in I} K_i$ the map $\omega$ coincides with the dynamical straightening map which is,
in particular, a topological conjugacy.

We need to show that $\omega$ is continuous: if $x_n\in K'_{m(c)}$ tends to $x_\infty\in K'_{m(c)}$, then
 $\omega(x_n)$ tends to $\omega(x_\infty)$.
Consider three cases.

\textbf{Case 1.} Assume $x_\infty\not\in \cup_{j\in J}\pi(K'_j)$, then $\omega(x_\infty)\not\in \cup_{j\in J}\pi(K'^B_j)$.
By \eqref{eq:thm3CombIsCompart} the
sequence $\omega(x_n)$ tends to $\omega(x_\infty)$
combinatorially. By Proposition \ref{prop:DynamYoccResAndDecorThm} the sequence
$\omega(c_n)$ tends to $\omega(c_\infty)$ in the usual topology.

\textbf{Case 2.} Assume $x_\infty$ and all $x_n$ belong to $\pi(K_i)$ for $i\in I$.
Then $\omega(x_n)\in K^B_i$ and the statement follows from the definition of $\omega$ because
$\omega\circ \pi$ coincides  with the dynamical straightening map from $K_i$
to $K^B_i$.

\textbf{Case 3.} Assume $x_\infty\in \pi(K'_j)\subset \pi(K_i)$ for
$i\in I$ and $j\in J$; assume also that all $x_n$ are not in
$\pi(K_i)$. Let $f_{c}^p:Y^p_\HH\to W^0_\HH$ be the first
renormalization map and let $K_1$ the periodic filled in Julia set
in $Y^p_\HH$. Define $T$ to be the set of preimages of
$\ovl{W^0_\HH\setminus Y^p_\HH}$ under $f_{c}^p:Y^p_\HH\to W_\HH^0$;
i.e. $T$ is the dynamic decoration tiling. By Proposition
\ref{prop:DynamYoccResAndDecorThm} the diameter of pieces in $T$
tends to $0$.

 Recall that $K_i$ is a preimage of $K_1$; let $n$ be the minimal number such that $f_{c}^n(K_i)=K_1$. Define $W(i)$, resp.\ $T(i)$,
to be the preimage of $W_\HH^0$, resp.\ of pieces in $T$, under
$f_{c}^n$ along the orbit of $K_i$. By shifting the sequence $x_n$
we may assume that all $x_n$ are in $\pi(W(i))$.

Recall that $\pi^{-1}(x_n)$ is either a single point or a pair of
$\basil$ points; denote by $T_n$ a piece in $T(i)$ intersecting $\pi^{-1}(x_n)$
 and by $a_n$ the point in the intersection
of $T_n$ and $K_i$. As the diameter of pieces in $T(i)$ tends to $0$ (because this is so for pieces in $T$) the distance
between $\pi^{-1}(x_n)$ and $a_n$ tends to $0$. By Proposition \ref{prop:DynamLeaves} the diameter
of $\pi^{-1}(x_n)$ tends to $0$. We conclude that $a_n$ tends to $\pi^{-1}(x_\infty)$.
By Case 2 the sequence $\omega(\pi(a_n))$ tends to $\omega(x_\infty)$.

Similar to $W(i)$ and $T(i)$ we define $W^B(i)$ and $T^B(i)$ in the dynamical plane of $g_{m(c)}$. It follows from Proposition \ref{prop:DynamYoccResAndDecorThm}
that the diameter of pieces in $T^B(i)$ tends to
$0$. Therefore, the distance between $\omega(\pi(a_n))$ and $\omega(x_n)$ tends to $0$. This implies that $\omega(x_n)$ tends to $\omega(x_\infty).$
\end{proof}

Recall that $g'_{m(c)}:\sph\to \sph$ is defined so that it has a
unique superattractive two cycle, call it $a$-$b$-cycle, the
dynamics of $g'^2_{m(c)}$ is locally conjugate to $z\to z^2$ at $a$
and at $b$, and all points in $\sph\setminus K'_{m(c)}$ are
attracted by the $a$-$b$-cycle. Therefore, we can extend $\omega$ to
$K'_{m(c)} \cup V_a\cup V_b$, where $V_a$, $V_b$ are small
neighborhoods of $a$, $b$ respectively. As
\[g'_{m(c)}:(\sph\setminus (K'_{m(c)}\cup
g'^{-1}_{m(c)}(\{a\}\cup\{b\})))\to (\sph\setminus
(K'_{m(c)}\cup\{a\}\cup\{b\}))\]
 is a covering we can extend
$\omega$ to $\sph\setminus K'_{m(c)}$ such that \[\omega \circ
g'_{m(c)}=g'_{m(c)}\circ\omega .\] This gives a conjugacy $\omega$
between $g'_{m(c)}$ and $g_{m(c)}$ such that the restriction of
$\omega$ to $K'_{m(c)}$ or to $\sph\setminus K'_{m(c)}$ is a
homeomorphism. We will leave it as an exercise to show that $\omega$
is a homeomorphism on $\sph$ (follows from local connectivity of
$\sph\setminus K'_{m(c)}$).

Let us now argue that if $K_c$ is locally connected, then $f_c\mcup L_B$ is the same as $f_c\mcup f_B$. Define $\Omega_{m(c)}=\sph\setminus K'^B_{m(c)}$.
By construction, $\Omega_{m(c)}$ is homeomorphic, say by $h$, to $\intr(K_B)$. We need to show that
$h$ extends to a homeomorphism between $\ovl{\Omega_{m(c)}}$ and $K_B$.
Let $E_1,E_2,E_3,\dots$ be an infinite sequence
of components in $\intr(K_B)$ such that the intersection $\ovl{E_i}\cap\ovl{E_{i+1}}$ consists of one point; in particular, $E_i\not=E_{j}$ if
$i\not=j$. We claim that $h$ extends to a homeomorphism from $\ovl{\cup E_i}$ onto  $\ovl{\cup h(E_i)}$. Indeed, it follows from the local
connectivity of $E_i$ that $h$ extends to $\partial E_i$ for every $i$. Therefore, it is sufficient to show that the sequence $E_i$ accumulates at
one point. If $E_i$ were accumulating at two points, say $x,y$, then the sequence $\pi^{-1}(E_i)$ of components in $\widehat{\C}\setminus \widehat{K_c}$
would accumulate at $\pi^{-1}(x)$ and $\pi^{-1}(y)$ which is the contradiction to the local connectivity of $K_c$; details are left to the reader.

%
%%%%%%%%%%%%%%%%%%%%%%%%%%%%%%%%%%%%%%%%%%%%%%%%%%%%%%%%%%%%%%%%%%%%%%%%
%%%%%%%%%%%%%%%%%%%%%%%%%%%%%%%%%%%%%%%%%%%%%%%%%%%%%%%%%%%%%%%%%%%%%%%%
%
%

%%%%%%%%%%%%%%%%%%%%%%%%%%%%%%%%%%%%%%%%%%%%%%%%%%%%%%%%%%%%%%%%%%%%%%%%%%%%%%%%%%%%%%%%%%%%%%

%%%%%%%%%%%%%%%%%%%%%%%%%%%%%%%%%%%%%%%%%%%%%%%%%%%%%%%%%%%%%%%%%%%%%%%%%%%%%%%%%%%%%%%%%%%%%%

%%%%%%%%%%%%%%%%%%%%%%%%%%%%%%%%%%%%%%%%%%%%%%%%%%%%%%%%%%%%%%%%%%%%%%%%%%%%%%%%%%%%%%%%%%%%%%

%%%%%%%%%%%%%%%%%%%%%%%%%%%%%%%%%%%%%%%%%%%%%%%%%%%%%%%%%%%%%%%%%%%%%%%%%%%%%%%%%%%%%%%%%%%%%%

\end{document}